\documentclass[11pt, a4paper, reqno]{amsart}
\usepackage{amsthm, amsmath, amsfonts, amssymb, appendix, dsfont, latexsym, mathrsfs, mathtools}
\usepackage{esint}
\usepackage{graphicx}
\usepackage{tikz}

\makeatletter
\usepackage{delarray, a4, color}
\usepackage{euscript}
\usepackage[latin1]{inputenc}
\usepackage{enumitem}
\usepackage{soul}
\usepackage{hyperref}
\hypersetup{
colorlinks,
menucolor=black,
linkcolor=black,
citecolor=black,
urlcolor=blue
}
\usepackage{cleveref}
\usepackage{crossreftools}
\theoremstyle{plain}
\newtheorem{theorem}{Theorem}[section]
\newtheorem{maintheorem}{Theorem} 
\crefname{maintheorem}{Theorem}{Theorems} 
\newtheorem*{theorem*}{Theorem}
\newtheorem{lemma}[theorem]{Lemma}
\newtheorem{corollary}[theorem]{Corollary}
\newtheorem{proposition}[theorem]{Proposition}

\theoremstyle{definition}
\newtheorem{definition}[theorem]{Definition}

\theoremstyle{remark}
\newtheorem{remark}[theorem]{Remark}
\newtheorem{remark*}[theorem]{Remark\textup{*}}

\numberwithin{equation}{section}

\DeclareMathAlphabet{\mathpzc}{OT1}{pzc}{m}{it}

\def\C {\mathbb{C}}

\def\N {\mathbb{N}}

\def\R {\mathbb{R}}

\def\Z {\mathbb{Z}}

\newcommand\1{{\ensuremath {\mathds 1} }}

\newcommand{\keyword}[1]{\textbf{#1}}

\newcommand{\bA}{\mathbf{A}}
\newcommand{\bF}{\mathbf{F}}
\newcommand{\bJ}{\mathbf{J}}

\newcommand{\be}{\mathbf{e}}
\newcommand{\bn}{\mathbf{n}}

\newcommand{\bx}{\mathbf{x}}
\newcommand{\by}{\mathbf{y}}
\newcommand{\bz}{\mathbf{z}}

\newcommand{\cE}{\mathcal{E}}

\newcommand{\cL}{\mathcal{L}}

\newcommand{\cP}{\mathcal{P}}

\newcommand{\cR}{\mathcal{R}}

\newcommand{\sSL}{\textup{SL}}
\newcommand{\sSU}{\textup{SU}}

\newcommand{\sU}{\textup{U}}

\newcommand{\CLGN}{C_{\rm LGN}}
\newcommand{\CH}{C_{\rm H}}
\newcommand{\dd}{{\rm d}} 
\newcommand{\ii}{{\rm i}} 

\newcommand{\NLL}{\mathrm{NLL}}

\DeclareMathOperator{\curl}{\mathrm{curl}}

\DeclareMathOperator{\sign}{\mathrm{sign}}

\DeclareMathOperator{\supp}{\mathrm{supp}}

\newcommand{\norm}[1]{\left\|#1\right\|}
\newcommand{\slot}{\cdot}

\newcommand{\dist}{\mathrm{dist}}
\newcommand{\diam}{\mathrm{diam}}
\renewcommand{\div}{\operatorname{div}}

\newcommand{\Lip}{\mathrm{Lip}}
\newcommand{\loc}{\mathrm{loc}}

\usepackage{placeins}

\usepackage{etoolbox}

\makeatletter
\patchcmd{\@setaddresses}{\indent}{\noindent}{}{}
\patchcmd{\@setaddresses}{\indent}{\noindent}{}{}
\patchcmd{\@setaddresses}{\indent}{\noindent}{}{}
\patchcmd{\@setaddresses}{\indent}{\noindent}{}{}
\makeatother

\title[A generalized Liouville equation and magnetic stability]{A generalized Liouville equation \\ and magnetic stability}

\author[A. Ataei]{Alireza ATAEI}
\address{(A. Ataei) Department of Mathematics, Uppsala University, Box 480, SE-751 06, Uppsala, Sweden}
\email{\url{alireza.ataei@math.uu.se}}

\author[D. Lundholm]{Douglas LUNDHOLM}
\address{\noindent (D. Lundholm) Department of Mathematics, Uppsala University, Box 480, SE-751 06, Uppsala, Sweden}
\email{\url{douglas.lundholm@math.uu.se}}

\author[D.-T. Nguyen]{Dinh-Thi NGUYEN}
\address{(D.-T. Nguyen) Department of Mathematics, Uppsala University, Box 480, SE-751 06, Uppsala, Sweden;
University of Science, Ho Chi Minh City, Vietnam;
Vietnam National University, Ho Chi Minh City, Vietnam}
\email{\url{ndthi@hcmus.edu.vn}}


\allowdisplaybreaks
\begin{document}

\begin{abstract}
This work considers two related families of nonlinear and nonlocal problems in the plane $\R^2$. 
The first main result derives the general integrable solution to a generalized Liouville equation using the Wronskian of two coprime complex polynomials. 
The second main result concerns an application to a generalized Ladyzhenskaya--Gagliardo--Nirenberg interpolation inequality, with a single real parameter $\beta$ interpreted as the strength of a magnetic self-interaction. The optimal constant of the inequality and the corresponding minimizers of the quotient are studied and it is proved that for $\beta \ge 2$, for which the constant equals $2\pi\beta$, such minimizers only exist at quantized $\beta \in 2\N$ corresponding to nonlinear generalizations of Landau levels with densities solving the generalized Liouville equation. This latter problem originates from the study of self-dual vortex solitons in the abelian Chern--Simons--Higgs theory and from the average-field-Pauli effective theory of anyons, i.e.\ quantum particles with statistics intermediate to bosons and fermions.
An immediate application is given to Keller--Lieb--Thirring stability bounds for a gas of such anyons which self-interact magnetically (vector nonlocal repulsion) as well as electrostatically (scalar local/point attraction), thus generalizing the stability theory of the 2D cubic nonlinear Schr\"odinger equation.
\end{abstract}

\maketitle
\setcounter{tocdepth}{2}
\tableofcontents

\section{Introduction and main results}\label{sec:intro}

This work concerns two nonlinear and nonlocal problems in the plane $\R^2$, and our main contributions will be summarized in \Cref{thm:liouville,thm:magneticstability} below.
The first one concerns the Liouville problem in the domain of complex geometry and nonlinear PDEs, 
while the second problem concerns functional analysis of interpolation inequalities and the calculus of variations, also framed in a geometric setting, involving a supersymmetric index theorem for the curvature of a vector bundle and the topology of its vortex solitons. 
What unifies these two problems and brings about their mutual solution is an application in mathematical physics, and more precisely a stability problem in quantum mechanics, involving intermediate quantum statistics (``anyons'') and Chern--Simons self-generated magnetic fields.
Thus, our discussion concerning the latter problem is also highly relevant for the proofs of our mathematical main results.

We proceed to give the necessary definitions and state our main results, their background motivation, key ideas in our proofs, and their most important implications.

\subsection{The Liouville problem}\label{sec:intro-liouville}

Liouville's equation has played a key role in understanding both geometry and physics. The initial motivation was to study isothermal coordinates on two-dimensional manifolds (see \cite{Liouville-53} and, e.g., \cite{Berger-71,ChaYan-87,KazWar-74a,KazWar-74b}), and, later on, the problem was generalized to study higher-dimensional manifolds whose metric is conformally equivalent to a constant scalar curvature metric, which is known as the Yamabe problem; see  \cite{BreMar-11,LeePar-87,Yamabe-60}. This problem has also been considered for metrics which may vanish at finite points, known as conic singularities, and the equation then changes by adding Dirac masses at the singular points; see \cite{E-21,P-898,P-905,T-91}. The equation has emerged from physics in several instances, 
and the most relevant for us are reviewed in
\cite{Dunne-95,Khare-05,RebSol-84,Tarantello-08}
and concern
Chern--Simons--Higgs theory \cite{Hagen-91,HorYer-98,JacPi-90a,JacPi-90b}, which we return to below.
We mention also some developments on the analysis side of this equation concerning regularity, 
classification of solutions, and blow-up of solutions; 
see \cite{BriLei-03,BriHouLei-05,BreMer-91,BT-02,CheLi-91,CheLi-93,ChoWan-94,ChoWan-95,LiSha-94,PraTar-01,Wu-24}. 

We consider the following generalized version of Liouville's equation.
Let $\varrho \in L^1(\R^2;\R_+)$ be a non-negative density function that satisfies 
\begin{align}\label{eq:Liouville-rho}
	-\Delta \log(\varrho) =2 K \varrho - \sum_{j=1}^n 4\pi  n_j \delta_{z_j},
\end{align}
weakly in $\R^2,$ where $K>0, z_j \in \C, n_j \in \N$, $n \in \N \cup \{0\}$, are constants, 
and $\delta_{z_j}$ is the unit Dirac mass at $z_j$. 
Note that we used the natural identification $\R^2 \cong \C$ by $z = z(x,y) := x+{\rm i}y$, where the complex plane has the additional multiplicative structure and choice of real axis. 
Then, by setting $f(z) := \prod_{j=1}^{n_j} (z-z_j)^{n_j}$ for every $z \in \C$, $\psi :=\log\left( \frac{2K \varrho}{|f|^2}\right)$, 
and using that $\Delta\log|z| = 2\pi\delta_0(z)$,
we obtain the equivalent equation
\begin{equation}\label{eq:Liouville-psi}
	-\Delta \psi = |f|^2 e^{\psi},
\end{equation}
weakly in $\R^2$,
where $|f|^2 e^{\psi} \in L^1(\R^2)$. 
As our first result, we prove the explicit form of all the solutions to this equation, the quantization of the values of $\int_{\R^2} |f|^2 e^{\psi}$, and the $\sU(2)$-symmetry of the solutions, where $\sU(2)$ denotes the unitary group of degree $2$.

\begin{maintheorem}[Generalized Liouville equation]
\label{thm:liouville}
Let $f\colon \C \to \C$ be a nonzero polynomial.
All the weak solutions $\psi \in L^1_\loc(\R^2;\R)$ 
of the equation \eqref{eq:Liouville-psi}, 
such that $\int_{\R^2} |f|^2 e^{\psi} < \infty$, are of the form
\begin{equation}\label{eq:Liouville-solution}
	\psi = \psi_{P,Q} := \log(8) - 2\log(|P|^2 + |Q|^2),
\end{equation}
where $P,Q$ are two coprime complex polynomials which satisfy $f= P' Q -P Q'$. 
Moreover, 
\begin{equation}\label{eq:Liouville-degree}
	\max(\deg(P),\deg(Q)) = \frac{\int_{\R^2}|f|^2 e^{\psi} }{8\pi},
\end{equation}
and if $(P,Q)$ and $(\tilde{P},\tilde{Q})$
are pairs of polynomials (not necessarily coprime) then
$\psi_{P,Q} = \psi_{\tilde{P},\tilde{Q}}$ if and only if
$$(\tilde{P},\tilde{Q}) = \Lambda(P,Q),$$ for some constant $\Lambda \in \sU(2)$.
\end{maintheorem}

Our proof of this theorem will be given in \Cref{sec:liouville}.
Further, as an immediate application and preparation for the next part, we describe the explicit general form of solutions for polynomials $f$ with a single root and polynomials $f$ of degree at most two. 
This is summarized in \Cref{cor:radiallysymmetric,cor:degreetwopoly}, respectively.

We now mention the previous results for this equation. The case that $f$ is a non-zero constant was first studied in \cite{Liouville-53}. The authors in \cite{CheLi-91} also proved that $\int_{\R^2} e^{\psi} = \frac{8 \pi}{|f|^2}$. The case of polynomials $f$ with a single root was studied in \cite{ChaKie-94,PraTar-01}. In the case of multiple roots, the blow-up problem has been studied to find the behavior of the solutions as they converge weakly; see \cite{BreMer-91,LiSha-94,Wu-24}. However, to the best of our knowledge, the general form of solutions for either \eqref{eq:Liouville-rho} or \eqref{eq:Liouville-psi} is missing in the mathematical literature. On the other hand, working in the context of our specific application below concerning self-dual Chern--Simons solitons, physicists have already predicted the general form of \eqref{eq:Liouville-solution} but without a rigorous proof; see \cite{JacPi-90a,JacPi-90b,Hagen-91,Jackiw-91,HorYer-98}, and the reviews \cite[p.~33]{Dunne-95}, \cite[Ch.~7.4]{Khare-05}, \cite{HorZha-09}.
In particular, they claimed that independent of $n_j$ being integer or not there always exists a global meromorphic function $g: \C \to \C$ such that 
\begin{align}\label{eq:Liouvillerepresentation}
	\varrho = \frac{4}{K} \frac{|g'|^2}{(1+ |g|^2)^2}.
\end{align}
This is known as Liouville's representation, indeed valid as a solution to the regular Liouville equation on any simply connected domain, and is unique up to M\"obius transformations; see, e.g., \cite[Remark~3]{BriLei-03}.
However, if $n_j$ is not an integer, then there exists no representation of such form; see \cite{BriHouLei-05,PraTar-01}. In fact, in that case there exists a possibility of having a multi-valued representation, such as $g(z) = z^{\alpha}$, where $\alpha \notin \Z$. Indeed, with this choice of $g$ in \eqref{eq:Liouvillerepresentation}, we derive
\begin{align*}
	-\Delta \log(\varrho) = 2 K \varrho - 4 \pi (\alpha -1) \delta_{0} 
\end{align*}
in $\R^2,$ where $\delta_0$ is the unit Dirac mass centered at $0.$ Hence, $\varrho$ satisfies Liouville's equation outside the origin but does not have Liouville's representation for a meromorphic function. 

In our proof of \Cref{thm:liouville} we fix this issue by the following steps. First, we derive the regularity $\psi \in C^{\infty}(\R^2)$ in \eqref{eq:Liouville-psi}, which was, to our knowledge, only previously shown with the extra assumption of $\psi \in W^{1,2}_{\loc}(\R^2)$ (see \cite{KazWar-74a}), or $e^{\psi} \in L^1_{\loc}(\R^2)$ (see \cite{BreMer-91,CheLi-93}). Second, we use a local (and possibly multi-valued) representation from \cite{ChoWan-94} which extends \eqref{eq:Liouvillerepresentation} to punctured disks, together with integrability $|f|^2 e^{\psi} \in L^1(\R^2)$, to show that there exists a local Liouville's representation \eqref{eq:Liouvillerepresentation} in every disk $D$ (possibly including a single root of $f$), for a meromorphic function $g$ which depends on $D$. 
All such local representations may be patched together
by using the M\"obius symmetry of different Liouville representations, to obtain a global Liouville representation for \eqref{eq:Liouvillerepresentation}. 
We then use integrability again to conclude that the function $g$ in \eqref{eq:Liouvillerepresentation} cannot have an essential singularity at infinity. 
This immediately implies that $g$ is a rational polynomial, and by comparing with \eqref{eq:Liouville-psi}, we deduce the general form \eqref{eq:Liouville-solution}. Once it is obtained, we may straightforwardly derive the quantization of $\int_{\R^2} |f|^2 e^{\psi}$ in \eqref{eq:Liouville-degree} by the Gauss--Green theorem. Finally, the $\sU(2)$ symmetry of solutions is derived by studying the ODE for polynomials $y,R$
\begin{align*}
	f y'' - f' y' + R y =0
\end{align*}
that gives an equivalent form of the Wronskian $P' Q - P Q' = f$. This method changes a nonlinear problem of amplitude $|P|^2 +|Q|^2$ into a linear problem of an ODE, which derives the $\sU(2)$ symmetry. Up to the knowledge of the authors, this symmetry and its connection to solutions of Liouville's equation has not been studied previously in the literature. 

We remark that to find all the solutions to \eqref{eq:Liouville-psi} with given $f$, by \Cref{thm:liouville}, it is enough to find all the pairs of coprime complex polynomials $P,Q$ such that $P' Q - P Q' = f$, which is known as \keyword{the inverse Wronskian problem} (see, e.g., \cite{E-21,Purbhoo-23,Scherbak-02,Sottile-10}). We determine an algorithm which solves the inverse Wronskian problem in finite steps; see \Cref{rem:algorithmWronsk}.

\subsection{The magnetic interpolation problem}\label{sec:intro-magnetic}

We now move on to a geometric setting in which the generalized Liouville problem \eqref{eq:Liouville-rho}-\eqref{eq:Liouville-psi} emerges naturally, involving solitonic self-generated magnetic fields.
In order to facilitate our presentation,
we formulate our results from a purely mathematical perspective concerning functional inequalities, and defer our more technical discussion (including an outline of our proofs)
to the context of our specific application in mathematical physics, in \Cref{sec:intro-stability}.

Our starting reference point for the magnetic problem is a well-known functional inequality, namely the following Ladyzhenskaya--Gagliardo--Nirenberg (LGN) interpolation inequality on $\R^2$ 
\cite{Ladyzhenskaya-58,Gagliardo-59,Nirenberg-59}
for the embedding $H^1 \hookrightarrow L^4$:
\begin{equation}\label{eq:LGN}
    \int_{\R^2}|\nabla u|^2 \int_{\R^2}|u|^2 
	\ge \CLGN \int_{\R^2}|u|^4,
    \quad \forall u \in H^1 = H^1(\R^2;\C).
\end{equation}
Equality is realized with the optimal constant 
$\CLGN:= \|\tau\|_{L^2}^2/2 \approx 0.931 \times 2\pi$ \cite{Weinstein-83} and the optimizer $\tau \in H^1 \cap C^\infty(\R^2;\R^+)$ is known as ``Townes soliton'' 
\cite{CiaGarTow-64,CiaGarTow-65,Fibich-15},
to which we shall return below.

The following geometric generalization of \eqref{eq:LGN} has been considered by Dolbeault et al.\ in \cite{DolEstLapLos-18}, motivated by Keller--Lieb--Thirring inequalities on manifolds.
One may promote the dense domain $C_c^\infty \subseteq H^1$ of smooth and compactly supported functions $u\colon \R^2 \to \C$ defined on
the trivial flat bundle $\R^2 \times \C$ to smooth and compactly supported global sections of a curved complex line bundle on $\R^2$.
Thus, the gradient $\nabla = (\partial_1,\partial_2)$ is replaced by a covariant derivative
$\nabla_{\bA} = \nabla + {\rm i}\bA$, where the vector potential
$\bA=(A_1,A_2) \colon \R^2 \to \R^2$ defines a smooth local
$\sU(1)$ connection 
and an associated curvature (pseudo)scalar field
$B \colon \R^2 \to \R$,
\begin{equation}\label{eq:mag-field}
    B = \curl \bA := \partial_1 A_2 - \partial_2 A_1.
\end{equation}
The curvature field $B$ is known in physics applications as an external \keyword{magnetic field}, 
and the corresponding magnetic interpolation inequality
\begin{equation}\label{eq:LGN-magnetic}
    \int_{\R^2}|\nabla_{\bA} u|^2 \int_{\R^2}|u|^2 
	\ge C\int_{\R^2}|u|^4,
    \quad \forall u \in C^\infty_c(\R^2;\C),
\end{equation}
was studied in \cite{DolEstLapLos-18}, with some weaker regularity assumptions on $\bA$ and $u$, and for fixed $B = \curl \bA$.
For our purposes, the most important observations concerning this bound are the following:

\begin{itemize}
\item
	Due to covariance w.r.t.\ \keyword{gauge transformations}, i.e.
	$$
	    u \mapsto \tilde{u} = e^{-{\rm i}\chi}u, \quad
	    \bA \mapsto \tilde{\bA} = \bA + \nabla\chi, \quad
	    B \mapsto \tilde{B} = B,
	$$
	for $\chi \in C_c^\infty(\R^2;\R)$,
	the constant in \eqref{eq:LGN-magnetic} depends on the geometry, i.e. $C=C(B)$, 
	and not on a particular equivalent gauge choice $\bA$.
\item
	Due to the \keyword{diamagnetic inequality} (see, e.g., \cite[Theorem~7.21]{LieLos-01})
	\begin{equation}\label{eq:diamag}
	    \int_{\R^2} |\nabla_{\bA}u|^2 \ge \int_{\R^2} \left|\nabla|u|\right|^2,
	\end{equation}
	the corresponding optimal constant can only increase:
	$C(B) \ge C(0) = \CLGN$.
\item
	Due to a supersymmetric identity (elaborated in \Cref{rmk:Pauli}), 
	the magnetic gradient integral also satisfies a \keyword{Bogomolnyi bound} (cf. \cite{Bogomolny-76,HloSpe-93}):
	\begin{equation}\label{eq:bogobound}
	    \int_{\R^2} |\nabla_{\bA}u|^2 \ge \pm \int_{\R^2} B|u|^2.
	\end{equation}
\end{itemize}

We now substitute the external magnetic field \eqref{eq:mag-field} for a ``self-generated'' field, in the following precise sense (cf.\ \cite{LunRou-15}):
\begin{definition}[Self-generated magnetic field and energy functional]\label{def:self-field}
Let $\varrho \in L^1(\R^2;\R_+)$ be a given density function, $\beta \in \R$, and
$\bA[\varrho] \colon \R^2 \to \R^2$ a magnetic vector potential which generates a magnetic field proportional to $\varrho$:
\begin{equation} \label{eq:mag-pot}
	\bA[\varrho](\bx) := (\nabla^\perp \log |\slot|) * \varrho \,(\bx)
    = \int_{\R^2} \frac{(\bx-\by)^\perp}{|\bx-\by|^2} \varrho(\by) \,\dd \by
\end{equation}
(defined in the principal value sense,
and we use the notation $\bx^\perp = (x_1,x_2)^\perp = (-x_2,x_1) \in \R^2$), 
such that the curvature $B[\varrho]$ associated to $\beta\bA[\varrho]$ is
\begin{equation}\label{eq:self-field}
	B[\varrho](\bx) :=
	\curl \beta\bA[\varrho](\bx) 
	= \beta \left(\Delta \log |\slot|\right) * \varrho (\bx) 
	= 2\pi\beta\varrho(\bx).
\end{equation}
Then, substituting $\varrho = |u|^2$ in \eqref{eq:mag-pot}-\eqref{eq:self-field}, we define the \keyword{magnetic self-energy}
\begin{equation}\label{eq:AF-func}
	\cE_{\beta}[u] 
	:= \int_{\R^2} \left|\left(\nabla + {\rm i}\beta\bA\left[|u|^2\right]\right) u\right|^2,
	\quad \beta \in \R, \ u \in H^1.
\end{equation}
\end{definition}

Thus, $\beta \in \R$ is the overall strength of the magnetic field,
with its sign defining the orientation of the field in the plane.
Note that rescaling 
$|u|^2$ by a positive number is equivalent to rescaling $\beta$ by the same number, and
henceforth in the sequel we will normalize the mass of the (probability) density $\varrho = |u|^2$ to unity
and interpret $\beta$ as the (possibly fractional) number of magnetic flux units\footnote{A magnetic flux unit is $2\pi$, an elementary phase circulation.} per unit mass.
Further, the gauge in \eqref{eq:mag-pot} is chosen such that $\div\bA[\varrho]=0$ (in accordance with convenience, conventions and applications; known as the Coulomb gauge), and the domain $H^1 = H^1(\R^2;\C)$ makes the above functionals well defined (see \cite[Appendix]{LunRou-15}; we comment further on this choice below).

Analogously to \eqref{eq:LGN-magnetic}, 
we study the optimal constant $C = C(\beta)$ 
in the LGN-type magnetic interpolation inequality
\begin{align}
\label{eq:LGN-selfmagnetic}
\int_{\R^2} \left|\left(\nabla + {\rm i}\beta\left(\int_{\R^2} |u|^2\right)^{-1}\bA\left[|u|^2\right]\right) u\right|^2 \int_{\R^2} |u|^2 \geq C\int_{\R^2} |u|^4, \quad \forall u \in H^1,
\end{align}
or equivalently,
after imposing the normalization $\int_{\R^2} |u|^2 = 1$,
\begin{equation}\label{eq:LGN-selfmagnetic-norm}
	\cE_{\beta}[u] \geq C \int_{\R^2}|u|^4, \quad u \in H^1.
\end{equation}

\begin{definition}[Self-magnetic minimization problem]\label{def:gamma}
For any $\beta \in \R$, we define
\begin{align}\label{eq:defgamma}
	\gamma_*(\beta) := \inf \left\{ \frac{ \cE_{\beta}[u]}{\int_{\R^2} |u|^4}
	: u \in H^1, \int_{\R^2} |u|^2 = 1 \right\}.
\end{align}
\end{definition}
Obviously, $\gamma_*(0) = \CLGN$, and further 
$\gamma_*(-\beta)=\gamma_*(\beta)$,
by complex conjugation symmetry $u \mapsto \bar{u}$ (orientation flip). Therefore, we may restrict to $\beta \ge 0$, without loss of generality.

\medskip

Our second main result concerns the behavior of the optimal embedding constant \eqref{eq:defgamma} and the existence, uniqueness, and explicit form of the corresponding minimizers. Moreover, it determines the symmetry group of their representation, which is $\R^+ \times \sSU(2)$, i.e., positive multiples of $2\times 2$ special unitary matrices. 

\begin{maintheorem}[Magnetic stability] \label{thm:magneticstability}
The following holds:
\begin{enumerate}[label=\text{(\roman*)}]
\item\label{itm:mstab-gamma}
	We have that $\beta \mapsto \gamma_*(\beta)$ is a Lipschitz function and satisfies
	\begin{equation}\label{eq:mag-bermuda}
        \gamma_*(\beta) > \max\{\CLGN, 2\pi\beta\}
        \quad \text{for every $0<\beta <2$,}
	\end{equation}
    and 
	\begin{equation}\label{eq:mag-susy}
        \gamma_*(\beta) = 2\pi\beta 
        \quad \text{for every $\beta \geq 2$.}
	\end{equation}
\item\label{itm:mstab-mini}
	Any minimizer of \eqref{eq:defgamma}, if it exists, is smooth. For small enough $0 < \beta < 2$, there exists a minimizer. For $\beta \geq 2$,
    minimizers exist if and only if $\beta \in 2 \N$, 
	and are of the form 
	\begin{equation}\label{eq:mag-solution}
		u = u_{P,Q} := \sqrt{\frac{2 }{\pi \beta}} \, \frac{\overline{P' Q - P Q'}}{|P|^2 + |Q|^2} ,
	\end{equation}
	where $P,Q$ are two coprime and linearly independent complex polynomials satisfying 
	$$
		\max (\deg(P),\deg(Q)) = \frac{\beta}{2}.
	$$
\item\label{itm:mstab-symm}
	Finally, $u_{P,Q} = u_{\tilde{P},\tilde{Q}}$ for two such pairs of polynomials $(P,Q),(\tilde{P},\tilde{Q})$ if and only if
	$(P,Q)=\Lambda(\tilde{P},\tilde{Q})$ for some constant $\Lambda \in \R^+ \times \sSU(2)$.
\end{enumerate}
\end{maintheorem}

Our proof will be given in \Cref{sec:magnetic}, and we also discuss a few symmetry-reduced cases in \Cref{sec:symmetric}.
As mentioned, it will be more natural to postpone 
our discussion of the proof and how it relates to any previous results
until \Cref{thm:magneticstability} is reframed in the context of the application below.
However, we can here make a few immediate general remarks.

\begin{remark}[Supersymmetry]\label{rmk:susy}
The nonlinearity, nonlocality, and inhomogeneity in \eqref{eq:LGN-selfmagnetic} add several layers of complexity in comparison to the external field problem \eqref{eq:LGN-magnetic}.
The lower bounds $\gamma_*(\beta) \ge \max\left\{\CLGN, 2\pi\beta\right\}$ were known previously (cf.\ \cite{LunRou-15,CorLunRou-17}) as they indeed follow immediately from the diamagnetic inequality \eqref{eq:diamag} resp.\ Bogomolnyi's bound \eqref{eq:bogobound}, 
which in the self-generated field reads
\begin{equation}\label{eq:bogobound-self}
    \cE_{\beta}[u]
	\ge \int_{\R^2} B[|u|^2]|u|^2 = 2\pi\beta \int_{\R^2} |u|^4, 
        \quad \beta \ge 0, \ u \in H^1.
\end{equation}
The minimizers \eqref{eq:mag-solution} at 
$\beta \in 2\N$  
exactly saturate this bound (then yielding a Pohozaev identity; cf.\ \cite{Pohozaev-65}, \cite[Ch.~6.2]{Fibich-15}) and thus realize supersymmetry in a sense to be clarified below (and crucial to our proof), whereas we find that such supersymmetry is broken for $0 \le \beta < 2$ due to the strict inequality in \eqref{eq:mag-bermuda}.
To the best of our knowledge, the results \labelcref{eq:mag-bermuda,eq:mag-susy} for $\beta \notin 2\Z$ are new and follow from the monotonicity of 
$\gamma_*(\beta)/\beta$
and a concentration-compactness argument.
\end{remark}

\begin{remark}[Vorticity and the degree of zeros of the minimizers]
\label{rmk:terms} 
The magnetic self-energy \eqref{eq:AF-func} expands to
\begin{equation}\label{eq:AF-crossterms}
    \cE_{\beta}[u] = \int_{\R^2} |\nabla u|^2 
        + 2\beta \int_{\R^2} \bA\left[|u|^2\right] \cdot \bJ[u]
        + \beta^2 \int_{\R^2} \left|\bA\left[|u|^2\right]\right|^2 |u|^2,
\end{equation}
where we have defined the \keyword{current} of $u$ by
\begin{equation}\label{eq:current}
	\bJ[u] := \frac{{\rm i}}{2}(u\nabla \bar{u} - \bar{u}\nabla u),
\end{equation}
and which vanishes identically in the case that $u$ is real-valued or has a constant phase.
The last integral in \eqref{eq:AF-crossterms} depends only on the density of $u$ and may also be written
\begin{equation}\label{eq:MM-curvature}
    \int_{\R^2} \left|\bA[\varrho]\right|^2 \varrho
    = \frac{1}{6} \int_{\R^6} \frac{\varrho(\bx)\varrho(\by)\varrho(\bx)}{\cR(\bx,\by,\bz)^2}\,  \dd\bx \dd\by \dd\bz,
\end{equation}
which is known as the \keyword{Menger--Melnikov curvature} of the probability measure with density $\varrho = |u|^2$,
and where $\cR(\bx,\by,\bz)$ denotes the radius of the circle defined by the
three points $\bx,\by,\bz \in \R^2$;
see \cite[Lem.~3.2]{HofLapTid-08} and \cite{LunRou-15}.
Thus, \eqref{eq:AF-crossterms} for real $u$ involves a competition between diffusion on the plane and localization on a line segment ($\cR \to \infty$).

We note that
there are significant 
differences between the minimizers for $\gamma_{\ast}(\beta)$ for the case that $\beta = 2$ and $\beta >2.$
Indeed, if $\beta = 2$, by setting $P=z$ and $Q=1$ in \eqref{eq:mag-solution}, we have a real minimizer
\begin{align}\label{eq:mag-solution-real}
    u_{P,Q}(z) = \sqrt{\frac{1}{\pi}} \, \frac{1}{|z|^2+ 1},
\end{align}
for $\gamma_{\ast}(2)$, and 
all the minimizers for $\gamma_{\ast}(2)$ are given by \eqref{eq:mag-solution-real} up to translation, dilation, and a constant phase.
However, for every $\beta >2$, there exist no real minimizers for $\gamma_{\ast}(\beta)$, and 
this is highly non-intuitive as we prove, for every $\beta \geq 2$, 
that there exists a sequence of real-valued functions $u_n \in C^{\infty}_c(\R^2;\R)$ which minimizes the ratio 
\eqref{eq:defgamma}
among $u \in H^1$ such that $\int_{\R^2} |u|^2 = 1 = \int_{\R^2} |u|^4$. 
See the further discussion in \Cref{rem:realminimizers,rem:realgammastar}. 

Another interesting difference is the degree of zeros (which is the same as vorticity in this case).
For every $\beta \in 2 \N$, the degree of zeros
of minimizers $u_{P,Q}$, together with their multiplicity, is between $\frac{\beta}{2}-1$ and $\beta -2$. This is an immediate implication of Theorem \ref{thm:magneticstability} and the degree of the Wronskian (see Lemma \ref{lem:degreeofwronskian}). In particular, the minimizers \eqref{eq:mag-solution-real} for $\gamma_{\ast}(2)$ have no zeros, and, for every $\beta \in 2 \N \setminus \{2\}$, the minimizers for $\gamma_{\ast}(\beta)$ have at least one zero (vortex).
We leave it as 
an open problem whether for $0 < \beta < 2$ any minimizers are necessarily non-real and whether they have any zeros.
\end{remark}

\begin{remark}[Vortex quantization]\label{rmk:quantization}
    The quantization of flux $\beta$ to even integers takes its root in complex analysis (compare the degrees of polynomials, i.e.\ entire analytic functions) via the specific form of the corresponding generalized Liouville equation, 
    as clarified below.
    A special case of \eqref{eq:mag-solution} at $\beta=2n$ is the \keyword{``vortex ring''}:    
	\begin{equation}\label{eq:vortexring}
        u_n(z) := \sqrt{\frac{n}{\pi}} \frac{b \, \overline{a(z-z_0)^{n-1}}}{|a(z-z_0)^n + c|^2 + b^2},
        \quad n \in \N,\ a \in \C \setminus \{0\},\ b>0,\ c,z_0 \in \C,
    \end{equation}
    for every $z \in \C,$ which describes a single vortex of degree $n-1$ at $z_0$ 
    (for $n=1$ we let $c=0$ and reobtain \eqref{eq:mag-solution-real}).
    For $c=0$ it is radially symmetric about $z_0$.
    We also note that, taking $c=0$ and replacing $n$ by any $\beta/2 \ge 1$ formally yields a minimizer $u_{\beta/2}$ with a regular and radially symmetric density on $\R^2$, but is not single-valued on $\R^2 \setminus \{z_0\}$ for $\beta \notin 2\Z$.
    This mirrors our above observations for the Liouville problem, and
    whether or not such multi-valued functions should be included in the analysis, i.e.\ whether to extend the domain $H^1$ of the functional \eqref{eq:AF-func}, is a question concerning potential applications and the foundations of quantum mechanics that we do not address in this work (compare Dirac's quantization condition \cite{Dirac-31}).
\end{remark}

\begin{remark}[Soliton manifold completeness]\label{rmk:JP}
    In the case that $Q(z) = a\prod_{j=1}^n (z-z_j)$, where $a \in \C \setminus \{0\}$ and $z_j \in \C$, has exactly $n=\beta/2$ simple roots,
    then 
    $$P = fQ \qquad \text{for} \qquad 
    f(z) = c_0 + \sum_{j=1}^n c_j(z-z_j)^{-1}, \quad 
    c_j \in \C,$$
    the partial fractions,
    and the minimizer \eqref{eq:mag-solution} becomes
    \begin{equation}\label{eq:JP-solutions}
        u_{P,Q} = \sqrt{\frac{2 }{\pi \beta}} \frac{\overline{f'Q^2}}{|fQ|^2 + |Q|^2}
        = \sqrt{\frac{2 }{\pi \beta}} \frac{\overline{f'}}{|f|^2 + 1}
        \frac{\overline{a}^2}{|a|^2}
        \prod_{j=1}^n \frac{(\overline{z-z_j})^2}{|z-z_j|^2}.
    \end{equation}
    Such generic solutions were found by Jackiw and Pi \cite{JacPi-90a},
    while the special case of the radially symmetric vortex rings \eqref{eq:vortexring} with $c=0$ were discussed earlier by the same authors \cite{JacPi-90b}, and later by others \cite{Dunne-99,Hagen-91} (see also \cite{HorYer-98,HorZha-09,JacPi-92}). 
    We are able to show by means of \Cref{thm:liouville} that \eqref{eq:mag-solution} exhausts all possibilities for $\beta \ge 2$ and thus prove the conjecture (claim) in \cite{JacPi-90a} that indeed the entire space of minimizers at $\beta=2n$ has real dimension $4n$, since it is identified as the $4(n+1)$-dimensional space of admissible pairs of polynomials $P,Q$ modulo the faithful orbits of a $4$-dimensional Lie group 
    (a subgroup of the M\"obius group).
\end{remark}

\begin{figure}
\scalebox{1.0}{\includegraphics{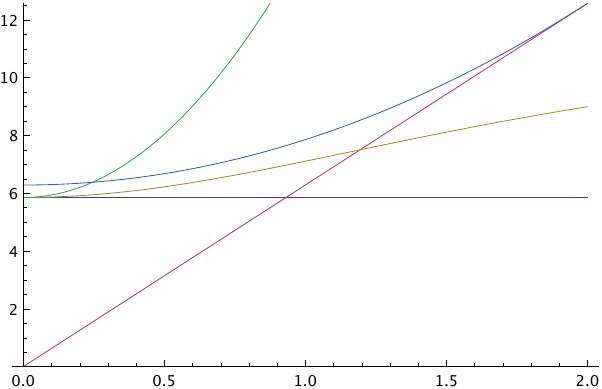}}
\caption{Our current upper and lower bounds for $\gamma_*(\beta)$; see \Cref{rmk:bermuda}. The upper bound from the left (green) is given by Townes soliton $u_0$ and the one from the right (blue) by the lowest ``vortex ring'' $u_1$. The lower bounds are $\CLGN$ (blue), $2\pi\beta$ (red) and \eqref{eq:gamma-bounds-refined} (yellow). Horiz./vert.\ axis for $\beta$ / $\gamma$.}
\label{fig:bounds}
\end{figure}

\begin{remark}[Improved bounds]\label{rmk:bermuda}
    Although we do not have a precise expression for $\gamma_*(\beta)$ in the region $0 < \beta < 2$,
    we are able to give refined lower and upper bounds, summarized in \Cref{cor:basicbounds,prop:vortexringenergy} (see \Cref{fig:bounds}):
	\begin{equation}\label{eq:gamma-bounds-refined}
		\max\left\{ \CLGN + \frac{\pi^2}{\gamma_*(\beta)}\beta^2, 2\pi\beta \right\} 
		\le \gamma_*(\beta) 
		\le \min\left\{ \CLGN\left(1+\frac{3}{2}\beta^2\right), 2\pi\beta + \frac{\pi}{2}(2-\beta)_+^2 \right\}
    \end{equation}
    for all $\beta \ge 0$. In particular, $\gamma_*(\beta)-\CLGN = O(\beta^2)$ as $\beta \to 0$.
    The upper bounds are derived from Townes soliton $u_0 := \tau/\|\tau\|_{L^2}$ and a Hardy inequality involving the circumradius \eqref{eq:MM-curvature} (the optimal constant for which is not known), respectively the unit ``vortex ring'' $u_1$ in \eqref{eq:mag-solution-real}-\eqref{eq:vortexring}.
    Further, by making an expansion of the form \eqref{eq:AF-crossterms} in \eqref{eq:LGN-selfmagnetic}, and applying Bogomolnyi's bound \eqref{eq:bogobound-self}, we derive that for any $u \in H^1$ and $\beta \neq 0$
    \begin{multline*}
        |\beta|^{-1} \int_{\R^2} |u|^2 \int_{\R^2} |\nabla u|^2 
        + |\beta| \left( \int_{\R^2} |u|^2 \right)^{-1} \int_{\R^2} \left|\bA\left[|u|^2\right]\right|^2 |u|^2 \\
        \ge 2\pi \int_{\R^2} |u|^4 
        - 2\sign(\beta) \int_{\R^2} \bA\left[|u|^2\right] \cdot \bJ[u].
    \end{multline*}
    Choosing an appropriate conjugation for $u$ and optimizing over all $\beta \neq 0$, we obtain the following homogeneous interpolation inequality involving the curvature \eqref{eq:MM-curvature}:
    \begin{equation}\label{eq:MM-interpol}
        \left( \int_{\R^2} |\nabla u|^2 \right)^{\frac{1}{2}}
        \left( \int_{\R^2} \left|\bA\left[|u|^2\right]\right|^2 |u|^2 \right)^{\frac{1}{2}}
        \ge \pi \int_{\R^2} |u|^4 + \left| \int_{\R^2} \bA\left[|u|^2\right] \cdot \bJ[u] \right|,
        \quad \forall u \in H^1,
    \end{equation}
    with equality if and only if $u \propto u_{P,Q}$ (corresponding to $\norm{u}_2^{-2} \norm{\nabla u}_2^{-1} \norm{\bA[|u|^2] u}_2 \beta \in 2\Z \setminus \{0\}$).
    This generalizes \cite[Prop.~2.3]{ByeHuhSeo-12}, in which non-negativity and radial symmetry of $u$ was assumed.
\end{remark}

\begin{remark}[Extended interpolation] \label{rmk:extint}
Finally, we comment about the specific form of the interpolation inequality \eqref{eq:LGN-selfmagnetic}. As considered by Dolbeault et al.\ in the case of the external field (see also \cite{FouRay-16}), one may study magnetically improved embedding into any intermediate $L^p$ space, $2 < p < \infty$, and again obtain lower bounds by means of the diamagnetic inequality \eqref{eq:diamag}, however, it is a priori not clear whether Bogomolnyi's bound and our crucial supersymmetric factorization method will be applicable if $p \neq 4$.
The case $p=4$ also comes natural from our application, which we discuss next.
\end{remark}

We summarize the main ideas for the proof of \Cref{thm:magneticstability}. First, we note that minimizers for $\gamma_{\ast}(\beta)$ are equivalent to the functions $u \in H^1$ which are zero levels of
\begin{equation}\label{functional:AFP}
\cE_{\beta,\gamma}[u] :=  \cE_{\beta}[u] - \gamma \int_{\R^2} |u|^4,   
\end{equation}
for $\gamma = \gamma_{\ast}(\beta).$
Now, to apply \Cref{thm:liouville}, we need to change the form of the functional $
\cE_{\beta,\gamma}[u]$
 into an integral of a square for $\gamma =2 \pi \beta$ and $\beta \geq 0$. This idea originates from \cite{AhaCas-79} for general external magnetic fields with finite flux; see also the discussion in \Cref{rmk:Pauli}. We need first our definition of the superpotential.
\begin{definition}
\label{def:superpotentialdef}
    The \keyword{superpotential functional} $\Phi$ is defined, in the principal value sense, by 
\begin{align*}
    \Phi[\varrho](\bx) := \int_{\R^2} \left(\log|\bx-\by| - \log(|\by|+1)\right) \varrho(\by) \,\dd \by,
\end{align*}
    for every $\bx \in \R^2$ and $\varrho \in L^1(\R^2;\R_+)$.
\end{definition}

Note that $\bA[\varrho] = \nabla^{\perp} \Phi[\varrho]$ for every $\varrho \in L^1(\R^2;\R_+)$ 
in the sense of distributions; see Lemma \ref{lem:superpot-relations}. 
Now, to prove \Cref{thm:magneticstability}, we first derive, by \Cref{lem:mag-identity}, the identity 
\begin{align*}
\cE_{\beta,2\pi \beta}[u] 
= \int_{\R^2} \left[
\left|(\nabla + {\rm i}\beta \bA\left[|u|^2\right]) u\right|^2
- 2\pi\beta |u|^4\right] 
= \int_{\R^2} \left|(\partial_1 - {\rm i} \partial_2) \left( e^{\beta \Phi[|u|^2]} u \right) \right|^2 e^{-2 \beta \Phi[|u|^2]}
\end{align*}
for every $\beta \geq 0$ and $u \in H^1$. Afterwards, we need the regularity of the minimizers for $\gamma_{\ast}(\beta)$ (see \Cref{lem:regularityofminimizer}) to imply that $u$ is a minimizer for $\gamma_{\ast}(\beta) = 2\pi \beta \ge 0$ if and only if 
\begin{align*}
u =\bar{f} \,e^{-\beta \Phi[|u|^2]},
\end{align*}
where $f$ is a complex analytic function. Then, by $u \in L^2(\R^2)$ and asymptotic behavior of $\Phi[|u|^2]$ at infinity (see \Cref{lem:newtonlemma}), we derive that $f$ is a polynomial. Hence, by defining the function $$\psi := -2\beta \Phi[|u|^2],$$ it is obtained that 
\begin{align*}
-\Delta \psi = \left|\sqrt{4 \pi \beta} f\right|^2 e^{\psi} \quad \text{in } \R^2,
\end{align*}
where $\sqrt{4 \pi \beta} f$ is a polynomial. Since $$\int_{\R^2} |\sqrt{4 \pi \beta} f|^2 e^{\psi} =\int_{\R^2} 4 \pi \beta |u|^2  = 4 \pi \beta< \infty,$$ we are in the position to apply \Cref{thm:liouville} to derive all the possible solutions $\psi$ and that $\beta \in 2 \N$. Moreover, the symmetry of solutions is derived by the symmetry of the generalized Liouville's equation in \Cref{thm:liouville}. Finally, the value of $\gamma_{\ast}(\beta)$ for every $\beta \geq 2$ will be determined precisely by the fact that $\gamma_{\ast}(\beta)=2\pi \beta$ for $\beta \in 2 \N$ and that $\frac{\gamma_{\ast}(\beta)}{\beta}$ is monotonically decreasing in $\beta$ (see \Cref{lem:keylemma}). In the case where $0 < \beta <2$, by employing further the concentration-compactness method (see \Cref{lem:minimizertrick}), we derive that $\gamma_{\ast}(\beta) > 2 \pi \beta$ for every $0<\beta <2$ and that the minimizers for $\gamma_{\ast}(\beta)$ exist for all $\beta \geq 0$ small enough.

\subsection{An equivalent optimization problem and stability in quantum mechanics}\label{sec:intro-stability}

The original LGN inequality \eqref{eq:LGN} specializes the Sobolev inequality and is a powerful quantitative measure of the uncertainty principle, responsible for the stability of quantum systems with few particles or weak attractions. 
Indeed, such stability is manifest in the 
Gross--Pitaevskii (GP) density functional theory \cite{Gross-61,Pitaevskii-61} (also known as the nonlinear Schr\"odinger (NLS) theory) which is fundamental to the qualitative and quantitative description of the behavior of quantum gases at low temperature. 
Also the Ginzburg--Landau (GL) effective theory of superconductivity \cite{GinLan-50,Ginzburg-09}
manifests a similar functional stability,
and both theories concern the phenomenon of Bose--Einstein condensation (BEC). 
Indeed, when many identical particles converge, their quantum statistics becomes crucially important, which usually refers to whether particles are bosons (such as the photons in laser light) which are free from exclusion and thus prone to condensation, or fermions (such as the electrons in matter) which obey Pauli's exclusion principle.
One of the most celebrated results in mathematical physics is the rigorous proof of stability of fermionic matter in 3D, due to Dyson and Lenard \cite{DysLen-67}, and simplified by Lieb and Thirring \cite{LieThi-75,LieThi-76} with the invention of new functional inequalities which elegantly combine uncertainty and exclusion.
Originally, only electrostatic interactions were considered, but later on mathematically interesting issues concerning self-generated magnetic fields were found \cite{FroLieLos-86,LosYau-86}, and subsequently resolved for fermionic matter \cite{Fefferman-95,LieLosSol-95,LieSei-10}.

In 2D, trapped effective particles may exhibit intermediate quantum statistics \cite{GolMenSha-80,GolMenSha-81,GolMenSha-85,LeiMyr-77,Wilczek-82a,Wilczek-82b}, then termed ``anyons'', and an important and still largely unresolved problem concerns the relationship between the exchange symmetry (given by the braid group) and the exclusion properties for a gas of many such particles; 
see \cite{Girardot-21,Lundholm-24,LunQva-20} for recent reviews.
Proposed effective models for anyon gases use
Chern--Simons (CS) theory to self-consistently couple a magnetic field to a mean-field particle distribution, in the form \eqref{eq:self-field}-\eqref{eq:AF-func};
see \cite{Dunne-95,Dunne-99,Enger-98,IenLec-92,Khare-05,Tarantello-08}
for review. 
Such models have been proposed to describe fundamental aspects of the fractional quantum Hall effect (FQHE) (see \cite{AroSchWil-84,EzaHotIwa-92,LopFra-91,ZhaHanKiv-89}, and \cite{Froehlich-etal-95,LamLunRou-23} for mathematical results),
and some such aspects were even recently experimentally confirmed \cite{Nakamura-etal-20}\footnote{Although as pointed out in \cite{ReaDas-23}, some ambiguities remain, rooted in the mathematical description.}.
Therefore,
we consider the following to be a contribution towards a mathematically rigorous approach to that problem, which generalizes the well-known NLS, GP, and GL theories to a density functional theory (DFT) of anyons.

Our application concerns the following \keyword{``average-field-Pauli'' functional} (with two parameters and an external potential)
\begin{equation}\label{eq:AFP-func}
    \cE_{\beta,\gamma,V}[u] 
    := 
    \int_{\R^2} \left[ \left|(\nabla + {\rm i}\beta\bA[|u|^2]) u\right|^2
    - \gamma |u|^4 
    + V|u|^2 \right],
    \qquad u \in H^1,
\end{equation}
with the constraint $\int_{\R^2} |u|^2=1$.
Its minimizers $u\colon \R^2 \to \C$ describe the approximate collective states (Bose--Einstein condensates) of a gas consisting of a large number of ``almost-bosonic'' anyons
with probability distribution $\varrho = |u|^2$.
The parameter $\beta \in \R$ describes the 
exchange statistics of particles, modeled as magnetically self-interacting by means of \labelcref{eq:self-field,eq:AF-func}, with $\beta$ total flux units.
Thus, for $\beta=0$ they are bosons while $\beta \gg 1$ (on the order of the number of particles) for a ``less bosonic'' and ``more fermionic'' anyon gas.
The parameter $\gamma \in \R$ describes the strength of a contact interaction, which is a simplification of a scalar two-body interaction. 
It is attractive if $\gamma>0$ and repulsive if $\gamma<0$.
Further, the system \eqref{eq:AFP-func} may be trapped by an external potential $V\colon \R^2 \to \R$ (typically $V(\bx) = |\bx|^2$ in a laboratory experiment), and has the ground-state energy (per particle)
\begin{equation}\label{eq:AFP-gse}
	E_{\beta,\gamma,V} := 
		\inf \left\{ \cE_{\beta,\gamma,V}[u] : 
		u \in H^1(\R^2;\C), 
		\int_{\R^2} |V| |u|^2 < \infty,
		\int_{\R^2} |u|^2 = 1 \right\}.
\end{equation}

Noting that the first two terms in \eqref{eq:AFP-func} scale similarly w.r.t.\ dilation (explained below), we obtain the following convenient reframing of \Cref{thm:magneticstability}.

\begin{maintheorem}[Stability for the almost-bosonic anyon gas]\label{thm:stability}
    Let $\beta \in \R$ and the potential $V$ be smooth and bounded from below. The optimal constant $\gamma_*(\beta)$ studied in \Cref{thm:magneticstability} is exactly the critical value for stability of \eqref{eq:AFP-gse} in the sense that $E_{\beta,\gamma,V} > -\infty$
    if $\gamma \le \gamma_*(\beta)$
    and $E_{\beta,\gamma,V} = -\infty$ if $\gamma > \gamma_*(\beta)$.
    Thus, if $V=0$ 
    and $\gamma=\gamma_*(\beta)$ then $E_{\beta,\gamma_*(\beta),0}=0$ for all $\beta$, 
    and if furthermore
    $\beta\ge 2$ then zero-energy ground states exist if and only if $\beta \in 2\Z$, and are then given exactly by the $2\beta$-dimensional soliton manifold \eqref{eq:mag-solution}.
	Finally, for any $\beta,\gamma \in \R$,
	\begin{equation}\label{eq:AFP-KLT}
		E_{\beta,\gamma,V} \ge E^{\rm KLT}_{\beta,\gamma,V} 
		:= \inf \left\{
		\int_{\R^2} \left[ (\gamma_*(\beta)-\gamma) \varrho^2 + V\varrho \right]
		: \varrho \in L^2(\R^2;\R_+), \ \int_{\R^2} \varrho = 1
		\right\}.
	\end{equation}
\end{maintheorem}

The equivalence between \Cref{thm:magneticstability,thm:stability} follows by taking dilations of $u \in H^1$,
\begin{equation}\label{eq:scaling}
u_{\lambda}(\bx) := \lambda u(\lambda \bx),
\quad \lambda > 0,
\end{equation}
which preserves the $L^2$-norm but scales both the magnetic self-energy $\cE_{\beta}$ and the $L^4$-term by $\lambda^2$ (see \cite[Lemma~3.4]{CorLunRou-17}).
Thus, if $\gamma > \gamma_*(\beta)$, we may take $\lambda \to \infty$ with an approximate minimizer of \eqref{eq:defgamma} to obtain instability in \eqref{eq:AFP-gse}, while if $\gamma \le \gamma_*(\beta)$ then stability holds by \eqref{eq:AFP-KLT} (which follows immediately by \Cref{def:gamma}).
In the following, we discuss the previous results concerning the minimization problem \eqref{eq:AFP-gse}, which thus relates both to \Cref{thm:magneticstability,thm:stability}.


The magnetically non-interacting case $\beta=0$ is the well-studied 2D GP/NLS theory;
see \cite[Section~6.2]{LieSeiSolYng-05} and \cite{Rougerie-EMS} for a mathematical exposition.
The 2D cubic NLS equation 
\begin{equation}\label{eq:NLS}
	-\Delta u + u - |u|^2u = 0
\end{equation}
is (after a suitable rescaling) the Euler--Lagrange equation for \eqref{eq:AFP-gse} with $\beta=0$, $V=0$.
It also emerges in laser optics with a self-focusing photon field \cite{CiaGarTow-64,CiaGarTow-65},
and its solution $u=\tau$, now known as ``Townes soliton'', was first studied in this context;
see \cite{Fibich-15} for a comprehensive mathematical treatment.
Existence of a solution was proved in \cite{Yankauskas-66}, its smoothness, positivity and radial symmetry up to constants and translations in 
\cite{GidNiNir-81} (see also \cite{BerLioPel-81}), uniqueness in \cite{LeoSer-87,Kwong-89}, and its relationship to the optimal constant $\gamma_*(0)=\CLGN$ in \cite{Weinstein-83}; 
see, e.g., \cite[Section~5.12 \& 6.3]{Fibich-15}.
Various physically motivated extensions of \labelcref{eq:AFP-func,eq:NLS} to the case of an external magnetic (or rotation) field with the magnetic Laplacian $\Delta_\bA := \nabla_\bA \cdot \nabla_\bA$ were considered in \cite{Seiringer-02,LieSei-06,Aftalion-06}.
See also \cite{DinNguRou-23,LewNamRou-17-proc,Rougerie-EMS}
for related works on rotating Bose gases in the focusing regime, and \cite{BleFou-18,ErdSol-04,FouRupMagAymMik-23} for some relevant works concerning the magnetic Pauli operator.
Further, the lower bound \eqref{eq:AFP-KLT} is known as a Keller--Lieb--Thirring inequality \cite{Keller-61,LieThi-76} and may also be used to estimate the negative spectrum for Schr\"odinger operators with negative potentials, relevant for both interacting bosons and fermions;
see \cite{FraLapWei-22,Lundholm-19} for recent reviews.

For $\beta,\gamma > 0$, the functional \eqref{eq:AFP-func} appears implicitly (and lacking domain considerations) in the Chern--Simons--Ginzburg--Landau--Higgs (CSGLH) theory; for the background and historical developments, see \cite[p.~37--38]{RebSol-84} and \cite{HloSpe-93} concerning GLH and Bogomolnyi's bound, and \cite{Dunne-95,Khare-05,Tarantello-08} concerning CS and self-duality\footnote{%
We here follow the later CS-Schr\"odinger/NLS/GP-like approach. In the GL case the main formal difference is the presence of an external constant magnetic field (this could also be included in \eqref{eq:AFP-func}) which adds a constant term to the Liouville equation, and, more crucially, a change of boundary conditions at $\infty$. ``Higgs'' concerns the form of the potential terms involving the critical $\gamma$.}. 
Briefly, taking variations of a corresponding action functional of fields $u$ and $\bA$, 
in which $\beta$ enters as a coefficient of a topological CS term,
the critical interaction $\gamma=2\pi\beta$ (``self-dual coupling'') amounts to a reduction of the otherwise second-order system of nonlinear PDEs to the first-order self-dual Bogomolnyi equations (see, e.g., \cite[p.~23]{Dunne-95} and \cite[Chapter~7.4]{Khare-05})
\begin{align}\label{eq:self-dual}
	(\partial_1 + \ii A_1)u = \pm \ii(\partial_2 + \ii A_2)u,
\end{align}
and
\begin{align}\label{eq:self-dual-B}
    \partial_1 A_2 - \partial_2 A_1 = 2\pi\beta |u|^2.
\end{align}
Away from any zeros of $u$, we may write $u = \sqrt{\varrho} e^{{\rm i}\omega}$, where $\varrho = |u|^2 > 0$. Then \eqref{eq:self-dual} becomes
$$
    \nabla\omega + \bA = \mp \frac{1}{2}\nabla^\perp \log \varrho,
$$
which, together with the constraint \eqref{eq:self-dual-B}, yields
$$
    -\Delta \log\varrho = \pm 4\pi\beta\varrho,
$$
i.e.\ the pointwise Liouville equation \eqref{eq:Liouville-rho},
provided that we take the same sign as $\beta$.
Seeking regular solutions to \eqref{eq:self-dual}-\eqref{eq:self-dual-B} for both $u$ and $\bA$, Jackiw and Pi first found the radial vortex ring solutions \eqref{eq:vortexring} in \cite{JacPi-90b}, and then the multi-vortex solutions (see \Cref{rmk:JP}) in \cite{JacPi-90a}.
Hagen in \cite{Hagen-91} pointed out a number of mathematical inconsistencies,
and subsequently a more careful but still incomplete mathematical analysis was done by Horv\'athy and Y\'era in \cite{HorYer-98}. Further, because of the substantial interest in determining the completeness of these explicit solutions, methods to evaluate their dimension based on index theory were also developed, but remained non-rigorous; see \cite{HonKimPac-90,JacWei-90,JacLeeWei-90,KimSohYee-90} and further comments below.
Note that there were related issues in the GLH theory, then with only an implicit solution but fully rigorous treatment by Taubes, who proved existence and uniqueness of a multi-vortex solution, starting from an assumption on quantized flux; see \cite{JafTau-80,Taubes-80,Tarantello-08}.
For CS, some further mathematical progress has been made with an equivariant (radial density) assumption; see e.g. \cite{LiLiu-22} for a recent overview.

For $\beta \gg 1$, $\gamma=-2\pi\beta$ (contact repulsion) and $V(\bx)=|\bx|^2$ (harmonic trap) the problem \eqref{eq:AFP-gse} was studied in the context of ideal anyons by Chitra and Sen \cite{ChiSen-92} who obtained an approximation 
$E_{\beta,\gamma,V} \approx \frac{14\sqrt{3}}{9}\sqrt{\beta}$, 
assuming radial symmetry of the density.

For $\beta > 0$, $\gamma=0$ and trapping $V(\bx) \sim |\bx|^s$ with $s>0$, the problem \eqref{eq:AFP-gse} was rigorously shown by Lundholm and Rougerie \cite{LunRou-15} to emerge from an underlying many-body quantum mechanics of extended (non-ideal) anyons in an almost-bosonic limit (see also \cite{Girardot-20} for external magnetic fields).
Further, Correggi et al.\ \cite{CorLunRou-17} proved, to leading order as $\beta \to \infty$, the local density approximation
$$
	E_{\beta,\gamma=0,V} \sim E^{\rm LDA}_{\beta,\gamma=0,V}
		:= \inf \left\{
		\int_{\R^2} \left[ 2\pi c\beta \varrho^2 + V\varrho \right]
		: \varrho \in L^2(\R^2;\R_+), \ \int_{\R^2} \varrho = 1
		\right\},
$$
for a universal constant $c \ge 1$, 
which yet has only been estimated numerically: $c \approx 1.18$ \cite{CorDubLunRou-19}.
Thus, by \eqref{eq:AFP-KLT} this provides an upper bound 
$\gamma_*(\beta)/(2\pi\beta) \lesssim c$ as $\beta \to \infty$.
We stress that $c>1$ is not a contradiction to our theorem, because indeed, guided by the numerics, Correggi et al.\ also conjectured the emergence of a locally homogeneous vortex distribution in the minimizers of $E_{\beta,0,V}$ at a microscopic scale, similar to the Abrikosov lattices of GL and GP theory \cite{AftBlaNie-06b}.
In other words, the densities $\varrho$ for minimizers of $E_{\beta,\gamma,V}$ 
in \eqref{eq:AFP-KLT} cannot be locally constant 
(an often made assumption in the physics literature; see also \cite{CorLunRou-18}), 
which matches with our observations concerning vorticity
for minimizers of 
$E_{\beta,\gamma_*(\beta),0}$
in \Cref{rmk:terms}. We end this section with a brief comment on this and leave this topic for future study.

As far as we are aware, except for the exceptional points $\beta \in 2\Z$, the precise critical coupling $\gamma = \gamma_*(\beta)$ in the stability problem of \Cref{thm:magneticstability,thm:stability} has not been studied previously in this generality.
Finally, for weakly interacting fermions as well as for ideal and almost-fermionic anyons it is known that there is ample margin for stability due to Pauli's exclusion principle \cite{LieThi-76} and its local generalizations \cite{FroMar-89,GolMaj-04,GirRou-23,LarLun-18,LunSei-18,LunQva-20,LunSol-13a,LunSol-13b,LunSol-14,Lundholm-17}, 
and effective DFTs replacing \eqref{eq:AFP-func} for such anyons have been proposed in \cite{ChiSen-92,GirRou-21,Hu-etal-21,LiBhaMur-92,Lundholm-24}.

\medskip

We are now in a position to summarize the main ideas in our proof of \Cref{thm:magneticstability,thm:stability}, and explain how we are able to resolve the issues found in 
the earlier works
\cite{JacPi-90a,JacPi-90b,Hagen-91,Jackiw-91,HorYer-98}.
Similarly to these, 
our starting point is to study the saturation of Bogomolnyi's bound \eqref{eq:bogobound-self} at $\gamma = 2\pi\beta$, and thus seek zero-energy states
$$
	\cE_{\beta,2\pi\beta,0}[u] = \cE_{\beta}[u] - 2\pi\beta \int_{\R^2} |u|^4 = 0,
	\quad \text{for some} \ u \in H^1 \text{ with } \int_{\R^2} |u|^2=1,
$$
leading up to the corresponding self-dual equations \eqref{eq:self-dual}, as shown (non-rigorously) in \cite{JacPi-90a}.
However, we must point out that there is more to these Bogomolnyi relations, 
rooted in a supersymmetric factorization of the Pauli operator, 
which was noted originally by Aharonov and Casher \cite{AhaCas-79} in the following form (without the notion of supersymmetric quantum mechanics),
independently by others in the supersymmetric context
\cite{deVScha-76,BarCasLus-76,deCRit-83,GenKri-85}
(attributing the completion of the square already to Feynman),
and even by Jackiw \cite{Jackiw-84,Jackiw-86} (but not used in \cite{JacPi-90b,JacPi-90a,Jackiw-91,JacPi-92}!):

\begin{remark}[Pauli supersymmetry]\label{rmk:Pauli}
Given an external magnetic field $B = \curl \bA$, 
the following factorization identity holds:
\begin{equation}\label{eq:susy-factorization}
\int_{\R^2} \left[ |\nabla_\bA u|^2 \pm B |u|^2 \right]
= \int_{\R^2} \bigl|(\partial_1 \pm {\rm i}\partial_2)(e^{\mp \Phi}u)\bigr|^2 e^{\pm 2\Phi},
\quad \text{if} \ \bA = \nabla^\perp \Phi,
\end{equation}
where 
$\Phi\colon \R^2 \to \R$ is (nowadays) known as a ``superpotential''.
If the gauge $\bA$ is divergenceless and sufficiently regular then this representation is always possible (also noted in \cite{FroLieLos-86} in 3D).
The l.h.s.\ of \eqref{eq:susy-factorization} describes the two spin components of a 2D Pauli operator, which is the square of a 2D Dirac operator (the ``supercharge''), and thus its joint kernel may be identified as the kernel of the Dirac operator 
(supersymmetric invariants).
In fact, Aharonov and Casher explicitly identified their (non-rigorous) theorem as a variant of the Atiyah--Singer 
index theorem \cite{AtiSin-63,Freed-21}.
Various mathematical issues and generalizations were subsequently pointed out in
\cite{Miller-82},
\cite[Chapter~6.4]{CycFroKirSim-87},
\cite{ErdSol-01,ErdVou-02},
\cite[Chapter~1.4]{FouHel-book},
and
\cite{FouRupMagAymMik-23}.
\end{remark}

In our self-generated case, 
we need to be more careful to define the appropriate superpotential $\Phi = \beta\Phi[|u|^2]$ (compare \Cref{def:superpotentialdef} and $\bA = \beta\bA[|u|^2] = \beta\nabla^\perp\Phi[|u|^2]$) and to derive the corresponding identity \eqref{eq:susy-factorization}.
This is done after the preliminary subsections, in \Cref{sec:mag-susy}.
We find that the kernel 
of \eqref{eq:susy-factorization}
in this case consists of functions $u \in H^1$ on the form
$$
u(\bx) = e^{-\Phi(\bx)} \overline{f(z)},
$$
where $f\colon \C \to \C$ is analytic with a limited growth depending on $\beta$, by the normalization
$\int_{\R^2} |f|^2 e^{-2\Phi} = 1$ and ``Newton's lemma'' for our $\Phi(\bx) \sim \beta \log |\bx|$ (see \Cref{lem:newtonlemma}).
Further, it is also required of the superpotential that 
$$
\Delta \Phi = \curl \bA = B = 2\pi\beta|f|^2e^{-2\Phi}.
$$
Thus, identifying $\psi = -2\Phi$ and rescaling $f$ by $\sqrt{4\pi\beta}$
yields precisely the generalized Liouville equation \eqref{eq:Liouville-psi} of \Cref{thm:liouville} with an integrability condition, 
thereby necessitating the 
general solution $\psi = \psi_{P,Q}$ with $P'Q-PQ' \propto f$
and with quantized flux $\beta \in 2\N$.
Therefore, up to a normalization constant, $\varrho = |u|^2$ is precisely the density in \eqref{eq:Liouville-rho}.
This first part of the proof of \Cref{thm:magneticstability} is given in \Cref{sec:mag-proof-existence,sec:mag-proof-largebeta}.

Our remaining study of $\gamma_*(\beta)$, 
which rests on the fact that $\gamma_*(\beta)=2\pi \beta$ for $\beta \in 2 \N$ and that $\beta \mapsto \frac{\gamma_*(\beta)}{\beta}$ is monotonically decreasing,
is carried out in \Cref{sec:mag-proof-largebeta,sec:mag-proof-smallbeta},
and finally, the symmetry considerations for our
solutions are derived in \Cref{sec:mag-symm,sec:symmetric}.

Note that we have avoided, in the above, some very nontrivial and important questions concerning regularity, and actually this will take up a significant portion of our proofs.
In particular, we need a priori regularity of the minimizers, obtained using bootstrap of the variational equation, 
in \Cref{sec:mag-variation}.
We also note that, although Jackiw and Pi were seeking regular solutions and were aware of some of the potential ambiguities with using only the pointwise Liouville equation \eqref{eq:Liouville-rho} for the density, the arbitrariness in their method, due to their gauge choice, did not allow them to correct it. Hagen pointed towards a better choice of gauge with a superpotential representation, and indeed Jackiw and Pi discussed this technique further and were content with it \cite{JacPi-91b,Jackiw-91}, but neither of them actually succeeded in completing the argument since they did not use the full force of the factorization \eqref{eq:susy-factorization} (including ``Newton's lemma'') to bound the total degree of vortices.
More importantly, they were lacking the complete proof of \Cref{thm:liouville} and the precise roles of integrability and regularity, which rest on more recent developments in analysis, as discussed in \Cref{sec:intro-liouville}. 
Neither Horv\'athy and Y\'era had a complete proof 
since they also assumed the single-valued Liouville representation \eqref{eq:Liouvillerepresentation}, together with additional local regularity and decay of $B \propto |u|^2$ (without studying the variational equation).
We note further that the implicit index-theory based approaches that aimed to verify the completeness of solutions also had ambiguities, either yielding only half the dimension of solutions \cite{Weinberg-79, HonKimPac-90, JacWei-90}, or were otherwise non-rigorous due to the lack of compactness \cite{JacLeeWei-90,KimSohYee-90}
(see \cite{Lundholm-10} for a related problem).
Therefore, we can now finally agree with Hagen \cite{Hagen-91} that,
``one thus sees that the final results of Jackiw and Pi [...] can be rigorously established. This could be of considerable importance in the future if studying such solutions develops into an active field of endeavor.''

\medskip

Before proceeding to the proofs sections, 
we conclude with an observation concerning our spaces of minimizers, of relevance to the FQHE problem (see, e.g., \cite{RajSig-20}).

\begin{remark}[Linear Landau levels]\label{rmk:Landau}
An important example of supersymmetric factorization and index theorem that we have drawn significant inspiration from is that of a constant magnetic field, originally considered by Landau \cite{Landau-30} (in a different gauge), in which
\begin{equation}\label{eq:Landau-field}
B = \text{constant} > 0,
\quad
\bA(\bx) = \frac{B}{2}\bx^\perp, \quad \text{and} \quad 
\Phi(\bx) = \frac{B}{4}|\bx|^2.
\end{equation}
Then \eqref{eq:susy-factorization} with the minus sign is the quadratic form of the Landau Hamiltonian operator:
\begin{equation}\label{eq:Landau-Hamiltonian}
H^{\rm Landau} := -\Delta_\bA - B = \bigoplus_{n=0}^\infty 2nB \1_{\cL_n}.
\end{equation}
Its eigenspaces $\cL_n \subseteq H^1(\R^2)$ are known as the Landau levels, each of which is of infinite dimension due to the infinite flux of the field.
For example, the lowest Landau level
$\cL_0 = \ker H^{\rm Landau}$ is the Segal--Bargmann space \cite{Bargmann-61} spanned by the Darwin--Fock \cite{Darwin-31,Fock-28} orthonormal eigenfunctions of increasing angular momenta $l$: 
$$
\psi_l(\bx) = 
\frac{(B/2)^{(l+1)/2}}{\sqrt{l!\,\pi}}
e^{-B|z|^2/4} \bar{z}^l,
\quad l=0,1,2,\ldots
$$
\end{remark}

In analogy to these conventional \emph{linear} Landau levels, we define for the self-generated field $B = 2\pi\beta|u|^2$, $\bA = \beta\bA[|u|^2]$, $\Phi = \beta\Phi[|u|^2]$, the following \emph{nonlinear} version:
    The \keyword{nonlinear Landau level (NLL)} at flux $\beta \ge 0$
    is defined as the $H^1$-sub\-manifold
    \begin{equation}\label{eq:definoffunctiongammastar}
	\NLL(\beta) := \left\{ u \in H^1 : \cE_{\beta,2\pi\beta,0}[u] = 0, \int_{\R^2} |u|^2 = 1 \right\}.
    \end{equation}
    In accordance with 
    \Cref{thm:magneticstability}, we refer to $\NLL(\beta=2n)$, $n \in \N$, as the $n$:th nonlinear Landau level, with $n=1$ the lowest level.
    Further, we consistently define $\NLL(-\beta) := \overline{\NLL(\beta)}$ for any $\beta > 0$.
We obtain from \Cref{thm:magneticstability} the following immediate corollary:

\begin{corollary}\label{cor:NLL-functional}
    In the $n$:th NLL, $\beta=2n$, if $\gamma\le 4\pi n$ then $E_{\beta,\gamma,V}$ is bounded above by the infimum of
    \begin{equation}\label{eq:AFP-NLL}
        \cE_{\beta,\gamma,V}[u_{P,Q}] 
        = \int_{\R^2} \left[
		(4\pi n - \gamma) \frac{1}{\pi^2 n^2} \frac{|P'Q-PQ'|^4}{(|P|^2+|Q|^2)^4}
		+ V \frac{1}{\pi n} \frac{|P'Q-PQ'|^2}{(|P|^2+|Q|^2)^2} \right]
    \end{equation}
    over the $4(n+1)$-dimensional space of all coprime and linearly independent complex polynomials $P$ and $Q$
    s.t. $\max(\deg P,\deg Q) = n$, 
    or, equivalently, over the $(4n-1)$-dimensional space of all coprime complex polynomials $P$ and $Q$ with highest-order coefficients $a>0$ resp. $1$ and s.t. $0 \le \deg Q < \deg P = n$.
\end{corollary}

\medskip\noindent\textbf{Acknowledgments.} 
Financial support from the Swedish Research Council 
(D.\,L., grant no.\ 2021-05328, ``Mathematics of anyons and intermediate quantum statistics'') 
and the Knut and Alice Wallenberg Foundation (D.-T.\,N.) is gratefully acknowledged.
An initial numerical analysis was based on Matlab code developed by Romain Duboscq. 
We also thank Eduardo Venturini for the shortened proof of Lemma \ref{lem:symmetryofamplitude}, as well as
Daniele Bartolucci, Alexandre Eremenko, Nicolas Rougerie, Wolfgang Staubach, Lars Svensson, and Gabriella Tarantello for their helpful comments and references.

\section{Generalized Liouville equation}\label{sec:liouville}

In this section, we prove \Cref{thm:liouville} and derive some special cases.
We need a few observations concerning the Wronskian of polynomials, the regularity of solutions to the generalized Liouville's equation, and a technical lemma.

\subsection{Preliminary observations}

We denote $\N = \{1,2,3,\ldots\}$, $\R_+ = [0,\infty)$, $\R^+ = (0,\infty)$ and $\C^\times = \C \setminus \{0\}$.
The space of complex polynomials $\C \to \C$ of arbitrary (finite) degree is denoted by $\cP$, and the nonzero polynomials by $\cP^\times = \cP \setminus \{0\}$.

\begin{definition}\label{def:wronskian}
The \keyword{Wronskian} $W$ is the bilinear and antisymmetric functional $W\colon \cP \times \cP \to \cP$ which is defined by\footnote{Note our sign convention.}
\begin{align*}
W(P,Q) := P' Q -P Q'.
\end{align*}
\end{definition}

\begin{lemma}[Degree of Wronskian]
\label{lem:degreeofwronskian}
Let $P,Q \in \cP$ be two complex polynomials. Then, we have two possible cases: 
\begin{enumerate}[label=\text{(\roman*)}]
\item\label{itm:wronsk-zero}
    $W(P,Q)=0$ if and only if $P$ and $Q$ are linearly dependent. 
\item\label{itm:wronsk-nonzero}
    $W(P,Q) \neq 0$, and 
\begin{align}
\label{eq:degreeofwronskian-nonmonic}
\deg(W(P,Q)) \le 2\max(\deg P,\deg Q) - 2.
\end{align}
Further, if $P,Q$ are monic with $\deg(Q) \geq \deg(P)$, then
\begin{align}
\label{eq:degreeofwronskian-monic}
\deg(W(P,Q)) = \deg(P) + \deg (Q-P) - 1.
\end{align}
\end{enumerate}
\end{lemma}
\begin{proof}
  Recall that 
  $P'Q-PQ' = 0$ if and only if $P$ and $Q$ are linearly dependent (by $Q^2(P/Q)' = 0$; see also \cite{Bocher-00} for generalizations).
 Now, assume that $P'Q-PQ' \neq 0$. Then, $P$ and $Q$ are both non-zero polynomials, where one of them is not a constant,
and further, to prove \eqref{eq:degreeofwronskian-nonmonic}, by rescaling we may w.l.o.g.\ assume that they are monic and that $\deg Q \ge \deg P$.

Now, assume that $m:=\deg(P) < n := \deg(Q).$ Then, $P(z) =  z^m + P_1(z)$ and $Q(z) = z^n + Q_1(z)$ for every $z \in \C$, where $P_1,Q_1$ are polynomials of degree at most $m-1,n-1,$ respectively. In conclusion,
\begin{align*}
W(P,Q)(z) ={} & (m z^{m-1} + P_1'(z))(z^n + Q_1(z))- (z^m + P_1(z)) (n z^{n-1} + Q_1'(z)) \\
={} & (m-n) z^{m+n-1} + m z^{m-1} Q_1(z) - n z^{n-1} P_1(z) + z^nP_1'(z) -  z^m Q_1'(z) \\
& + P_1'(z)Q_1(z) - P_1(z)Q_1'(z),
\end{align*}
for every $z \in \C.$ This implies \eqref{eq:degreeofwronskian-nonmonic} and
\begin{align}
\label{eq:degreeofwronskiandiff}
    \deg(W(P,Q)) = \deg(P)+ \deg(Q)- 1=\deg(P)+ \deg(Q-P)- 1.
\end{align}

It remains to consider the case that $m:= \deg(P)=\deg(Q) \ge 1$ and $P,Q$ are not linearly dependent. Then, $P(z) =  z^m + P_1(z)$ and $Q(z) = z^m + Q_1(z)$ for every $z \in \C$, where $P_1 \neq Q_1$ are polynomials of degree at most $m-1$. Hence, $0 \le \deg(Q - P) < \deg(Q)$ and therefore, by the properties of the Wronskian and our previous case \eqref{eq:degreeofwronskiandiff}, we derive
\begin{align*}
   \deg( W(P,Q)) = \deg(W(P, Q -  P)) = \deg(P) + \deg(Q -  P) -1,
\end{align*}
which completes the proof.
\end{proof}

\begin{lemma}[A priori regularity]
\label{lem:apriorireg}
Let $V \in C^{\infty}(\R^2;\R)$ and $\psi \in L^1_\loc(\R^2;\R)$ be a weak solution of
\begin{equation}\label{eq:L-regularity}
-\Delta \psi = V e^{\psi},
\end{equation}
in $\R^2$ such that $V e^\psi \in L^1_{\loc}(\R^2;\R)$. 
Then, $\psi \in C^\infty(\R^2;\R)$.
\end{lemma}

\begin{proof}
We first prove that the r.h.s.\ of \eqref{eq:L-regularity} is actually in $L^2_\loc$.
Consider a smooth non-negative regularization
$$
\phi_{\varepsilon}(\bx) =  \varepsilon^{-2}\phi(\varepsilon^{-1}\bx),
$$ for $\bx \in \R^2$ and $\varepsilon>0$, where $\phi \in C^{\infty}_c(\R^2;\R_+)$ such that $\int_{\R^2} \phi=1$. Then, $\psi_{\varepsilon}:= \phi_{\varepsilon} \ast \psi$ satisfies 
\begin{align*}
-\Delta \psi_{\varepsilon} = \phi_{\varepsilon} \ast (V e^{\psi})
\end{align*}
strongly in $\R^2.$ Now, let $B \subset \R^2$ be an arbitrary (finite) ball and let $\Tilde{\psi}_{\varepsilon} \in W^{1,2}_0(B;\R)$ 
satisfy the Dirichlet boundary-value problem
\begin{equation}
\label{eq:Dirichletequation}
\begin{aligned}
-\Delta \Tilde{\psi}_{\varepsilon} &= \phi_{\varepsilon} \ast (V e^{\psi}), && \textup{in } B,\\
\Tilde{\psi}_{\varepsilon} &=0, && \textup{on } \partial B.
\end{aligned}
\end{equation}
Then, by standard elliptic regularity (see \cite[Chapter 9]{GilTru-01}), $\Tilde{\psi}_{\varepsilon} \in C^{\infty}(B;\R).$ Moreover, 
\begin{align*}
-\Delta (\Tilde{\psi}_{\varepsilon} - \psi_{\varepsilon}) = 0,\quad \textup{in } B, 
\end{align*}
and by the mean value property of harmonic functions, we have 
\begin{align}
\label{eq:harmonicfunctionestim}
\|\Tilde{\psi}_{\varepsilon} - \psi_{\varepsilon}\|_{L^{\infty}\left(\frac{1}{2}B\right)} \leq \frac{1}{|B|} \|\Tilde{\psi}_{\varepsilon} - \psi_{\varepsilon}\|_{L^1(B)}.
\end{align}
Hence, to obtain regularity properties of $\psi_{\varepsilon}$ and $\psi$, we need to derive uniform regularity estimates on $\Tilde{\psi}_{\varepsilon}$.
To estimate $\Tilde{\psi}_{\varepsilon}$, we define the truncation function
\begin{align*}
T_k(\Tilde{\psi}_{\varepsilon})(\bx) := \begin{cases}
\Tilde{\psi}_{\varepsilon}(\bx), \quad &|\Tilde{\psi}_{\varepsilon}(\bx)| \leq k,\\
k \sign(\Tilde{\psi}_{\varepsilon}(\bx)), \quad &|\Tilde{\psi}_{\varepsilon}(\bx)| > k,
\end{cases}
\end{align*}
for every $k>0.$ Now, let us fix $k>0$ which will be determined later. Then, $T_k(\Tilde{\psi}_{\varepsilon}) \in W^{1,2}_0(B;\R)$ by standard properties of Sobolev functions; see \cite[Chapter~4]{EvaGar-92}. Using $T_k(\Tilde{\psi}_{\varepsilon})$ as a test function in \eqref{eq:Dirichletequation}, we derive
\begin{align}
\label{eq:w12estimateontruncation}
\int_{B} |\nabla T_k(\Tilde{\psi}_{\varepsilon})|^2 \leq k C(B,\varepsilon),
\end{align}
where $C(B,\varepsilon) := \int_{B} \phi_{\varepsilon}\ast \left(V e^{\psi}\right)$ for every $\varepsilon>0.$ Then, by the Trudinger--Moser inequality \cite{Moser-71,Trudinger-67}
(see also \cite[Chapter~2.4]{Tarantello-08}), we have
\begin{align}
\label{eq:expoentialestimate}
\int_{B} e^{4\pi \frac{|T_k(\Tilde{\psi}_{\varepsilon})|^2}{kC(B,\varepsilon)}} \leq C |B|,
\end{align}
where $C$ is a constant independent of $B,\varepsilon,k.$ Now, we use \eqref{eq:expoentialestimate} to prove $L^p$-bounds on $e^{|\Tilde{\psi}_{\varepsilon}|}$ for every $p \geq 1.$ 
Define the set 
\begin{align*}
A_{\varepsilon,k} := \left\{\bx \in B: \left|\Tilde{\psi}_{\varepsilon}(\bx)\right| \geq k\right\}.
\end{align*}
Then, by \eqref{eq:expoentialestimate} and Markov's inequality, we obtain
\begin{align*}
|A_{\varepsilon,k}| \leq e^{-\frac{4\pi}{C(B,\varepsilon)} k} \int_{B} e^{4\pi \frac{|T_k(\Tilde{\psi}_{\varepsilon})|^2}{kC(B,\varepsilon)}} \leq C |B| e^{-\frac{4\pi}{C(B,\varepsilon)} k}.
\end{align*}
Hence, for every $0<\alpha <\frac{4\pi}{C(B,\varepsilon)}$, we have, by the layer-cake representation,
\begin{align*}
\int_{B} e^{\alpha |\Tilde{\psi}_{\varepsilon}|} &=  \int_{\{\bx \in B: |\Tilde{\psi}_{\varepsilon}(\bx)| \leq 1\}} e^{\alpha |\Tilde{\psi}_{\varepsilon}|} + \int_{\{\bx \in B: |\Tilde{\psi}_{\varepsilon}(\bx)| > 1\}} e^{\alpha |\Tilde{\psi}_{\varepsilon}|}
\\ & \leq |B| (e^{\alpha}+1) + \alpha \int_{1}^{\infty}  e^{\alpha \lambda} |A_{\varepsilon,\lambda}| \,\dd\lambda
\\ & \leq |B| (e^{\alpha}+1) + C |B|\alpha\int_{1}^{\infty}  e^{\left(\alpha-\frac{4\pi}{C(B,\varepsilon)}\right) \lambda}   \,\dd\lambda 
\\ & \leq |B| (e^{\alpha}+1) + \frac{ C |B|\alpha }{ \left(\frac{4\pi}{C(B,\varepsilon)}-\alpha\right)} e^{\left(\alpha-\frac{4\pi}{C(B,\varepsilon)}\right) }.
\end{align*}
Now, note that we have $\frac{4\pi}{C(B,\varepsilon)} \to \infty$ if $|B|+ \varepsilon \to 0$. Therefore, for every $\alpha >0$ there exist $\varepsilon>0$ and a ball $B$ small enough, depending on $\alpha$, such that
\begin{align*}
\int_{B} e^{\alpha |\Tilde{\psi}_{\varepsilon}|} 
\leq |B| (e^{\alpha}+1) + \frac{ C |B|\alpha }{ \left(\frac{4\pi}{C(B,\varepsilon)}-\alpha\right)} e^{\left(\alpha-\frac{4\pi}{C(B,\varepsilon)}\right) } < \infty,
\end{align*}
Let us take $B,\varepsilon$ small enough such that $\alpha =2 <\frac{4\pi}{C(B,\varepsilon)}$. Then,
\begin{align*}
\int_{B} e^{2 |\Tilde{\psi}_{\varepsilon}|}  
\leq |B| (e^{2}+1) + \frac{ 2 C |B|}{ \left(\frac{4\pi}{C(B,\varepsilon)}-2\right)} e^{\left(2-\frac{4\pi}{C(B,\varepsilon)}\right) }.
\end{align*}
Letting $\varepsilon \to 0$ and using \eqref{eq:harmonicfunctionestim}, together with Fatou's lemma, we have 
\begin{align}
\label{eq:L2boundonexpoent}
\int_{B} e^{2 |\psi|} < \infty,
\end{align}
by taking $B$ small enough. In fact, by a covering argument, \eqref{eq:L2boundonexpoent} holds for any (finite) ball $B \subset \R^2$. 
Hence, the r.h.s. of \eqref{eq:L-regularity} is in $L^2(B)$.
By elliptic regularity (see \cite[Chapter~9]{GilTru-01}), we obtain $\psi \in W^{2,2}(B;\R)$, and by Sobolev embedding, $\psi \in L^\infty_\loc$.
In conclusion, by a standard bootstrap argument 
we arrive at $\psi \in C^{\infty}(\R^2;\R)$.
\end{proof}

\begin{lemma}
\label{lem:magicalcalculation}
Let $P,Q$ be two complex polynomials and $z \in \C$ a point where $$|P(z)|^2 + |Q(z)|^2 >0.$$ Then,
\begin{equation}\label{eq:magicalcalculation}
\frac{1}{4}\Delta \log(|P|^2 + |Q|^2)(z) = \frac{|P'(z)Q(z) -P(z) Q'(z)|^2}{(|P(z)|^2 + |Q(z)|^2)^2}.
\end{equation}
\end{lemma}
\begin{proof}
Note that $|P|^2 + |Q|^2>0$ in a neighborhood around the point $z$. Hence,
\begin{align*}
\text{l.h.s. of } \eqref{eq:magicalcalculation}  
& = \partial_{\bar{z}}\partial_z \log (|P(z)|^2+|Q(z)|^2) \\
& = \partial_{\bar{z}} \frac{P'(z)\overline{P(z)} + Q'(z)\overline{Q(z)}}{|P(z)|^2+|Q(z)|^2} \\
&= \frac{|P'(z)|^2 + |Q'(z)|^2}{|P(z)|^2+|Q(z)|^2}
- \frac{(P'(z)\overline{P(z)} + Q'(z)\overline{Q(z)})(P(z)\overline{P'(z)} + Q(z)\overline{Q'(z)})}{(|P(z)|^2+|Q(z)|^2)^2} \\
&= \text{r.h.s. of } \eqref{eq:magicalcalculation},
\end{align*}
which completes the proof.
\end{proof}

\subsection{Proof of Theorem~\ref{thm:liouville}}

The proof of our first main theorem is divided into three parts, namely concerning the sufficient and necessary form of the general solutions and the symmetry of the solutions.

\subsubsection{Sufficient form of solutions}

Let $f, P,Q$ be non-zero complex polynomials, which satisfy both $\gcd(P,Q)=1$ (coprime) and $f = P' Q - P Q'$. Since $P$ and $Q$ are non-zero and $\gcd(P,Q)=1$, we have $|P|^2 + |Q|^2 >0$ in $\R^2$. 
Hence,
$\frac{|P'Q - P Q'|^2}{(|P|^2+ |Q|^2)^2}$ is a smooth function. Moreover, by
\Cref{lem:degreeofwronskian} 
we have 
\begin{align*}
\frac{|P'(z)Q(z) - P(z) Q'(z)|^2}{(|P(z)|^2+ |Q(z)|^2)^2} \leq C |z|^{-4},
\end{align*}    
for large enough $z \in \C$ and a constant $C>0$, which depends on the degrees and coefficients of $P,Q$. In conclusion,
\begin{align*}
\frac{|P'(z)Q(z) - P(z) Q'(z)|^2}{(|P(z)|^2+ |Q(z)|^2)^2} \in L^1(\R^2).  
\end{align*}

Now, we prove that 
\begin{align*}
\psi = \psi_{P,Q} :=  \log(8) - 2\log(|P|^2 + |Q|^2)
\end{align*}
satisfies  $\int_{\R^2} |f|^2 e^{\psi} < \infty$ and $-\Delta \psi = |f|^2 e^{\psi}$ strongly in $\R^2.$ By definition of $\psi$, we have
\begin{align*}
\int_{\R^2} |f|^2 e^{\psi} =  8 \int_{\R^2} \frac{|P' Q -P Q'|^2}{(|P|^2 + |Q|^2)^2} < \infty.
\end{align*}
Moreover, by \Cref{lem:magicalcalculation}, we have 
\begin{align*}
& -\frac{1}{8} \Delta \psi = \frac{1}{4}\Delta \log (|P|^2+|Q|^2)   
= \frac{1}{8} |f|^2 e^{\psi},
\end{align*}
which proves that $-\Delta \psi = |f|^2 e^{\psi}$ in $\R^2$.

Finally, by the Gauss--Green theorem, we have 
\begin{align*}
& \frac{\int_{\R^2} |f|^2 e^{\psi}}{8\pi} = -\frac{1}{8\pi  } \lim_{R \to \infty} \int_{B(0,R)} \Delta \psi
= -\frac{1}{8 \pi} \lim_{R \to \infty} \int_{\partial B(0,R)} \nabla \psi \cdot \bn \,|\dd z| \\
&= \frac{1}{8 \pi} \lim_{R \to \infty} \int_{\partial B(0,R)} 4\frac{(\Re( P(z)\overline{P'(z)} + Q(z)\overline{Q'(z)}), \Im(P(z)\overline{P'(z) }+Q(z)\overline{Q'(z)})) \cdot \bn}{|P(z)|^2 + |Q(z)|^2}\, |\dd z|
\\ &=\max(\deg(P),\deg(Q)),
\end{align*}
where $\bn$ is the \emph{outward} unit normal vector on $\partial B(0,R)$.
This completes the first part of the theorem.
\qed

\subsubsection{Necessary form of solutions}
In this step, we prove the general form of solutions to the generalized Liouville equation \eqref{eq:Liouville-psi}. The proof is divided into three parts:
\begin{itemize}
    \item We construct local meromorphic solutions (sections) around every point.
    \item We use M\"obius transformations to patch such local solutions to a global one.
    \item We study its behavior at infinity.
\end{itemize}

\smallskip

Let $f\colon \C \to \C$ be a nonzero polynomial and 
$\psi \in L^1_\loc(\R^2;\R)$ 
be a weak solution of 
\begin{align*}
-\Delta \psi = |f|^2 e^{\psi},
\end{align*}
in $\R^2$, where $\int_{\R^2} |f|^2 e^{\psi}< \infty$. Then, by \Cref{lem:apriorireg}, we obtain that $\psi \in C^{\infty}(\R^2).$ Now, define the new function
\begin{align}
\label{eq:newfunction}
\phi := \log(|f|^2) + \psi.
\end{align}
Then,
\begin{align*}
-\Delta \phi = e^{\phi} - \sum_{j=1}^{k} 4\pi n_j \delta_{z_j},
\end{align*}
in $\R^2$ and $\int_{\R^2} e^{\phi} < \infty$, where $z_j$ are the roots of $f$ with multiplicities $n_j$ for $1 \leq j \leq k$ and $\delta_{z_j}$ is the unit Dirac mass at $z_j$. Then, by a generalization of Liouville's theorem on radially symmetric domains; see \cite[Corollary 3.4]{BriHouLei-05} and \cite[Theorem 3]{ChoWan-94}, we have the generalized Liouville representation 
\begin{align}\label{eq:liouvillelocalfomulation}
e^{\phi} = 8 \frac{|g_{D}'|^2}{(1+ |g_{D}|^2)^2},
\end{align}
on the set $D \setminus \{z_j\}_{j=1}^{\deg(f)}$ for every disk $D$ which contains (at its center) at most one root of $f$.
Further, either $g_{D}$ is meromorphic (if $D$ contains no root of $f$) or $g_{D}(z) = h_D(z) (z-z_i)^{\alpha_{D}}$ if $z_i \in D$, where $\alpha_D \in \R$ and $h_D$ is single-valued meromorphic on the set $D \setminus  \{z_j\}_{j=1}^{\deg(f)}$. Moreover, since $e^{\phi} \in L^1(\R^2)$, by the proof of \cite[Thm.~5]{ChoWan-94}, we have that such $h_{D}$ in above has in fact no essential singularity on $D$. Let $D$ be a disk such that $D \cap \{z_j\}_{j=1}^{\deg(f)} = z_i$ where $1 \leq i \leq n.$ Then, $g_{D} =\Tilde{h}_D (z-z_i)^{\alpha_{D}+k}$ where $k \in \Z$ and $\Tilde{h}_D $ is a holomorphic function in a neighborhood $z_i$, which is nonzero at $z_i$. Thus, by \Cref{lem:apriorireg},
\eqref{eq:newfunction}, and \eqref{eq:liouvillelocalfomulation}, we have 
\begin{align*}
O(1) + 2 n_i \log|z-z_i| = \phi = O(1) + 2(\alpha_D + k) \log|z-z_i|, 
\end{align*}
in a neighborhood of $z_i$. In conclusion, $\alpha_D = n_i-k \in \Z$ and $g_D$ is a meromorphic function in $D$. 

Now, we want to define a global meromorphic function by patching M\"obius transformations of functions $g_{D}$. 
Since there are finitely many roots of $f$, by taking an $\varepsilon>0$ small enough, we can assume that for each square $Q$ in $\R^2$ of side length $\varepsilon$ or smaller, there exists a meromorphic function $g_{Q}$ on $Q$ such that the Liouville representation
\begin{align}
\label{eq:liouvillefomulationcube}
e^{\phi} = 8 \frac{|g_{Q}'|^2}{(1+ |g_{Q}|^2)^2}
\end{align}
is valid in $Q$. Let $Q_1 := \left(-\frac{\varepsilon}{2},\frac{\varepsilon}{2}\right) \times \left(-\frac{\varepsilon}{2},\frac{\varepsilon}{2}\right)$. Define $g_1$ as the corresponding meromorphic function on $Q_1$. Now, we define the meromorphic functions $g_n$ on 
$$Q_n := \left(-\frac{n \varepsilon}{2},\frac{n \varepsilon}{2}\right) \times \left(-\frac{n \varepsilon}{2},\frac{n \varepsilon}{2}\right)$$ for $n>1$ by induction as follows: 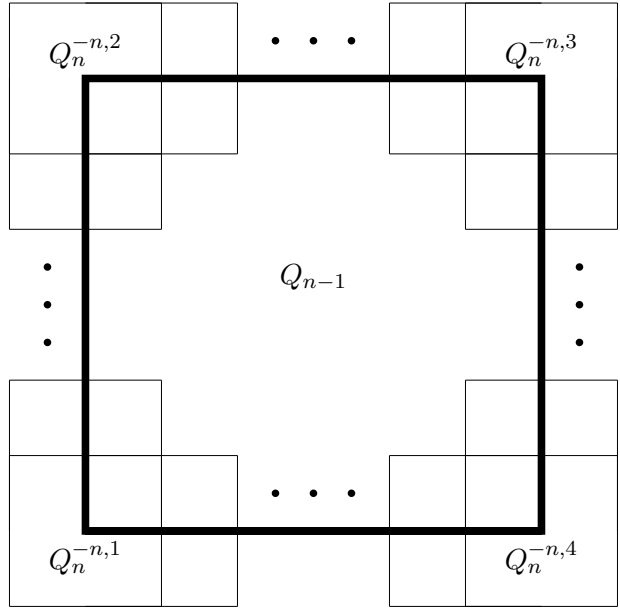
\begin{figure}
\centering
\begin{tikzpicture}
\node[above] at (4,7) {$Q^{-n,2}_n$};
\node[below] at (4,1) {$Q^{-n,1}_n$};
\node[above] at (10,7) {$Q^{-n,3}_n$};
\node[below] at (10,1) {$Q^{-n,4}_n$};
\node[circle,inner sep=1pt,fill=white,label=above:{$Q_{n-1}$}] at (7,4) {};
\draw[line width=1mm,black] (4,1) -- (10,1) -- (10,7) -- (4,7) -- (4,1) -- (10,1);
\draw[line width=0.1mm,black] (3,6) -- (3,8);
\draw[line width=0.1mm,black] (3,8) -- (5,8);
\draw[line width=0.1mm,black] (5,8) -- (5,6);
\draw[line width=0.1mm,black] (5,6) -- (3,6);
\draw[line width=0.1mm,black] (4,8) -- (6,8);
\draw[line width=0.1mm,black] (6,8) -- (6,6);
\draw[line width=0.1mm,black] (6,6) -- (4,6);
\draw[line width=0.1mm,black] (3,7) -- (3,5); 
\draw[line width=0.1mm,black] (3,5) -- (5,5);
\draw[line width=0.1mm,black] (5,5) -- (5,7);
\node[circle,inner sep=1pt,fill=black] at (6.5,7.5) {};
\node[circle,inner sep=1pt,fill=black] at (7,7.5) {};
\node[circle,inner sep=1pt,fill=black] at (7.5,7.5) {};
\node[circle,inner sep=1pt,fill=black] at (3.5,4.5) {};
\node[circle,inner sep=1pt,fill=black] at (3.5,4) {};
\node[circle,inner sep=1pt,fill=black] at (3.5,3.5) {};
\draw[line width=0.1mm,black] (3,0) -- (5,0);
\draw[line width=0.1mm,black] (5,0) -- (5,2);
\draw[line width=0.1mm,black] (5,2) -- (3,2);
\draw[line width=0.1mm,black] (3,2) -- (3,0);
\draw[line width=0.1mm,black] (3,1) -- (3,3);
\draw[line width=0.1mm,black] (3,3) -- (5,3);
\draw[line width=0.1mm,black] (5,3) -- (5,1);
\draw[line width=0.1mm,black] (4,0) -- (6,0);
\draw[line width=0.1mm,black] (6,0) -- (6,2);  
\draw[line width=0.1mm,black] (6,2) -- (4,2);  
\node[circle,inner sep=1pt,fill=black] at (6.5,1.5) {}; \node[circle,inner sep=1pt,fill=black] at (7,1.5) {}; \node[circle,inner sep=1pt,fill=black] at (7.5,1.5) {};
\draw[line width=0.1mm,black] (9,0) -- (11,0);
\draw[line width=0.1mm,black] (11,0) -- (11,2);
\draw[line width=0.1mm,black] (11,2) -- (9,2);
\draw[line width=0.1mm,black] (9,2) -- (9,0);
\draw[line width=0.1mm,black] (8,0) -- (10,0);
\draw[line width=0.1mm,black] (10,2) -- (8,2);
\draw[line width=0.1mm,black] (8,2) -- (8,0);
\draw[line width=0.1mm,black] (11,1) -- (11,3);
\draw[line width=0.1mm,black] (11,3) -- (9,3);
\draw[line width=0.1mm,black] (9,3) -- (9,1);
\draw[line width=0.1mm,black] (9,6) -- (11,6);
\draw[line width=0.1mm,black] (11,6) -- (11,8);
\draw[line width=0.1mm,black] (11,8) -- (9,8);
\draw[line width=0.1mm,black] (9,8) -- (9,6);
\draw[line width=0.1mm,black] (11,7) -- (11,5);
\draw[line width=0.1mm,black] (11,5) -- (9,5);
\draw[line width=0.1mm,black] (9,5) -- (9,7);
\draw[line width=0.1mm,black] (8,6) -- (10,6);
\draw[line width=0.1mm,black] (10,8) -- (8,8);
\draw[line width=0.1mm,black] (8,8) -- (8,6);
\node[circle,inner sep=1pt,fill=black] at (10.5,4.5) {};
\node[circle,inner sep=1pt,fill=black] at (10.5,4) {};
\node[circle,inner sep=1pt,fill=black] at (10.5,3.5) {};
\end{tikzpicture}
\caption{Construction of squares.}
\label{fig:squares}
\end{figure}
Define the squares
\begin{align*}
Q_n^{i,1} &:=\left( -\frac{n \varepsilon}{2},  -\frac{(n-2) \varepsilon}{2}\right) \times \left(\frac{i \varepsilon}{2},\frac{(i+2) \varepsilon}{2}\right), \\
Q_n^{i,2} &:= \left(\frac{i \varepsilon}{2},\frac{(i+2) \varepsilon}{2}\right) \times \left( \frac{(n-2) \varepsilon}{2},  \frac{n \varepsilon}{2}\right), \\
Q_n^{i,3} &:= \left( \frac{(n-2) \varepsilon}{2},  \frac{n \varepsilon}{2}\right) \times \left(-\frac{(i+2) \varepsilon}{2},-\frac{i \varepsilon}{2}\right), \\
Q_n^{i,4} &:= \left(-\frac{(i+2) \varepsilon}{2},-\frac{i \varepsilon}{2}\right) \times \left( -\frac{n \varepsilon}{2},  -\frac{(n-2) \varepsilon}{2}\right) 
\end{align*}
for every $-n \leq i < n-2$. 
Now, we construct meromorphic functions $g_n$ on $Q_n$ using induction.  Let $g_{n-1}$ be the constructed meromorphic function on $Q_{n-1}$, and assume that meromorphic $g_{i,j}$ are defined on $Q_n^{i,j}$ for integers $-n\leq i < n-2$, $1 \leq j \leq 4$ and satisfying \eqref{eq:liouvillefomulationcube}. Since there are finitely many points $\{z_l:1 \leq l \leq k \}$ in the non-empty convex set $Q^{i,j}_n \cap Q_{n-1}$, we can find a smaller square $\Tilde{Q}^{i,j}_n \subset Q^{i,j}_n \cap Q_{n-1}$ which does not contain any points in $\{z_l:1 \leq l \leq k \}$ and has the representation \eqref{eq:liouvillefomulationcube}. Hence, by uniqueness up to M\"obius transformations, \cite[Rem.~3]{BriLei-03}, we must have 
\begin{equation}\label{eq:liouvillesmallsquare}
    g_{n-1} = \frac{a_{i,j} g_{i,j}-\overline{c_{i,j}}}{c_{i,j} g_{i,j} + \overline{a_{i,j}}},
\end{equation}
in any such $\Tilde{Q}^{i,j}_n$ for every integers $-n\leq i < n-2,1 \leq j \leq 4$, where $a_{i,j}, c_{i,j} \in \mathbb{C}$ satisfy $|a_{i,j}|^2 +  |c_{i,j}|^2>0$. By the unique continuation property of meromorphic functions, we obtain that 
\eqref{eq:liouvillesmallsquare} holds
in any $ Q^{i,j}_n \cap Q_{n-1}$ for every integers $-n\leq i < n-2$, $1 \leq j \leq 4$. 
Thus, by redefining $g_{i,j}$ by 
\begin{align*}
\Tilde{g}_{i,j} := \frac{a_{i,j} g_{i,j}-\overline{c_{i,j}}}{c_{i,j} g_{i,j} + \overline{a_{i,j}}}, 
\end{align*}
on $Q_n^{i,j}$, we obtain meromorphic functions which are equal to $g_{n-1}$ in $Q_n^{i,j} \cap Q_{n-1}$ and satisfy
\begin{align*}
e^{\phi} = 8 \frac{|\Tilde{g}_{i,j} '|^2}{(1+ |\Tilde{g}_{i,j} |^2)^2},
\end{align*}
in $Q_n^{i,j}$ for every  $-n\leq i < n-2$, $1 \leq j \leq 4$. We need to show that the function $g_n$ defined as $g_{n-1}$ on $Q_{n-1}$ and $\Tilde{g}_{i,j}$ on $Q^{i,j}_n$ is the suitable candidate on $Q_n$, which only requires that $\Tilde{g}_{i_1,j_1}, \Tilde{g}_{i_2,j_2}$ coincide in the intersection of two different squares $Q^{i_1,j_1}_n,Q^{i_2,j_2}_n$. To prove this, we prove the claim that $Q^{i_1,j_1}_n \cap Q^{i_2,j_2}_n \cap Q_{n-1}$ is a non-empty open convex set if $Q^{i_1,j_1}_n \cap Q^{i_2,j_2}_n$ is non-empty.
By definition $Q_n^{i_1,j_1},Q_n^{i_2,j_2}$ are different and have non-empty intersection if and only if $j_1 = j_2$, $|i_1 -i_2| =1$ or $j_2 = j_1+1$, $i_1 = n-3,i_2 = -n+1$, where $Q_n^{i_1,5}:= Q_n^{i_1,1}$ for every $-n \leq i_1 \leq n-2.$ We prove the claim in one special case and the other cases follow by the same argument. Assume that $i_2=i_1+1$, $j_2 = j_1=1$. Then,
\begin{align*}
Q_n^{i_1,j_1} \cap Q_n^{i_2,j_2} \cap Q_{n-1} &= \left( -\frac{n \varepsilon}{2},  -\frac{(n-1) \varepsilon}{2}\right) \times \left(\frac{(i_1+1) \varepsilon}{2},\frac{(i_1+2) \varepsilon}{2}\right), 
\end{align*}
and $\Tilde{g}_{i_1,j_1} = \Tilde{g}_{i_2,j_2} = g_{n-1}$ on $Q_n^{i_1,j_1} \cap Q_n^{i_2,j_2} \cap Q_{n-1}.$ This completes the proof of the claim. Since $\Tilde{g}_{i_1,j_1}, \Tilde{g}_{i_2,j_2}$ are meromorphic functions which coincide on the subset $ Q_n^{i_1,j_1} \cap Q_n^{i_2,j_2} \cap Q_{n-1}$ of $Q_n^{i_1,j_1} \cap Q_n^{i_2,j_2},$ we imply that $\Tilde{g}_{i_1,j_1}$ and $ \Tilde{g}_{i_2,j_2}$ coincide in $Q_n^{i_1,j_1} \cap Q_n^{i_2,j_2}.$ Hence, $g_n$ is a well-defined meromorphic function on $Q_n$, which coincides with $g_{n-1}$ on $Q_{n-1}$ and satisfies 
\begin{align*}
e^{\phi} = 8 \frac{|g'_{n}|^2}{(1+ |g_{n}|^2)^2}
\end{align*}
on $Q_n$. This completes the construction of $g_n\colon Q_n \to \C$. Now, define $g\colon \C \to \C$ by $g= g_n$ on $Q_n$. Then, $g$ is a well-defined meromorphic function on $\C$ (a global section) and satisfies  
\begin{align}
\label{eq:meromorphicfun}
e^{\phi} = 8 \frac{|g'|^2}{(1+ |g|^2)^2}.
\end{align}

Now, we prove that the singularity of $g$ at $\infty$ is not essential (see also \cite[Lem.~3]{HorYer-98} for an alternative proof).
First, we note that $\int_{\R^2} |f|^2 e^{\psi} < \infty$ and $\psi \in C^{\infty}(\R^2)$ by \Cref{lem:apriorireg}. Then, $\int_{\R^2} (|f| ^2+1) e^{\psi} < \infty$. Hence,
by \cite[Thm.~2]{CheLi-93}, we have 
\begin{align*}
-C -\beta \log(|\bx|+1) \leq \psi(\bx) \leq - \beta \log(|\bx|+1) + C
\end{align*}
for every $\bx \in \R^2$ and some positive constants $C,\beta$ such that $\beta > 2\deg(f)+2$.
In conclusion, 
\begin{align}\label{eq:derivativearctangent}
|\nabla \arctan(|g|)|(\bx)	=	\frac{|\nabla |g(\bx)|\, |}{1+ |g(\bx)|^2} \leq  \frac{|g'(\bx)|}{1+ |g(\bx)|^2} =  \sqrt{\frac{1}{8}} e^{\frac{\phi(\bx)}{2}} = \sqrt{\frac{1}{8}} |f(\bx)| e^{\frac{\psi(\bx)}{2}}  \leq \frac{\Tilde{C}}{|\bx|^{\alpha+1}},
\end{align}
for every $|\bx| >1,$ which does not belong to the zeros or poles of $g,$ and positive constants $\Tilde{C}$ and $\alpha$. Now, take $R>1$ such that there exists no zero or pole of $g$ on the set $\partial B(0,R)$. First, by using the polar coordinate description of 
$\nabla = \be_r \frac{\partial}{\partial r} + \be_\theta \frac{1}{r} \frac{\partial}{\partial \theta}$
and integrating \eqref{eq:derivativearctangent} from $\theta_1 \in \R$ to $\theta_2 \in \R$ (mod $2\pi$), we obtain
\begin{align*} 
  \frac{1}{R} \biggl|  \int_{\theta_1}^{\theta_2} \partial_{\theta} \arctan(|g|)(R e^{i \theta}) \, \dd \theta \biggr| \leq \frac{2 \pi \Tilde{C}}{R^{\alpha+1}}.
\end{align*}
Hence,
\begin{align}\label{eq:thetaintegration}
    \left| \arctan(|g|)(R e^{i \theta_2}) - \arctan(|g|)(R e^{i\theta_1}) \right| \leq \frac{2 \pi \Tilde{C}}{R^{\alpha}},
\end{align}
for every $\theta_1,\theta_2 \in \R.$ Now, if $|g(\bx)| > \frac{1}{2}$ for every $\bx \in \R^2$ large enough, then by Casorati--Weierstrass theorem, we obtain that $g$ has no essential singularity at $\infty$ and hence it is a rational polynomial. Thus, we can restrict ourselves to the case that for every $M>0$ there exist $r >M$ and $\theta  \in \R$ such that 
\begin{align*}
    |g(r e^{i \theta})| \leq \frac{1}{2}.
\end{align*}
Now, by using \eqref{eq:thetaintegration}, we derive that for every $M>0$ there exists $r >M$ such that 
\begin{align}
\label{eq:smallaboslute}
    |g(r e^{i \theta})| < 1,
\end{align}
for every $\theta \in \R.$ Hence, we can choose $R$ large enough such that there exists no root or pole of $g$ on $\partial B(0,R)$ and 
\begin{align}
\label{eq:conditiononarctan}
  \arctan(|g|)(R e^{i \theta}) + \frac{2 \Tilde{C} \alpha}{R^{\alpha}}\leq \frac{\pi}{3}.
\end{align} for every $\theta \in \R$.
Now, by integrating \eqref{eq:derivativearctangent} from $R$ to $r>R$ and using \eqref{eq:conditiononarctan}, it is obtained that
\begin{align*}
 0 \leq \arctan(|g(r e^{{\rm i}\theta})|) \leq \arctan(|g(R e^{{\rm i}\theta})|) +\Tilde{C} \alpha \left( \frac{1}{R^{\alpha }} - \frac{1}{r^{\alpha }}\right) \leq \frac{\pi}{3},
\end{align*}
for every $r >R$ and real $\theta.$ By taking the tangent of both sides, we derive that
\begin{align*}
|g(r e^{{\rm i}\theta})| \leq \frac{|g(R e^{{\rm i}\theta})| + \tan\left( \Tilde{C} \alpha \left( \frac{1}{R^{\alpha }} - \frac{1}{r^{\alpha}} \right) \right)}{1-|g(R e^{{\rm i}\theta})| \tan\left(\Tilde{C} \alpha \left( \frac{1}{R^{\alpha }} - \frac{1}{r^{\alpha}} \right) \right)}.
\end{align*}
In conclusion, 
\begin{align*}
\lim_{r \to \infty} \sup_{\theta} \frac{ |g(r e^{{\rm i}\theta})|}{r^{\alpha}} =0,
\end{align*}
which implies that $g$ has at most pole type singularity at infinity. Hence, $g = \frac{P}{Q}$ for two non-zero and linearly independent polynomials $P,Q$ such that $\gcd(P,Q)=1$.

Finally, we use the equation for $g$ to derive the properties of $P,Q$. By \eqref{eq:meromorphicfun}, we have 
\begin{align*}
\frac{|P' Q - Q' P|}{|P|^2 + |Q|^2} = \sqrt{\frac{1}{8}} e^{\frac{\psi}{2}} |f|. 
\end{align*}
Hence, by considering the roots, we may identify
\begin{align*}
P' Q - Q' P = c f \quad \text{and} \quad
|P|^2 + |Q|^2 =\sqrt{8}\,  |c|\,  e^{-\frac{\psi}{2}}
\end{align*}
for a constant $c \in \C^\times$. Now, replacing $P$ and $Q$ by $\frac{c}{\sqrt{|c|}}P$ and $\sqrt{|c|}Q$, respectively, we obtain 
\begin{align*}
P' Q -  P Q' =  f, \quad \gcd(P,Q)= 1, \quad \text{and} \quad |P|^2 + |Q|^2 =  \sqrt{8} \,  e^{-\frac{\psi}{2}}.
\end{align*}
In conclusion,
\begin{align*}
\psi = \log( e^{\psi} ) = \log\left( \frac{8}{(|P|^2 + |Q|^2)^2} \right)= \log(8) - 2\log(|P|^2 + |Q|^2) = \psi_{P,Q},
\end{align*}
which proves the second part of the theorem.
\qed

\subsubsection{The symmetry space of the solutions}

In this subsection, we prove the final part of \Cref{thm:liouville}, i.e.\ $\sU(2)$-symmetry of the solutions to the generalized Liouville's equation. To complete the proof, we need to recall the connection between the Wronskian and ODEs.
We note that the space of $2 \times 2$ complex matrices $\C^{2 \times 2}$ acts on bipolynomials $\cP^2$ by $\Lambda (P,Q) (z) = \Lambda (P(z),Q(z))$ for every $z \in \C.$

\begin{proposition}
\label{prop:actiononwronksian}
    For every $\Lambda \in \C^{2 \times 2}$ and $(P,Q) \in \cP^2$, we have 
\begin{align*}
    W(\Lambda (P,Q)) = \det(\Lambda) W(P,Q).
\end{align*}
In particular, the action of the special linear group $\sSL(2) = \sSL(2;\C)$ on $\cP^2$ preserves the Wronskian.
\end{proposition}

\begin{proof}
    The proof follows by noting that 
\begin{align*}
    W(\Lambda(P,Q)) 
    = \det \left[ \Lambda \begin{pmatrix}
P' & P \\
Q' & Q
\end{pmatrix} \right]
= \det(\Lambda) \det  \begin{pmatrix}
P' & P \\
Q' & Q
\end{pmatrix},
\end{align*}
for every $\Lambda \in \C^{2 \times 2}$ and $ (P,Q) \in \cP^2.$
\end{proof}

\begin{proposition}
\label{prop:Wronskiantrick}
Let $P,Q \in \cP$ and $f = W(P,Q).$ Then, $P,Q$ are solutions $y \in \cP$ to the ODE 
\begin{align*}
f y'' - f' y' + W(P',Q') \, y =0.
\end{align*}
\end{proposition}
\begin{proof}
If $f = W(P,Q),$ then by taking a derivative we arrive at $f' = P'' Q - P Q''.$ Hence,
\begin{align*}
f P'' - f' P' =  (P' Q -P Q') P'' -(P'' Q - P Q'') P' =    W(Q',P')\, P.
\end{align*}
Likewise, we derive that $W(Q',P')\, Q = f Q'' - f' Q'$, which completes the proof.
\end{proof}

\begin{lemma}
\label{lem:Wronskian-ODEequivalence}
Let $f \in \cP $ be non-zero. Then, up to multiplication by constants, all the polynomial solutions $P,Q$ to  
\begin{align*}
W(P,Q) = f,
\end{align*}
are equivalent (up to constant factors) to two linearly independent polynomial solutions $y$ of the ODE
\begin{align*}
f y'' - f' y' + R y=0, 
\end{align*}
where $R \in \cP$.
\end{lemma}
\begin{proof}
Consider the first implication and assume that $(P,Q) \in \cP^2$ satisfies $W(P,Q) = f$.
Then, by \Cref{prop:Wronskiantrick}, we obtain that $P$ and $Q$ are solutions $y \in \cP$ of 
\begin{align*}
f y'' - f' y' + R y=0,
\end{align*}
where $R= W(P',Q').$ 

For the other implication, consider three polynomials $P,Q,R$ which satisfy 
\begin{align*}
f y'' - f' y' + R y=0, \qquad y \in \{P,Q\},
\end{align*}
where $P$ and $Q$ are linearly independent.
Then, by Abel's identity (away from the zeros of $f$ and $W(P,Q)$), we have 
\begin{align*}
\frac{W(P,Q)'}{W(P,Q)} = \frac{f'}{f} 
\quad \Leftrightarrow \quad 
P' Q -P Q' = C e^{\int \frac{f'}{f}} = C f,
\end{align*}
where $C \in \C^\times$.
Hence, replacing the pair $(P,Q)$ by $(C P,Q)$ completes the proof.
\end{proof}

\begin{lemma}
\label{lem:degreecondi}
Let $R,f \in \cP$ and $R \neq 0$. Assume that the ODE
\begin{align}
\label{eq:odeformat}
f y'' - f' y' +Ry =0
\end{align}
has at least two linearly independent polynomial solutions.
Then, all the solutions of \eqref{eq:odeformat}
are of the form $\lambda P + \mu Q$, where $\lambda,\mu \in \C$ and $P,Q$ are two linearly independent polynomial solutions to \eqref{eq:odeformat}, which satisfy
\begin{align*}
\deg(P) > \deg(Q) \quad \text{and} \quad 
\deg(P) +\deg(Q) = \deg(f) +1.
\end{align*}

\end{lemma}
\begin{proof} First, note that if $f=0$, then $0$ is the only solution $y$ of \eqref{eq:odeformat}. Hence, $f \neq 0$.  
Let $P,Q$ be two linearly independent complex polynomial solutions of \eqref{eq:odeformat}. If $\deg(P) = \deg(Q)$ and highest power terms in $P,Q$ are $a z^{\deg(P)},b z^{\deg(Q)}$, respectively, then we can replace $Q$ by $ Q - \frac{b}{a} P$ which satisfies \eqref{eq:odeformat} and $\deg\left(Q - \frac{b}{a} P\right) < \deg(P).$ Hence, without loss of generality, we can assume that $P$ and $Q$ satisfy also $\deg(Q) < \deg(P)$ and are monic. Now, by \Cref{lem:Wronskian-ODEequivalence}, up to a multiplication by constants, we have $W(P,Q) = f$. 
In conclusion, by \Cref{lem:degreeofwronskian} and in particular
\eqref{eq:degreeofwronskian-monic}, we have 
\begin{align*}
    \deg(P) + \deg(Q) = \deg(f)+1.
\end{align*}
Let $y$ be any complex polynomial solution to \eqref{eq:odeformat} and  $I$ be an open interval in $\R$ which does not contain any roots of $f$. By uniqueness of solutions to ODEs, we have $y = \lambda P +\mu Q$ in $I$, where $\lambda, \mu \in \C$. By unique continuation of analytic functions, we derive that $y = \lambda P +\mu Q$ in $\C$, which completes the proof.
\end{proof}

\begin{remark}[Bound on degrees]
\label{rem:algorithmWronsk}
\Cref{lem:degreeofwronskian,lem:Wronskian-ODEequivalence,lem:degreecondi} imply that $\deg(R) \leq \deg(f)-2$ and $\deg(y) \leq \deg(f)+1$ in \eqref{eq:odeformat}. Hence, by finite steps, one can derive all the possible polynomial solutions to the polynomial ODE \eqref{eq:odeformat}.
\end{remark}

Finally, to complete the proof of \Cref{thm:magneticstability}, it is sufficient to prove the following lemma.
\begin{lemma}
\label{lem:symmetryofamplitude}
Let $P_i,Q_i \in \cP$
for $i=1,2$.  Then, $$|P_1|^2 + |Q_1|^2 = |P_2|^2 + |Q_2|^2,$$ if and only if there exists a constant $\Lambda \in \sU(2)$ such that $(P_2,Q_2) = \Lambda (P_1,Q_1)$.
\end{lemma}
\begin{proof}
We start with the observation that $\sU(2)$ preserves the norm/amplitude $|(P,Q)| := \sqrt{|P|^2+|Q|^2}$ pointwise. Hence, if $(P_2,Q_2) = \Lambda (P_1,Q_1)$ for $\Lambda \in \sU(2),$ then 
\begin{align*}
|P_1|^2 + |Q_1|^2 = |P_2|^2 + |Q_2|^2.
\end{align*}

Now, for the other implication, assume that 
\begin{align}
\label{eq:normequality}
|P_1|^2 + |Q_1|^2 = |P_2|^2 + |Q_2|^2
\quad \text{on} \ \C.
\end{align}
By using transformations in $\sU(2)$, we can w.l.o.g. assume that $\deg(P_1) >\deg(Q_1)$, $\deg(P_2) >\deg(Q_2)$, and that the coefficients of the highest order are real and positive for $P_1,P_2$. By \eqref{eq:normequality}, we have $\deg(P_1)=\deg(P_2)$. Write $P_1 = \sum_{n=0}^{\deg(P_1)} c_n^1 z^n, P_2 =\sum_{n=0}^{\deg(P_1)} c_n^2 z^n, Q_1=\sum_{n=0}^{\deg(P_1)} d_n^1 z^n, Q_1=\sum_{n=0}^{\deg(P_1)} d_n^2 z^n.$
Then, for every non-negative integers $ m,n $, we have 
\begin{align*}
    \partial_z^m \partial_{\bar{z}}^n \left( |P_1|^2 + |Q_1|^2 \right)(0) = \partial_z^m \partial_{\bar{z}}^n \left( |P_2|^2 + |Q_2|^2 \right)(0),
\end{align*}
which implies that 
\begin{align}
\label{eq:crucialequalforcoeff}
    c_m^1 \overline{c_n^1} + d_m^1 \overline{d_n^1} = c_m^2 \overline{c_n^2} + d_m^2 \overline{d_n^2}.
\end{align}
Hence, by setting $m=n= \deg(P_1)$ and using $\deg(P_1) = \deg(P_2) > \max(\deg(Q_1),\deg(Q_2))$, 
\begin{align*}
     |c_n^1|^2 = |c_n^2|^2.
\end{align*}
Since the $c_n^1,c_n^2$ are positive real constants, we get $c_n^1 = c_n^2$. Now, by letting $n= \deg(P_1)$ and $m$ be arbitrary in \eqref{eq:crucialequalforcoeff}, we obtain
\begin{align*}
    c_m^1 \overline{c_n^1}  = c_m^2 \overline{c_n^2} = c_m^2 \overline{c_n^1}.
\end{align*}
Hence,
\begin{align*}
    c_m^1 = c_m^2,
\end{align*}
or equivalently 
\begin{align*}
    P_1 = P_2.
\end{align*}
By a similar argument, we get $Q_1=Q_2$ up to multiplication by an element in $\sU(1)$. This concludes the proof.
\end{proof}

\begin{remark}
\label{rmk:degreesless}
    Note that we may use the symmetry of \Cref{lem:symmetryofamplitude} together with \Cref{lem:Wronskian-ODEequivalence,lem:degreecondi} to obtain that $\psi_{P,Q} = \psi_{\tilde{P},\tilde{Q}}$ with $\deg \tilde{P} > \deg \tilde{Q}$.
    This is also exemplified below.
\end{remark}

\subsection{Some special solutions of interest}

In this part, we study several interesting special cases. First, we note that by \Cref{thm:liouville}, to find all the solutions to \eqref{eq:Liouville-psi} we need all the coprime pairs of polynomials $(P,Q) \in \cP^2$ such that $W(P,Q)= f$. A simple example of such pairs of polynomials is $P = C F$ and $Q = \frac{1}{C}$, where $C \neq 0$ is a constant and $F$ is a primitive of $f.$ To simplify the notations, we make the following definition.
\begin{definition}\label{def:wronskian-pair}
We call a pair $(P,Q) \in \cP^2$ a \keyword{coprime Wronskian pair} for $f \in \cP$ if  $P,Q$ are both non-zero, $W(P,Q) =f$, and $\gcd(P,Q)=1.$
\end{definition}
Now, we show that if $f \in \cP$ has only one root (with possible multiplicity) then all the coprime Wronskian pairs for $f$ are indeed generated by $1$ and a primitive of $f.$

\begin{lemma}
\label{lem:vortexwronskian}
Let $n$ be a non-negative integer, $a \in \C^{\times},$ and $z_0 \in \C$. Then, $(P,Q)$ is a coprime Wronskian pair for $a (z-z_0)^n$ if and only if
$$
P= \alpha_1 (z-z_0)^{n+1} + \beta_1 \quad \text{and} \quad Q=  \alpha_2 (z-z_0)^{n+1} + \beta_2,
$$
where 
\begin{align*}
\alpha_1 \beta_2 - \alpha_2 \beta_1 = \frac{a}{n+1}.
\end{align*}
\end{lemma}
\begin{proof}
Let $(P,Q)$ be a coprime Wronskian pair for $a (z-z_0)^n.$ Then, $P,Q$ are linearly independent.
Hence, by \Cref{prop:Wronskiantrick}, we have
\begin{align*}
P (Q'' P' -P'' Q' ) = a(z-z_0)^{n-1} ((z-z_0) P'' - n P'),
\end{align*}
and similarly for $Q$. Now, we consider two cases. First, we assume that 
$Q'' P' -P'' Q' =W(Q',P')$ is non-zero, but necessarily it has at most $n-2$ roots. Then, since 
$$
\deg((z-z_0) P'' - n P') \leq \deg(P)-1,
$$
and $P$ has $\deg(P)$ roots, 
we have $P(z_0)=0$. Likewise, $Q(z_0)=0$, which contradicts $\gcd(P,Q)=1.$ This concludes that $Q'' P' -P'' Q' =0$ and 
\begin{align*}
(z-z_0)  P''(z) - n P' =(z-z_0) Q''(z) - n Q'(z) =0,
\end{align*}
for every $z \in \C$.
Hence,
\begin{align*}
\left ( \frac{P'}{(z-z_0)^n} \right )' = \left ( \frac{Q'}{(z-z_0)^n} \right )' =0,
\end{align*}
for all $z \in \C \setminus \{z_0\}$. 
Then, by integrating both sides, we derive that 
\begin{align*}
P(z)= \alpha_1 (z-z_0)^{n+1} + \beta_1 \quad \text{and} \quad
Q(z) = \alpha_2 (z-z_0)^{n+1} + \beta_2.
\end{align*}
Moreover, by inserting $P,Q$ in $W(P,Q)=a (z-z_0)^n$, we obtain
\begin{align*}
\alpha_1 \beta_2 - \alpha_2 \beta_1 = \frac{a}{n+1}.
\end{align*}
Now, we show that $\gcd(P,Q)=1$ if $\alpha_1 \beta_2 - \alpha_2 \beta_1 = \frac{a}{n+1}$ holds. Assume that $P(w_0)=Q(w_0)=0$ for some $w_0 \in \C.$ Then, 
\begin{align*}
\beta_1 =- \alpha_1 w_0^{n+1} \quad \text{and} \quad \beta_2 = -\alpha_2 w_0^{n+1}.
\end{align*}
Hence, $\alpha_1 \beta_2 - \alpha_2 \beta_1 = 0$ which is in contradiction with $|a|>0$. Thus, $\gcd(P,Q)=1$ which completes the proof.
\end{proof}

\begin{corollary}[Single-root solutions]
\label{cor:radiallysymmetric}
Let $z_0 \in \C$, $n$ be a non-negative integer, and $K$ be a positive constant.  Then, all the weak solutions $\psi \in L^1_{\loc}(\R^2;\R)$ of $-\Delta \psi = K |z-z_0|^{2n} e^{\psi}$ in $\C \cong \R^2$, which satisfy the condition $\int_{\R^2} |z-z_0|^{2n} e^{\psi}<\infty$, are of the form 
\begin{align}
\label{eq:vortexliouville}
\psi(z) = \log\left( \frac{8}{\left(| a (z-z_0)^{n+1} + c |^2 + \frac{K}{|a|^2 (n+1)^2}\right)^2} \right),
\end{align}
for $z \in \C$, where $c\in \C$ and $a \in \C^{\times}$.
Furthermore, if $z_0 =0$ and $\psi$ is radially symmetric, then
\begin{align*}
\psi(z) = \log\left( \frac{8}{\left(|a|^2 |z|^{2n+2} + \frac{K}{|a|^2 (n+1)^2}\right)^2} \right),
\end{align*}
for every $z \in \C,$ where $a \in \C^{\times}$. 
\end{corollary}
\begin{remark}\label{rmk:Liouvillesolsymm}
    The latter, radial solution was already known from \cite{JacPi-90b}.
    The possibility of the non-radial solutions $c \neq 0$, and thus the complete solution for a single root of arbitrary multiplicity, was discussed in \cite{ChaKie-94,PraTar-01}.
\end{remark}
\begin{proof}
By \Cref{thm:liouville}, we have all the solutions $\psi \in L^1_{\loc}(\R^2)$ of $-\Delta \psi = K|z-z_0|^{2n} e^{\psi}$ in $\R^2$, where $\int_{\R^2} |z- z_0|^{2n} e^{\psi}<\infty$, are of the form \begin{align*}
\psi = \log(8) - 2\log(|P|^2 +|Q|^2),
\end{align*}
where $(P,Q)$ is a coprime Wronskian pair for $\sqrt{K} z^n$. Hence, by \Cref{lem:vortexwronskian}, we derive
\begin{align*}
P(z)= \alpha_1 (z-z_0)^{n+1} + \beta_1, \quad
Q(z) = \alpha_2 (z-z_0)^{n+1} + \beta_2,
\end{align*}
where 
\begin{align}
\label{eq:determincondi}
\alpha_1 \beta_2 - \alpha_2 \beta_1 = \frac{\sqrt{K}}{n+1}.
\end{align}
In conclusion, all the solutions are of the form
\begin{align*}
\psi(z) = \log\left( \frac{8}{(| \alpha_1 (z- z_0)^{n+1} + \beta_1 |^2 + | \alpha_2 (z-z_0)^{n+1} + \beta_2|^2)^2} \right),
\end{align*}
where the constants $\alpha_1,\alpha_2,\beta_1,\beta_2$ satisfy \eqref{eq:determincondi}. Now, define the matrix
(cp. \Cref{rmk:degreesless})
\begin{align*}
    \Lambda := \frac{1}{\sqrt{|\alpha_1| ^2 + |\alpha_2|^2}} \begin{pmatrix}
    \alpha_2 & -\alpha_1 \\
    \overline{\alpha_1} & \overline{\alpha_2}
    \end{pmatrix},
\end{align*}
for every $z \in \C.$ Then, $\Lambda \in \sSU(2)$ and, by the symmetry of solutions in Theorem \ref{thm:liouville} and action of $\Lambda$ on $(P,Q)$, we obtain \eqref{eq:vortexliouville}.  Finally, if $z_0=0$ and we focus on the radially symmetric solutions $\psi$, then we have the extra condition $c=0$. Hence,
\begin{align*}
\psi(z) = \log\left( \frac{8}{\left(|a|^2 |z|^{2n+2} + \frac{K}{ |a|^2 (n+1)^2 } \right)^2} \right),
\end{align*}
for every $z \in \C$.
\end{proof}

Now, we consider the case of the Wronskian having degree $2$.
\begin{lemma}
\label{lem:coprimewronsdegree2}
Let $a \in \C^{\times}, b,c \in \C$. Then, $(P,Q)$ is a coprime Wronskian pair for the polynomial $f(z) = a z^2 + bz +c$ if and only if $(P,Q) = \Lambda (P_0,Q_0)$ where $\Lambda \in \sSL(2)$ and one of the following two cases occur:
\begin{enumerate}[label=(\roman*)]
\item The constants $a \in \C^\times$, $b,c \in \C$ are arbitrary and
$$P_0 =  \frac{a}{3} z^3 + \frac{b}{2} z^2 + c z \quad \text{and} \quad Q_0 = 1,$$ 

\item It holds that $c \neq \frac{b^2}{4 a}$ 
($f$ has two distinct roots) and
$$P_0 =  z^2  -\frac{c}{a} \quad \text{and} \quad Q_0 = a z + \frac{b}{2}.$$
\end{enumerate}
\end{lemma}

\begin{proof}
   It is easy to check that if $(P,Q) = \Lambda (P_0,Q_0)$ where $\Lambda \in \sSL(2)$ and one of $(i)$ or $(ii)$ holds, then $(P,Q)$ is a coprime Wronskian pair for $a z^2 + bz + c.$ Now, assume that $(P,Q)$ is a Wronskian pair for $a z^2 + bz + c.$ Then,
   by \Cref{lem:Wronskian-ODEequivalence} and \Cref{lem:degreecondi}, we have $(P,Q)= \Lambda (P_0,Q_0)$, where $\Lambda \in \sSL(2)$ and  $P_0,Q_0$ are polynomials, where $\deg(P_0) >\deg(Q_0)$, $\deg(P_0)+ \deg(Q_0) = 3$, and
\begin{align*}
W(P_0,Q_0) = a z^2 + bz  + c.
\end{align*}
Hence, we have either $\deg(P_0)=3$, $\deg(Q_0)=0$ or $\deg(P_0)=2$, $\deg(Q_0)=1$.

Now, let us assume the first case that $\deg(P_0)=3$ and $\deg(Q_0)=0$. 
Then, $Q_0 = C \neq 0$ is a constant and
\begin{align*}
a z^2 + bz  + c= P_0'Q_0 - P_0 Q_0' = P_0' Q_0.
\end{align*}
Moreover, by the freedom of choice of $\Lambda \in \sSL(2)$, we can take any $\Tilde{\Lambda} (P_0,Q_0)$ instead of $(P_0,Q_0)$ such that
\begin{align*}
   \Tilde{\Lambda} = \begin{pmatrix}
        \lambda & \mu \\
        0 & \frac{1}{\lambda},
    \end{pmatrix}
\end{align*}
where $\lambda \in \C^{\times}, \mu \in \C.$ In particular,
we can assume that $Q_0=1$ and $P_0(0)=0$.
Hence, $C= 1$ and
\begin{align*}
P_0 = \frac{a}{3} z^3 + \frac{b}{2} z^2 + c z.
\end{align*}
Note that in this case, we have also $\gcd(P_0,Q_0)=1.$

In the second case, where $\deg(P_0)=2$ and $\deg(Q_0)=1$, again by freedom of choice of $\Tilde{\Lambda} (P_0,Q_0)$, where $\Tilde{\Lambda} \in \sSL(2)$, we can assume that $P_0$ is a monic polynomial and the coefficient of $P_0$ of power $\deg(Q_0)=1$ is zero. In conclusion, $P_0 =  z^2 + a_1$, $Q_0 = a_2 z + b_2$. Then,
\begin{align*}
a z^2 + bz  + c= P_0' Q_0 - P_0 Q_0' = a_2 z^2+ 2b_2 z-a_1 a_2,
\end{align*}
which implies that
\begin{align*}
a_2 =a,\quad b_2 = \frac{b}{2} \quad \text{and} \quad a_1 = -\frac{c}{a}.
\end{align*}
Moreover, in the latter case, $P_0$ and $Q_0$ have a common root if and only if 
\begin{align*}
0= P_0\left( -\frac{b}{2a} \right) = \frac{1}{a} \left(\frac{b^2}{4a} -c \right).
\end{align*}
Thus, $P_0, Q_0$ are coprime if and only if $c \neq \frac{b^2}{4a}$.
\end{proof}

The next corollary determines all the coprime Wronskian pairs for polynomials of degree two. This is partly motivated by a question posed in \cite{BriHouLei-05}.

\begin{corollary}[Degree-two solutions]
\label{cor:degreetwopoly}
Let $b,c \in \C$ and $a \in \C^{\times}$. Then, all the weak solutions $\psi \in L^{1}_{\loc}(\R^2;\R)$ of $$-\Delta \psi = |a z^2 + b z + c|^{2} e^{\psi},$$ in $\R^2$, where $\int_{\R^2} |a z^2 + b z + c|^{2} e^{\psi}<\infty$, are of the form 
\begin{align*}
\psi = \log(8) - 2 \log(|P|^2 + |Q|^2),
\end{align*}  
where $P = \lambda P_0 +\mu Q_0$, $Q = \frac{1}{\lambda} P_0 $ for some $\mu \in \C, \lambda \in \C^{\times},$
and polynomials $P_0,Q_0$ are of one the following forms:
\begin{enumerate}[label=(\roman*)]
\item The constants $a \in \C^\times$, $b,c \in \C$ are arbitrary and
$$P_0 =  \frac{a}{3} z^3 + \frac{b}{2} z^2 + c z \quad \text{and} \quad Q_0 = 1,$$ 

\item It holds that $c \neq \frac{b^2}{4 a}$ and
$$P_0 =  z^2  -\frac{c}{a} \quad \text{and} \quad Q_0 = a z + \frac{b}{2}.$$
\end{enumerate}
\end{corollary}
\begin{proof}
By \Cref{thm:liouville} and \Cref{lem:degreecondi}, we derive that all the weak solutions $\psi \in L^1_{\loc}(\R^2)$ of  $$-\Delta \psi = |a z^2 + b z + c|^{2} e^{\psi},$$ where $\int_{\R^2} |a z^2 + b z + c|^{2} e^{\psi}<\infty$, are of the form 
\begin{align*}
\psi = \log(8) - 2 \log(|P|^2 + |Q|^2),
\end{align*}
where $(P,Q)$ is a coprime Wronskian pair for $a z^2 + bz  + c$.
Hence, by \Cref{lem:coprimewronsdegree2}, we have $(P,Q)= \Lambda (P_0,Q_0)$, where $\Lambda \in \sSL(2)$ and  $P_0,Q_0$ are polynomials, where either $(i)$ or $(ii)$ holds. Moreover, we can use the symmetry of solutions in Theorem \ref{thm:liouville} and QR decomposition of the matrix $\Lambda,$ to reduce to the case that $\Lambda$ is upper triangular, which completes the proof.
\end{proof}

\begin{remark}
In \Cref{cor:degreetwopoly}, by change of variable $\varrho =|a z^2 + bz + c|^2 e^{\psi}$, we have classified all the solutions to 
\begin{align*}
- \Delta \log(\varrho) = \varrho - \sum_{i=1}^2 4 \pi \delta_{z_i} 
\end{align*}
in $\R^2$ with $\int_{\R^2} \varrho < \infty,$ where $z_i \in \C$
(compare \cite{BriHouLei-05}). Furthermore, if $c \neq \frac{b^2}{4a}$, or $f(z) := a z^2 + bz + c$ has two distinct roots, then there exist coprime Wronskian pairs for $f$, which are not generated by $1$ and a primitive of $f$.
We have not found these observations elsewhere in the mathematical literature.
(\emph{Note added:} After submitting our first version of the manuscript, we became aware of \cite[Lemma~6.1]{TT-24}, which discusses one of the special cases above.)
\end{remark}

\newpage 

\section{Magnetic stability}\label{sec:magnetic}

In this section, we prove \Cref{thm:magneticstability} by applying \Cref{thm:liouville}. We derive the result in several steps as follows.
\begin{itemize}
\item We set our notation and prove some preliminary results on magnetic energy inequalities.
\item  We derive the regularity and Euler--Lagrange equation for the minimizers of \eqref{eq:defgamma}.
\item  We show the supersymmetric factorization of $\cE_{\beta,2 \pi \beta}[u]$ and asymptotic behavior of the superpotential $\Phi[|u|^2]$ for $u \in H^1(\R^2).$
\item  For every $\beta \in 2 \N$, we prove the sufficient form of $u_{P,Q}$ for the minimizer of \eqref{eq:defgamma}, where $P,Q \in \cP$ are linearly independent with $\gcd(P,Q)=1$ and\\ $\max(\deg(P),\deg(Q))= \frac{\beta}{2}.$ 
\item  By using the monotonicity of $\frac{\gamma_{\ast}(\beta)}{\beta}$ for every $\beta>0$, we deduce that $\gamma_{\ast}(\beta)=2\pi \beta$ for every $\beta \geq 2$. Furthermore, the necessity of $\beta \in 2 \N$ with the general form of minimizers of $\gamma_{\ast}(\beta)$ are shown for \eqref{eq:defgamma} if $\beta \geq 2.$ 
\item Afterwards, we use bounds on $\gamma_{\ast}(\beta)$, local Lipschitz regularity of $\gamma_{\ast}(\beta),$ and the concentration-compactness lemma of Lions to prove that $\gamma_{\ast}(\beta)>\CLGN$ for $\beta>0$, $\gamma_{\ast}(\beta) >2\pi \beta$ for $\beta <2$, and that for small enough $\beta >0$, there exists a minimizer for \eqref{eq:defgamma}. 
\item  Finally, by using the symmetry of the generalized Liouville's solution \eqref{eq:Liouville-solution}, we prove the $\R^+ \times \sSU(2)$-symmetry of minimizers of \eqref{eq:defgamma}.
\end{itemize}

\subsection{Preliminaries on magnetic inequalities}\label{sec:mag-prelim}

In the following, we derive some properties of the magnetic self-energy \eqref{eq:AF-func}.

\begin{lemma}\label{lem:basicinequalities}
Let $u \in H^1(\R^2)$. We have the following inequalities.
\begin{enumerate}[label=(\roman*)]
\item {\bf Diamagnetic inequality.} For every $\beta \in \R$,
\begin{equation}\label{eq:AF-diamagnetic}
\int_{\R^2} \bigl|\nabla |u|\bigr|^2 \leq \int_{\R^2} \left|(\nabla + {\rm i}\beta \bA\left[|u|^2\right]) u\right|^2.
\end{equation}

\item {\bf Hardy-type inequality.} There exists a universal constant 
$0 < \CH \le 3/2$  
such that
\begin{equation}\label{eq:AF-H1}
\int_{\R^2} \left|\bA\left[|u|^2\right]u\right|^2
\le \CH \left(\int_{\R^2} |u|^2\right)^2 \left(\int_{\R^2} \bigl|\nabla|u|\bigr|^2\right).
\end{equation}

\item {\bf Gagliardo--Nirenberg inequality.} For every $p \geq 2$, there exists a universal constant $C = C(p) > 0$ such that
\begin{align}\label{eq:gagliardonirenberg}
C \int_{\R^2} |u|^p \leq \left(\int_{\R^2} |u|^2\right) \left(\int_{\R^2} |\nabla u|^2\right)^{\frac{p-2}{2}}.
\end{align}
\end{enumerate}
\end{lemma}

We refer to \cite[Theorem~7.21]{LieLos-02} for the proof of the diamagnetic inequality \eqref{eq:AF-diamagnetic} and \cite[Lemma~3.4]{LunRou-15} for the proof of the Hardy inequality \eqref{eq:AF-H1} involving the circumradius $\cR$ in \eqref{eq:MM-curvature}. It is worth noting that the constant in \eqref{eq:AF-H1} was given by $\CH=\frac{3}{2}$ in \cite{LunRou-15} but it might not be optimal (see \cite[Remark~3.7(ii)]{HofLapTid-08}).
Finally, the Gagliardo--Nirenberg inequality \eqref{eq:gagliardonirenberg} is a consequence of the Sobolev embedding \cite{Gagliardo-59,Nirenberg-59} 
(see also, e.g., \cite[Theorem~8.5]{LieLos-02}).

\begin{proposition}\label{prop:LGN-bound}
For every $\beta \in \R$, we have
\begin{align*}
\gamma_{\ast}(\beta) \geq \gamma_{\ast}(0) >0.
\end{align*}
\end{proposition}

\begin{proof}
By the diamagnetic inequality \eqref{eq:AF-diamagnetic} and the LGN inequality \eqref{eq:LGN}, we have
$$
\frac{\cE_{\beta}[u]}{\int_{\R^2} |u|^4} \geq \frac{\int_{\R^2} |\nabla |u||^2}{\int_{\R^2} |u|^4} \geq \CLGN = \gamma_{\ast}(0) 
$$
for every $u \in H^1(\R^2)$ and $\beta \in \R$. This yields the desired result. 
\end{proof}

\begin{lemma}
For any $u \in H^1(\R^2)$, we have
\begin{equation}\label{lem:A.Jbound}
 \int_{\R^2} \bigl| \bA\left[|u|^2\right] \bigr| \, | \bJ[u] |
\le \sqrt{\frac{3}{2}} \left(\int_{\R^2} |u|^2\right) \left(\int_{\R^2} |\nabla u|^2\right).
\end{equation}
Moreover, for every $\beta \geq 0$,
\begin{equation}\label{lem:boundgradient}
\int_{\R^2} |\nabla u|^2 \le \left( 1 + \sqrt{\frac{3}{2}}\beta \int_{\R^2}|u|^2 \right)^2 \cE_{\beta}[u].
\end{equation}
\end{lemma}
\begin{proof}
Note that by definition \eqref{eq:current}, we have $\bJ[u] = \Im(\bar{u} \, \nabla u) $. Hence, by the Cauchy--Schwarz inequality and the Hardy inequality \eqref{eq:AF-H1}, as well as the diamagnetic inequality \eqref{eq:AF-diamagnetic} for $\beta=0$, we have
\begin{align*}
\int_{\R^2} \bigl| \bA\left[|u|^2\right] \bigr|\,  | \bJ[u] |  
\leq  \left(\int_{\R^2} |\bA\left[|u|^2\right] u|^2\right)^{\frac{1}{2}} \left(\int_{\R^2} |\nabla u|^2\right)^{\frac{1}{2}}
\leq \sqrt{\frac{3}{2}} \left(\int_{\R^2} |u|^2\right) \left(\int_{\R^2} |\nabla u|^2\right),
\end{align*}
which proves \eqref{lem:A.Jbound}.

On the other hand, \eqref{lem:boundgradient} is deduced from the diamagnetic inequality \eqref{eq:AF-diamagnetic}, the Hardy inequality \eqref{eq:AF-H1}, and the following triangle inequality:
$$
\norm{\nabla u}_{L^2} 
\le \norm{\nabla u + {\rm i}\beta \bA\left[|u|^2\right]u}_{L^2}
+ \beta \norm{\bA\left[|u|^2\right]u}_{L^2}.
$$
\end{proof}

\begin{proposition}\label{prop:youngconv}
Let $v \in L^p(\R^2)$ for $1 < p < 2$ and $r = \frac{2p}{2-p}$. Then, $\bA[v] \in L^r(\R^2)$ and 
$$
\norm{\bA[v]}_{L^r} \leq C_p \norm{v}_{L^p},
$$
where $C_p$ is a universal constant depending on $p$.
In particular, $\bA[|u|^2] \in L^r(\R^2)$ for all $2 < r < \infty$ and for $u \in H^1(\R^2)$.
\end{proposition}

\begin{proof}
Note that $\frac{1}{|\bx|} \in L^{2,w}(\R^2)$ (weak $L^2$).
Hence, by weak-type Young's convolution inequality (see, e.g., \cite[Section~4.3]{LieLos-01}),
we have
$$
\norm{\bA[v]}_{L^r} \leq C_p \norm{ \frac{1}{|\bx|} }_{L^{2,w}} \norm{v}_{L^p}.
$$
The second statement follows by the above inequality, together with Sobolev embedding $H^1(\R^2) \hookrightarrow L^p(\R^2)$ for all $2 \leq p <\infty$.
\end{proof}

\begin{lemma}
\label{lem:productlemma}
Let $p \geq 2$, $n,k \in \N$, and
$f_i \in W^{k,p}(\R^2)$ for every $1 \leq i \leq n$. 
Then, $\prod_{i=1}^n f_i \in W^{k,q}_{\loc}(\R^2)$ for every $1\leq q <p$.
\end{lemma}
\begin{proof}
For each $1 \leq i \leq n$, by H\"older's inequality, $f_i \in W^{k,s}_{\loc}(\R^2)$ for every $1\leq s\leq p$. By Sobolev's inequality (see e.g., \cite[Theorem 8.8]{LieLos-01}), $f_i \in W^{k-1,r}_{\loc}(\R^2)$ for every $1 \leq s <2$ and $s\leq r \leq \frac{2s}{2-s}$. In other words, $f_i \in W^{k-1,r}_{\loc}(\R^2)$ for every $1 \leq r< \infty$. Now, we prove the lemma by induction on $n$. For $n=1$, the Lemma follows immediately from H\"older's inequality. Now, if the lemma holds for $n$, we have $\prod_{i=1}^{n} f_i \in W^{k-1,q}_{\loc}(\R^2)$ for every $1 \leq q <p$.
Hence, by the Leibniz formula and H\"older's inequality, we derive
\begin{align*}
\int_{B} \left|\nabla^k  \prod_{i=1}^{n+1} f_i\,\right|^q \leq{} & (k+1)^{q+1} \int_{B}  \left|\sum_{j=0}^{k-1} \left | \nabla^{j} \prod_{i=1}^n f_i \right | \, \right|^q \left|\sum_{j=0}^{k}  | \nabla^{j}  f_{n+1}  | \, \right|^q \\
& + (k+1)^{q+1}\left|\sum_{j=0}^{k} \left | \nabla^{j} \prod_{i=1}^n f_i \right | \, \right|^q \left|\sum_{j=0}^{k-1}  | \nabla^{j}  f_{n+1}  | \, \right|^q  \\
\leq{} & (k+1)^{q+1} \left(\int_{B}  \left|\sum_{j=0}^{k-1} \left | \nabla^{j} \prod_{i=1}^n f_i \right | \, \right|^{\frac{pq}{p-q}}\right)^{\frac{p-q}{p}}  \left(\int_{B}\left|\sum_{j=0}^{k}  | \nabla^{j}  f_{n+1}  | \, \right|^p\right)^{\frac{q}{p}} \\
& + (k+1)^{q+1} \left(\int_{B} \left|\sum_{j=0}^{k} \left | \nabla^{j} \prod_{i=1}^n f_i \right |  \, \right|^p \right)^{\frac{q}{p}} \left(\int_{B}  \left|\sum_{j=0}^{k-1} | \nabla^{j} f_{n+1} | \, \right|^{\frac{pq}{p-q}}\right)^{\frac{p-q}{p}} \\
<{} & \infty,
\end{align*}
for every $1 \leq q<p$ and ball $B \subset \R^2$, where $\nabla^j$ denote the corresponding matrix of $j$th derivatives. This completes the proof.
\end{proof}

In order to make use of \Cref{lem:productlemma}, we need the following Lemma which relies on the regularity of the superpotential functional $\Phi$ defined in \Cref{def:superpotentialdef}.

\begin{lemma}\label{lem:superpot-relations}
    For every $f \in L^1(\R^2)$, we have 
    $\Phi[f] \in L^1_{\loc}(\R^2)$, 
    $\bA[f] \in L^1_{\loc}(\R^2;\R^2),$ and 
    $\bA[f] = \nabla^{\perp} \Phi[f]$ in the distribution sense. Moreover, $\Delta \Phi[f] = 2\pi f$ in the distribution sense in $\R^2$ for every $f \in L^1(\R^2)$.
\end{lemma}
\begin{proof}
    Let $f \in L^1(\R^2)$ and $R>1.$ Then,
\begin{align*}
   \frac{1}{4} \leq  \frac{|\by|-|\bx|}{|\by|+1} \leq \frac{|\bx-\by|}{|\by|+1} \leq \frac{|\bx| +|\by|}{|\by|+1} \leq \frac{3}{2}
\end{align*}
for every $\bx \in B(0,R)$, $\by \in \R^2 \setminus B(0,2R).$ 
Hence,
\begin{align*}
       \int_{B(0,R)} |\Phi[f](\bx)| \,\dd \bx 
       ={} & \int_{B(0,R)} \biggl|  \int_{\R^2} (\log(|\bx-\by|) - \log(|\by|+1)) f(\by) \dd \by \biggr| \,\dd \bx \\ 
       \leq{} & \int_{B(0,R)} \biggl|  \int_{\R^2 \setminus B(0,2R)} (\log(|\bx-\by|) - \log(|\by|+1)) f(\by) \dd \by \biggr| \,\dd \bx \\ 
       & + \int_{B(0,R)} \biggl|  \int_{ B(0,2R)} (\log(|\bx-\by|) - \log(|\by|+1)) f(\by) \dd \by \biggr| \,\dd \bx \\ 
       \leq{} & (\log(4) + \log(2R +1)) |B(0,R)| \, \|f\|_{L^1}  \\ 
       &+ \|f\|_{L^1(B(0,2R))} \int_{B(0,3R)} |\log|\bx|| \dd \bx < \infty,
    \end{align*}
where we used Fubini's theorem in the last inequality.
Moreover,
\begin{align*}
     \int_{B(0,R)} |\bA[f](\bx)| \,\dd \bx 
     & \leq \int_{B(0,R)}  \int_{\R^2} \frac{|f(\by)|}{|\bx-\by|} \dd \by \dd \bx \\ 
     & = \int_{B(0,R)} \left(\int_{B(0,2R)} \frac{|f(\by)|}{|\bx-\by|} \dd \by + \int_{\R^2 \setminus B(0,2R)} \frac{|f(\by)|}{|\bx-\by|} \dd \by  \right) \dd \bx \\ 
     & \leq 6 R \pi \, \|f\|_{L^1(B(0,2R))} + \frac{|B(0,R)| \, \|f\|_{L^1}}{R} < \infty.
\end{align*}
Now, for every $\psi \in C_c^{\infty}(B(0,R))$, we have
\begin{equation}
\label{eq:fubini}
\begin{aligned}
    \int_{\R^2} \Phi[f](\bx) \, \frac{\partial}{\partial x_1} \psi(\bx) \dd \bx 
    & =  \int_{\R^2}   \int_{\R^2} (\log(|\bx-\by|) - \log(|\by|+1)) f(\by)   \, \frac{\partial}{\partial x_1} \psi(\bx) \dd \by \dd \bx
    \\ & = \int_{\R^2}  f(\by)  \int_{\R^2} (\log(|\bx-\by|) - \log(|\by|+1))   \, \frac{\partial}{\partial x_1} \psi(\bx) \dd \bx \dd \by,
\end{aligned}
\end{equation}
where we used Fibini's theorem in the second equality.
Now, let $\eta_{\by,\varepsilon} \in C_c^{\infty}(B(\by,2\varepsilon))$ be a sequence of functions for $\by \in \R^2,\varepsilon>0$, such that $0 \leq \eta_{\by,\varepsilon} \leq 1 $, $ \eta_{\by,\varepsilon}|_{B(\by,\varepsilon)}=1$, and $|\nabla \eta_{\by,\varepsilon}| \leq \frac{C}{\varepsilon}$ where $C>0$ is a constant. Then,
\begin{align*}
   &\lim_{\varepsilon \to 0} \int_{\R^2} (\log(|\bx-\by|) - \log(|\by|+1))   \, \frac{\partial}{\partial x_1} (\eta_{\by,\varepsilon}(\bx) \psi(\bx)) \,\dd \bx \\
   &= \int_{\R^2} (\log(|\bx-\by|) - \log(|\by|+1))   \, \frac{\partial}{\partial x_1}  \psi(\bx) \,\dd \bx
\\ & + \lim_{\varepsilon \to 0}\int_{\R^2} (\log(|\bx-\by|) - \log(|\by|+1)) \, \psi(\bx)  \, \frac{\partial}{\partial x_1} \eta_{\by,\varepsilon}(\bx)  \,\dd \bx,
\end{align*}
and 
\begin{align*}
  & \biggl| \int_{\R^2} (\log(|\bx-\by|) - \log(|\by|+1)) \, \psi(\bx)   \, \frac{\partial}{\partial x_1} \eta_{\by,\varepsilon}(\bx)  \,\dd \bx \biggr|
\\&\leq  \frac{C \, \|\psi\|_{L^{\infty}}}{\varepsilon}\int_{B(\by,2\varepsilon)} |\log(|\bx-\by|)| \,\dd \bx + \frac{C \, \|\psi\|_{L^{\infty}} \, |B(\by,2\varepsilon)| \, \log(|\by|+1)}{\varepsilon}
\\ & \leq 2\pi C \, \|\psi\|_{L^{\infty}} \, \varepsilon (2|\log(2\varepsilon)| + 1) + \frac{C \, \|\psi\|_{L^{\infty}} \, |B(\by,2\varepsilon)| \, \log(|\by|+1)}{\varepsilon}.
\end{align*}
Hence, 
\begin{align*}
   & \lim_{\varepsilon \to 0} \int_{\R^2} (\log(|\bx-\by|) - \log(|\by|+1))   \, \frac{\partial}{\partial x_1} (\eta_{\by,\varepsilon}(\bx) \psi(\bx)) \dd \bx \\&= \int_{\R^2} (\log(|\bx-\by|) - \log(|\by|+1))   \, \frac{\partial}{\partial x_1}  \psi(\bx) \dd \bx.
\end{align*}
In conclusion, by \eqref{eq:fubini}, integration by parts, and Lebesgue's dominated convergence, we derive that
\begin{align*}
      \int_{\R^2} &\Phi[f](\bx) \, \frac{\partial}{\partial x_1} \psi(\bx) \,\dd \bx = \int_{\R^2}  f(\by)  \int_{\R^2} (\log(|\bx-\by|) - \log(|\by|+1))   \, \frac{\partial}{\partial x_1} \psi(\bx) \,\dd \bx \dd \by \\
       &=  \int_{\R^2}  f(\by) \lim_{\varepsilon \to 0} \int_{\R^2} (\log(|\bx-\by|) - \log(|\by|+1))   \, \frac{\partial}{\partial x_1} (\eta_{\by,\varepsilon}(\bx) \psi(\bx)) \,\dd \bx \dd \by
       \\ &=  -\int_{\R^2}  f(\by)  \int_{\R^2} \frac{x_1 - y_1}{|\bx-\by|^2} \, \psi(\bx) \,\dd \bx \dd \by =-\int_{\R^2} A_2[f](\bx) \psi(\bx) \,\dd \bx,
\end{align*}
which proves that $ \frac{\partial}{\partial x_1}\Phi[f] =A_2[f] $ in the distribution sense. Here, we used Fubini's theorem in the last equality.
Likewise, we have $A_1[f] =-\frac{\partial}{\partial x_2}\Phi[f]  $
in the distribution sense, which completes the proof. Finally, 
\begin{align*}
    \Delta \Phi[f] = \nabla^{\perp} \cdot \nabla^{\perp} \Phi[f] = \nabla^{\perp} \cdot \bA[f] = 2 \pi f,
\end{align*}
in distribution sense in $\R^2$.
\end{proof}

\begin{lemma}[Regularity of superpotential]
\label{lem:regularityofgauge}
Let $f \in L^{1} \cap W^{k,p}_{\loc}(\R^2) $ for some non-negative integer $k$ and $p>1$.  Then, $\Phi[f] \in L^{\infty}_{\loc} \cap W^{k+2,p}_{\loc}(\R^2)$,
and if $f \in C^\infty$ then $\Phi[f] \in C^\infty$.
\end{lemma}
\begin{proof}
We first show that $\Phi[f] \in L^{\infty}_{\loc}(\R^2)$. Indeed, for every $\bx \in \R^2$, we have  
\begin{align*}
|\Phi[f](\bx)| \leq{} \int_{\R^2} \bigl|\log|\bx-\by| - \log(|\by|+1) \bigr| \, |f(\by)| \,\dd\by. 
\end{align*}
On the one hand, for $\by \in B(0, 2|\bx|+1)$, we use H\"older's inequality to obtain
\begin{align*}
& \int_{ B(0, 2|\bx| +1)} \bigl|\log|\bx-\by| - \log(|\by|+1) \bigr| \, |f(\by)| \,\dd\by \\ 
\leq & \int_{ B(0, 2|\bx| +1)} \left( \bigl|\log|\bx-\by|\bigr| |f(\by)| + \bigl|\log(|\by|+1)\bigr| |f(\by)| \right) \dd\by \\ 
\leq &  \left(\int_{ B(0, 2|\bx| +1)} |\log|\by|  \, |^{\frac{p}{p-1}} \dd  \by \right)^{\frac{p-1}{p}} \left(\int_{B(0, 2|\bx| +1)} |f(\by)|^p \dd  \by\right)^{\frac{1}{p}} \\ 
&+  \left(\int_{ B(0, 2|\bx| +1)}|\, \log(|\by|+1) |^{\frac{p}{p-1}} \dd  \by \, \right)^{\frac{p-1}{p}}\left(\int_{B(0, 2|\bx| +1)} |f(\by)|^p \dd  \by\right)^{\frac{1}{p}}
\end{align*}
On the other hand, for $\by \in \R^2 \setminus  B(0, 2|\bx| +1)$, we have
\begin{align*}
\frac{1}{2} \leq \frac{|\by|-|\bx|}{|\by|+1} \leq \frac{|\bx -\by|}{|\by|+1} \leq \frac{|\bx| +|\by|}{|\by|+1} \leq \frac{3}{2},
\end{align*}
which implies that 
\begin{align*}
\int_{\R^2 \setminus B(0, 2|\bx| +1)} |\log|\bx-\by| - \log(|\by|+1) | \, |f(\by)| \dd \by &= \int_{\R^2 \setminus B(0, 2|\bx| +1)} \left|\log \frac{|\bx-\by| }{|\by|+1} \right| \, |f(\by)| \dd  \by \\
& \leq 2\log(2)\int_{\R^2 \setminus B(0, 2|\bx| +1)}  |f(\by)| \dd  \by. 
\end{align*}
Hence, $\Phi[f] \in L^{\infty}_{\loc}(\R^2)$. Now, by \Cref{lem:superpot-relations}, we derive $\Delta \Phi[f] =2 \pi f$. Since $\Phi[f] \in L^{\infty}_{\loc}(\R^2)$ and $f \in W^{k,p}_{\loc}(\R^2)$, we derive that $ \Phi[f] \in W^{k+2,p}_{\loc}(\R^2)$, by a standard elliptic regularity estimate; see \cite[Chapter 9]{GilTru-01}.
\end{proof}

\subsection{Variational equation}\label{sec:mag-variation}

In this subsection, we show the regularity of the minimizers for $\gamma_{\ast}(\beta)$.
We need the following definition and lemma.

\begin{definition}\label{def:A-general}
    For every $\bF \in L^1(\R^2;\C^2)$ and $\bx \in \R^2$, we define
\begin{align*}
   \bA * [\bF](\bx) := \lim_{\varepsilon \to 0}\int_{\R^2 \setminus B(\bx, \varepsilon)} \frac{(\bx-\by)^{\perp}}{|\bx-\by|^2} \cdot \bF(\by) \,\dd \by.
\end{align*}
\end{definition}

\begin{lemma}
\label{lem:regularityofmagneticgauge}
 Let $f \in L^{1} \cap W^{k,p}_{\loc}(\R^2)$ and $\bF \in  L^{1} \cap W^{k,p}_{\loc}(\R^2;\C^2)$ for some non-negative integer $k$ and real $p>1$.  Then, $\bA[f] \in L^{\infty}_{\loc} \cap W^{k+1,p}_{\loc}(\R^2;\C^2)$ and $\bA * [\bF] \in L^{\infty}_{\loc} \cap W^{k+1,p}_{\loc}(\R^2)$.
\end{lemma}
\begin{proof}
    The proof follows by \Cref{lem:superpot-relations,lem:regularityofgauge}, $\bA[f] = (-\partial_2 \Phi[f], \partial_1 \Phi[f])$, and noting that 
    \begin{align*}
        \bA * [(F_1,F_2)] = -\partial_2 \Phi[F_1]  -\partial_1 \Phi[F_2],
    \end{align*}
    where $\bF = (F_1,F_2).$
\end{proof}

\begin{lemma}[Regularity of minimizers]
\label{lem:regularityofminimizer}
Let $\beta \in \R$. If $u \in H^1(\R^2)$ is a minimizer for $\gamma_{\ast}(\beta)$, i.e., $\cE_{\beta,\gamma_*(\beta)}[u]=0$, then it satisfies the Euler--Lagrange equation
\begin{multline}\label{eq:EL}
\Big[-\left(\nabla + {\rm i}\beta\bA\left[|u|^2\right]\right)\cdot \left(\nabla + {\rm i}\beta\bA\left[|u|^2\right]\right) - 2\beta^2 \bA * \left[\bA\left[|u|^2\right]|u|^2\right]\\
- 2\beta \bA * \left[\bJ[u]\right] - 2\gamma_{*}(\beta)|u|^2\Big] u = \lambda u
\end{multline}
weakly in $\R^2$, where
\begin{align*}
\lambda  := 
\int_{\R^2} \left[-|\nabla u|^2 + \beta^2 \left|\bA\left[|u|^2\right]\right|^2|u|^2\right].
\end{align*}
Moreover, $u$ belongs to $C^{\infty}(\R^2)$.
\end{lemma}
\begin{remark}
    Note that with $\beta=0$ this reduces to the well-known cubic NLS equation \eqref{eq:NLS} studied in \cite{Weinstein-83}.
    Also, an Euler--Lagrange equation for the corresponding confined problem \eqref{eq:AFP-gse} at $\beta \in \R$ and $\gamma=0$ was given in \cite[Appendix]{CorLunRou-17}, but without a proof of regularity.
    By following the same proof as below, assuming $V \in L^1_\loc$ and the existence of a minimizer $u \in H^1 \cap L^2(\R^2,|V|\dd\bx)$ s.t.\ $\cE_{\beta,\gamma,V}[u] = E_{\beta,\gamma,V} > -\infty$, it introduces an extra term $+Vu$ in the l.h.s.\ of \eqref{eq:EL}, replacing $\gamma_*(\beta)$ by $\gamma$, and the constant 
    $$\lambda = 2E_{\beta,\gamma,V} - \int_{\R^2} \left[|\nabla u|^2 + V|u|^2 - \beta^2 \left|\bA\left[|u|^2\right]\right|^2|u|^2\right].$$ 
    Further, if $V \in C^\infty$ then also $u \in C^\infty$.
\end{remark}

\begin{proof}[Proof of \Cref{lem:regularityofminimizer}]
Let $v \in C^{\infty}_c(\R^2)$ and $u \in H^1(\R^2)$ satisfy $\cE_{\beta,\gamma_*(\beta)}[u]=0$. Then, by 
\Cref{def:gamma}, the function $u$ minimizes $\cE_{\beta,\gamma_*(\beta)}[w] \ge 0$ among all $w \in H^1(\R^2)$ (normalized in $L^2$).
Hence,
$$
\frac{\dd }{\dd \varepsilon}\Big|_{\varepsilon = 0} \cE_{\beta,\gamma_*(\beta)}\left[\frac{u+\varepsilon v}{\|u+\varepsilon v\|_{L^2}}\right] = 0.
$$
We write
$$
\cE_{\beta,\gamma_*(\beta)}\left[\frac{u+\varepsilon v}{\|u+\varepsilon v\|_{L^2}}\right] = \frac{\mathcal{E}_1[u+\varepsilon v]}{\|u+\varepsilon v\|_{L^2}^2} + \frac{\mathcal{E}_2[u+\varepsilon v]}{\|u+\varepsilon v\|_{L^2}^4} + \frac{\mathcal{E}_3[u+\varepsilon v]}{\|u+\varepsilon v\|_{L^2}^6},
$$
where, with the notation $w_0(\bx) := \log |\bx|$,
\begin{align*}
\mathcal{E}_1[u]  :={} & \int_{\R^2} |\nabla u|^2, \\
\mathcal{E}_2[u] :={} & \int_{\R^2} \left[ 2\beta \bA\left[|u|^2\right] \cdot \bJ[u] - \gamma_{*}(\beta)|u|^4 \right] \\
={} & \beta \iint_{\R^4} \nabla^{\perp} w_0(\bx-\by) |u(\by)|^2 \cdot {\rm i}(u \nabla \bar{u}-\bar{u} \nabla u)(\bx) \,\dd \bx \dd \by - \gamma_{*}(\beta) \int_{\R^2} |u|^4,\\
\mathcal{E}_3[u] :={} & \beta^2 \int_{\R^2} \left|\bA\left[|u|^2\right]\right|^2 |u|^2 \\
={} & \beta^2 \iiint_{\R^6} \nabla^{\perp} w_0(\bx-\by) \cdot \nabla^{\perp} w_0(\bx-\bz) |u(\bx)|^2 |u(\by)|^2 |u(\bz)|^2 \,  \dd  \bx \dd  \by \dd  \bz.
\end{align*}
First, using $\|u\|_{L^2}=1$, we have
$$
\frac{\dd }{\dd \varepsilon}\Big|_{\varepsilon = 0} \|u+\varepsilon  v\|_{L^2}^{2k} = \lim_{\varepsilon \to 0} \frac{\|u+\varepsilon  v\|_{L^2}^{2k} - \|u\|_{L^2}^{2k}}{\varepsilon} = 2k\Re\int_{\R^2}  \bar{v}u
$$
for $k \in \N$, and
(in the distribution sense)
$$
\frac{\dd }{\dd \varepsilon}\Big|_{\varepsilon = 0} \mathcal{E}_1[u+\varepsilon  v] = \lim_{\varepsilon \to 0} \frac{\mathcal{E}_1[u+\varepsilon  v] - \mathcal{E}_1[u]}{\varepsilon} = 2\Re\int_{\R^2} \nabla \bar{v} \cdot \nabla u = 2\Re\int_{\R^2} \bar{v}(-\Delta u).
$$
Second,
\begin{align*}
\frac{\dd }{\dd \varepsilon}\Big|_{\varepsilon = 0} \mathcal{E}_2[u+\varepsilon  v] ={} & \lim_{\varepsilon \to 0} \frac{\mathcal{E}_2[u+\varepsilon  v] - \mathcal{E}_2[u]}{\varepsilon} \\
={} & 2\beta \Re \iint_{\R^4}  \Big( \nabla^{\perp} w_0(\bx-\by) u(\by) \bar{v}(\by) \cdot {\rm i}(u \nabla \bar{u}-\bar{u} \nabla u)(\bx) \\
& +\nabla^{\perp} w_0(\bx-\by)|u(\by)|^2 \cdot {\rm i}(u\nabla\bar{v}-\bar{v}\nabla u)(\bx)\Big) \dd  \bx \dd  \by - 4\gamma_*(\beta) \Re\int_{\R^2} \bar{v}|u|^2u \\
={} & -4 \Re\int_{\R^2} \bar{v}\left(\beta u\nabla^{\perp} w_0 * \bJ[u] + {\rm i}\beta \nabla u \cdot \bA\left[|u|^2\right] + \gamma_*(\beta)|u|^2u\right),
\end{align*}
by using integration by parts and $\div \bA = 0$.
Third,
\begin{align*}
\frac{\dd }{\dd \varepsilon}\Big|_{\varepsilon = 0} \mathcal{E}_3[u+\varepsilon  v] ={} & \lim_{\varepsilon \to 0} \frac{\mathcal{E}_3[u+\varepsilon  v] - \mathcal{E}_3[u]}{\varepsilon} \\
={} & 2 \beta^2 \Re \iiint_{\R^6} \nabla^{\perp}w_0(\bx-\by) \cdot \nabla^{\perp} w_0(\bx-\bz) \bigg[u(\bx) \bar{v}(\bx)|u(\by)|^2|u(\bz)|^2 \\
& +|u(\bx)|^2 u(\by) \bar{v}(\by)|u(\bz)|^2+|u(\bx)|^2|u(\by)|^2 u(\bz) \bar{v}(\bz)\bigg]\, \dd  \bx \dd  \by \dd  \bz \\
={} & 2\beta^2 \Re \int_{\R^2} \bar{v}\left( \big|\bA\left[|u|^2\right]\big|^2u - 2(\nabla^{\perp} w_0) * (|u|^2 \bA\left[|u|^2\right])u\right).
\end{align*}
Thus, using $\|u\|_{L^2}=1$, we get
\begin{align*}
0 ={} & \frac{\dd }{\dd \varepsilon}\Big|_{\varepsilon = 0} \cE_{\beta,\gamma_*(\beta)}\left[\frac{u+\varepsilon v}{\|u+\varepsilon v\|_{L^2}}\right] \\
={} & \frac{\dd }{\dd \varepsilon}\Big|_{\varepsilon = 0} \mathcal{E}_1[u+\varepsilon v] + \frac{\dd }{\dd \varepsilon}\Big|_{\varepsilon = 0} \mathcal{E}_2[u+\varepsilon v] + \frac{\dd }{\dd \varepsilon}\Big|_{\varepsilon = 0} \mathcal{E}_3[u+\varepsilon v] \\
& - \mathcal{E}_1[u]\frac{\dd }{\dd \varepsilon}\Big|_{\varepsilon = 0} \|u+\varepsilon v\|_{L^2}^2 - \mathcal{E}_2[u]\frac{\dd }{\dd \varepsilon}\Big|_{\varepsilon = 0} \|u+\varepsilon v\|_{L^2}^4 - \mathcal{E}_3[u]\frac{\dd }{\dd \varepsilon}\Big|_{\varepsilon = 0} \|u+\varepsilon v\|_{L^2}^6 \\
={} & 2 \Re \int_{\R^2} \bar{v}\Big[-\Delta u + \beta^2\left|\bA\left[|u|^2\right]\right|^2u - 2\beta^2 \left(\nabla^{\perp} w_0 * \left[|u|^2 \bA\left[|u|^2\right] \right]\right) u \\
& - 2\beta (\nabla^{\perp} w_0 * \bJ[u]) u - 2\ii \beta \bA\left[|u|^2\right] \cdot \nabla u - 2\gamma_*(\beta)|u|^2u -  \left(\mathcal{E}_1[u]+2\mathcal{E}_2[u]+3\mathcal{E}_3[u]\right)u\Big] \\
={} & 2 \Re \int_{\R^2} \bar{v}\Big[-\left(\nabla + \ii\beta \bA\left[|u|^2\right]\right) \cdot \left(\nabla + \ii\beta \bA\left[|u|^2\right]\right) - 2\beta^2 \bA * \left[|u|^2 \bA\left[|u|^2\right]\right] \\
& - 2 \beta \bA * \left[\bJ[u]\right] - 2\gamma_*(\beta)|u|^2 - \lambda\Big]u.
\end{align*}
Here, we have set 
$$
\lambda = \mathcal{E}_1[u] + 2\mathcal{E}_2[u] + 3\mathcal{E}_3[u] = 2\cE_{\beta,\gamma_*(\beta)}[u] - \mathcal{E}_1[u] + \mathcal{E}_3[u] = \int_{\R^2} \left[-|\nabla u|^2 + \beta^2 \left|\bA\left[|u|^2\right]\right|^2|u|^2\right].
$$
Replacing $v$ by ${\rm i}v$ in the above equation, we obtain the desired equation \eqref{eq:EL}.

Now, by using a bootstrap argument, we prove that the solution $u \in H^1$ of \eqref{eq:EL} is smooth. Note that $u\in L^{r}(\R^2)$ for every $2\leq r<\infty$, by Sobolev embedding. We rewrite the equation  \eqref{eq:EL} as follows
\begin{align}
\label{eq:shorterformofequa}
\Delta u = \sum_{i=1}^6 I_i,
\end{align}
where
\begin{align*}
I_1 &:= \beta^2\left|\bA\left[|u|^2\right]\right|^2u, \\
I_2 &:= - 2\beta^2 \bA 
*\left[|u|^2 \bA\left[|u|^2\right]\right] u,\\
I_3 &:= - 2\beta \bA * \left[\bJ[u]\right] u,\\
I_4 &:=  - 2i \beta \bA\left[|u|^2\right] \cdot \nabla u,\\
I_5 &:=  - 2\gamma_*(\beta)|u|^2u,\\
I_6 &:= -  \lambda u.
\end{align*}
Let $u \in W^{k,p}_{\loc}(\R^2;\C)$ for some $p>1$ and $k \in \N$
(we can start with $p=2$ and $k=1$ since $u\in H^1(\R^2;\C)$). Our first step is to prove that $I_i \in W^{k-1,q}_{\loc}(\R^2)$ for every $1<q<p$. 

For the term $I_1$, we use \Cref{lem:productlemma} to deduce that $|u|^2 \in L^1(\R^2) \cap W^{k,q}_{\loc}(\R^2)$ for every $1< q <p$. Then, by \Cref{lem:regularityofmagneticgauge}, we obtain
\begin{align}
\label{eq:regularityA}
\bA\left[|u|^2\right] \in W^{k+1,q}_{\loc}(\R^2;\C^2)
\end{align}
for every $1 \leq q<p$. Hence, by Sobolev embedding theorem, we have $\bA\left[|u|^2\right] \in W^{k,p}_{\loc}(\R^2)$, and, by \Cref{lem:productlemma} and Sobolev embedding theorem, we conclude that $I_1 \in W^{k-1,q}_{\loc}(\R^2)$ for every $1 \leq q <\infty$. 

For the term $I_2$, we use 
H\"older's inequality and \Cref{prop:youngconv} to obtain
\begin{align*}
\int_{\R^2} |u|^2 |\bA\left[|u|^2\right]| 
\leq \| u\|^{2}_{L^{2r}} \|\bA\left[|u|^2\right]\|_{L^{\frac{r}{r-1}}} 
\leq C_r \| u\|^{2}_{L^{2r}} \|u\|^2_{L^{\frac{4r}{3r-2}}} < \infty.
\end{align*}
for every $1<r<2$, where $C_r$ depends on $r$. Hence, $|u|^2 \bA\left[|u|^2\right] \in L^1(\R^2).$ Moreover, $|u|^2 \bA\left[|u|^2\right] \in W^{k,q}_{\loc}(\R^2)$ for every $1 \leq q<p$, by \Cref{lem:productlemma}, Sobolev embedding, and \eqref{eq:regularityA}. Hence, by Sobolev embedding, we have $|u|^2 \bA\left[|u|^2\right] \in W^{k-1,p}_{\loc}(\R^2)$. In conclusion, by \Cref{lem:regularityofmagneticgauge}, we derive that $\bA * [|u|^2 \bA\left[|u|^2\right]] \in W^{k,p}_{\loc}(\R^2)$. This together with \Cref{lem:productlemma} implies that $I_2 \in W^{k,q}_{\loc}(\R^2)$ for every $1<q<p.$ 

Now, we prove the estimate on $I_3$. By Cauchy--Schwarz inequality, we imply that
\begin{align*}
\int_{\R^2} |\bJ[u] | \leq \int_{\R^2} |\nabla u| \, |u| \leq \| \nabla u \|_{L^2} \| u\|_{L^2}< \infty.
\end{align*}
Hence, $\bJ[u] \in L^1(\R^2)$ and, by \Cref{lem:productlemma}, we have $\bJ[u] \in W^{k-1,q}_{\loc}(\R^2;\C)$ for every $1 \leq q <p.$ In conclusion, by \Cref{lem:regularityofmagneticgauge} and Sobolev embedding, it is obtained that $\bA \ast [\bJ[u]] \in W^{k-1,p}_{\loc}(\R^2;\C)$. Combining this with \Cref{lem:productlemma} gives $I_3 \in W^{k-1,q}_{\loc}(\R^2)$ for every $1 \leq q <p$. 

For the term $I_4$, by \eqref{eq:regularityA} and \Cref{lem:productlemma}, we derive $I_4 \in W^{k-1,q}_{\loc}(\R^2;\C)$. 

Finally, for the terms $I_5, I_6$, we can use \Cref{lem:productlemma} to derive that both belong to $W^{k-1,q}_{\loc}(\R^2)$ for every $1 \leq q < p$. This completes the first step. 

Now, we use \eqref{eq:shorterformofequa} and elliptic regularity, see \cite[Chapter 9]{GilTru-01}, to derive that $u \in W^{k+1,q}_{\loc}(\R^2)$ for every $1\leq q < p$. Hence, by Sobolev embedding, we conclude $u \in W^{k,r}_{\loc}(\R^2)$ for every $1<r < \infty$. 
To complete the bootstrap argument, we use the first step and the above argument, to conclude $u \in W^{k+1,p}_{\loc}(\R^2)$. This proves a higher regularity for $u.$ By repeating the argument and using the fact that $u \in H^1(\R^2),$ we derive that $u \in W^{k,2}_{\loc}(\R^2)$ for all integer $k>0.$ Finally, by Morrey's inequality, we conclude that $u \in C^{\infty}(\R^2)$.
\end{proof}

\subsection{Supersymmetric magnetic identity}\label{sec:mag-susy}

In this subsection, we prove the square factorization of the magnetic self-energy and the asymptotic behavior of the superpotential $\Phi[|u|^2]$ for $u \in H^1(\R^2)$
(a 2D variant of ``Newton's theorem'' for gravitational potentials).

\begin{lemma}[Magnetic identity]
\label{lem:mag-identity}
For any magnetic potential $\bA = (A_1,A_2) \in C^{\infty}(\R^2;\R^2)$ 
and $u \in C^{\infty}(\R^2;\C)$ we have the 
pointwise identity
\begin{equation}\label{eq:mag-identity-pt}
|(\nabla + {\rm i}\bA)u|^2 \mp \bA \cdot \nabla^\perp(|u|^2)
= |(\partial_1 \pm {\rm i}\partial_2)u + {\rm i}(A_1 \pm {\rm i}A_2)u|^2 \pm \nabla^\perp \cdot \bJ[u].
\end{equation}
Further, if $u$ has compact support, we derive
\begin{equation}\label{eq:mag-identity}
\int_{\R^2} \left[ |(\nabla + {\rm i}\bA)u|^2 \pm \curl\bA |u|^2 \right]
= \int_{\R^2} |(\partial_1 \pm {\rm i}\partial_2)u + {\rm i}(A_1 \pm {\rm i}A_2)u|^2.
\end{equation}
\end{lemma}
\begin{proof}
We have
\begin{align*}
\text{r.h.s. of } \eqref{eq:mag-identity-pt} ={} & |(\partial_1 \pm {\rm i}\partial_2)u|^2 + |(A_1 \pm {\rm i}A_2)u|^2 + 2\Re (\partial_1 \pm {\rm i}\partial_2)u \, \overline{{\rm i}(A_1 \pm {\rm i}A_2)u} \\
& \pm \frac{{\rm i}}{2} \bigl((-\partial_2 (u \partial_1 \overline{u} - \overline{u}\partial_1 u)+ \partial_1(u \partial_2 \overline{u} - \overline{u}\partial_2 u)\bigr) \\
={} & |\partial_1 u|^2 + |\partial_2 u|^2 \pm  2 \Re (\partial_1 u 
\, \overline{{\rm i} \partial_2 u}) + 2\Re(\partial_1 u \overline{{\rm i}A_1 u} + \partial_2 u \overline{{\rm i}A_2 u})  \\
& \pm 2\Re(-\partial_1 u \overline{A_2 u} + \partial_2 u \overline{A_1 u}) + |\bA u|^2 \pm 2 \Im (\partial_2 u \partial_1 \overline{u}) \\
={} & |\nabla u|^2 + |\bA u|^2+ 2 \Re (\nabla u \cdot \overline{{\rm i}\bA u}) \mp \bA \cdot \nabla^{\perp} (|u|^2)
\\
={} & \text{l.h.s. of } \eqref{eq:mag-identity-pt}.
\end{align*}
The identity \eqref{eq:mag-identity} is a direct result of integration by parts.
\end{proof}

\begin{proposition}
\label{cor:magneticselfdualformula}
If $\beta \in \R$ and $u \in H^1(\R^2)$, then 
\begin{align*}
\int_{\R^2} \left[ |(\nabla+{\rm i}\beta \bA\left[|u|^2\right])u|^2 \pm 2\pi \beta |u|^4 \right]
= \int_{\R^2} \left| (\partial_1 \pm \ii\partial_2)u + \ii\beta (A_1[|u|^2] \pm \ii A_2[|u|^2])u \right|^2.
\end{align*}
In particular, the Bogomolnyi bound \eqref{eq:bogobound-self} holds.
\end{proposition}
\begin{proof}
If $u \in C^{\infty}_c(\R^2;\C)$, then the identity was derived directly by \Cref{lem:mag-identity}. If $u \in H^1(\R^2)$, then the result can be obtained by \cite[Prop.~3.5]{LunRou-15} and a density argument.
\end{proof}

\begin{lemma}[Supersymmetric factorization]
\label{lem:mag-super}
If $\beta \in \R$ and $u \in H^1(\R^2)$, 
then
\begin{equation}\label{eq:mag-super}
\int_{\R^2} \left[\left|\left(\nabla+{\rm i}\beta \bA\left[|u|^2\right]\right)u\right|^2 \pm 2 \pi \beta |u|^4 \right]
= \int_{\R^2} \left|(\partial_1 \pm {\rm i}\partial_2)(e^{\mp \beta \Phi[|u|^2]}u)\right|^2 e^{\pm 2\beta \Phi[|u|^2]}
\end{equation}
\end{lemma}
\begin{proof}
The lemma follows by using $$\bA\left[|u|^2\right] = (-\partial_2 \Phi[|u|^2], \partial_1 \Phi[|u|^2]),$$
in \Cref{cor:magneticselfdualformula}, where we apply \Cref{lem:superpot-relations,lem:regularityofgauge}, together with $|u|^2 \in L^1 \cap L^p(\R^2)$, $p > 1$.
\end{proof}

\begin{lemma}[``Newton's lemma'']
\label{lem:newtonlemma}
If $f \in L^1 \cap L^p(\R^2)$ for some $p>1$, then
\begin{align*}
\lim_{|\bx| \to \infty}\frac{\Phi[f](\bx)}{\log|\bx|} = \int_{\R^2} f .
\end{align*}
\end{lemma}
\begin{proof} We follow the same argument as the proof of \cite[Lemma~1.1]{CheLi-93}.
Let $R>1$ be an arbitrary constant and $\bx \in \R^2$ be such that $|\bx|>4R$. We write $\R^2 = S_1 \cup S_2 \cup S_3$ where
\begin{align*}
S_1 &= \{\by \in \R^2: |\by-\bx|\leq 1\}, \\
S_2 &= \{\by \in \R^2: |\by-\bx|>1, |\by| \leq R\}, \\
S_3 &= \{\by \in \R^2: |\by-\bx|>1, |\by|> R\}.
\end{align*}
Then, we have 
\begin{align*}
\frac{\left| \Phi[f](\bx) - \log |\bx| \int_{\R^2} f \right|}{\log|\bx|} 
&\le \frac{\int_{\R^2} \bigl| \log |\bx-\by| - \log (|\by|+1) - \log|\bx| \bigr| \, |f(\by)| \,\dd \by}{\log|\bx|}
\\ &\le I_1 +I_2 +I_3, 
\end{align*}
where 
\begin{align*}
I_i :=  \frac{\int_{S_i} \bigl| \log |\bx-\by| - \log (|\by|+1) - \log|\bx| \bigr| |f(\by)| \, \dd \by}{\log|\bx|},
\end{align*}
for every $i=1,2,3.$

If $\by \in S_1$, then 
$\left| |\by| - |\bx| \right| \leq |\by-\bx| \leq 1$
and, since $|\bx|>4R>4$, 
$$0 \leq \log (|\by|+1) \leq \log(|\bx|+|\by-\bx|+1) \leq \log(|\bx|+2) \le 2\log|\bx|.$$
Hence, by the triangle inequality and H\"older's inequality, we have
\begin{align*}
I_1  &\leq 3 \int_{B(\bx,1)} |f(\by)|\, \dd \by + \frac{\int_{B(\bx,1)} \left| \log |\bx-\by| \right|\, |f(\by)|\, \dd \by}{\log|\bx|} \\ 
&\leq 3 \int_{B(\bx,1)}    |f(\by)| \,\dd \by + \frac{\left(\int_{B(0,1)} \left| \log |\by| \right|^{\frac{p}{p-1}}\, \dd  \by \right)^{\frac{p-1}{p}} \left(\int_{B(\bx,1)} |f(\by)|^p \dd \by\right)^{\frac{1}{p}}}{\log|\bx|} \\ 
&\leq 3 \int_{B(\bx,1)}    |f(\by)|\, \dd \by + \frac{\left(\int_{B(0,1)} \left| \log |\by| \right|^{\frac{p}{p-1}}\, \dd  \by \right)^{\frac{p-1}{p}} \left(\int_{\R^2} |f(\by)|^{p} \dd \by\right)^{\frac{1}{p}}}{\log|\bx|},
\end{align*}
which implies that $I_1$ goes to zero uniformly as $|\bx| \to \infty$.

Next, if $\by \in S_2$ then
\begin{align*}
\frac{3}{4} \leq \frac{|\bx| -|\by|}{|\bx|}  \leq \frac{|\bx -\by|}{|\bx|} \leq \frac{|\bx| +|\by|}{|\bx|} \leq \frac{5}{4}.
\end{align*}
Hence
\begin{align*}
I_2 &\leq  \frac{\int_{B(0,R) } \left| \log (|\by|+1) \right|\, |f(\by)| \,\dd \by}{\log|\bx|} +  \frac{\int_{B(0,R)} \left| \log \frac{|\bx-\by|}{|\bx|} \right| \, |f(\by)|\, \dd \by}{\log|\bx|}
\\ &\leq  \log (R+1) \frac{\int_{B(0,R) }  |f(\by)| \dd \by}{\log|\bx|} +  2\frac{\int_{B(0,R)} |f(\by)|\, \dd \by}{\log|\bx|},
\end{align*}
which implies that $I_2 $ goes to zero uniformly as $|\bx| \to \infty$ (for a fixed $R$).

Finally, to estimate $I_3$, we note that 
\begin{align*}
\frac{\left| \log |\bx-\by| - \log (|\by|+1) - \log|\bx| \right| }{\log|\bx|},
\end{align*}
is bounded for $\by \in S_3$ by a constant independent of $R$. In fact, if $\by \in S_3, |\by| \leq 2|\bx|$, then
\begin{align*}
\frac{\left| \log |\bx-\by| - \log (|\by|+1) - \log|\bx| \right| }{\log|\bx|} \leq   \frac{\left| \log |\bx-\by| \right| }{\log|\bx|}+  \frac{\left|\log (|\by|+1)\right| }{\log|\bx|} + 1 \leq \frac{2 \log(3)}{\log(4)} + 3,
\end{align*}
and for $\by \in S_3, |\by| > 2|\bx|$, we have 
\begin{align*}
\frac{\left| \log |\bx-\by| - \log (|\by|+1) - \log|\bx| \right| }{\log|\bx|} &\leq  \frac{\left| \log \frac{|\bx-\by|}{|\by|+1}    \right| }{\log|\bx|} +1 \leq \log(4)+1,
\end{align*}
where we used
\begin{align*}
\frac{1}{4}\leq  \frac{1}{2}  \frac{|\by|}{|\by|+1} \leq  \frac{|\by|- |\bx|}{|\by|+1}  \leq \frac{|\bx-\by|}{|\by|+1} \leq \frac{|\bx|+|\by|}{|\by|+1} \leq \frac{3}{2}. 
\end{align*}
In conclusion, 
\begin{align*}
I_3 \leq \left(\frac{2 \log(3)}{\log(4)} + 3\right) \int_{\R^2 \setminus B(0,R)} |f(\by)|\, \dd  \by,
\end{align*}
and taking $R$ as large as we want, by $f \in L^1(\R^2)$, completes the proof.
\end{proof}

\subsection{Proof of Theorem~\ref{thm:magneticstability}}\label{sec:mag-proof}

After these preliminaries, we may now proceed with the proof of our second main theorem, starting with the existence of the explicit solutions \eqref{eq:mag-solution} and then proceeding with their uniqueness for $\beta \ge 2$.

\subsubsection{Quantization of $\beta$: sufficient form of solutions}\label{sec:mag-proof-existence}

For $\beta \in 2 \N$, we
prove that if $P,Q \in \cP$ are linearly independent and such that $\max(\deg P,\deg Q) = \frac{\beta}{2}$ and $\gcd(P,Q)=1$, then
\begin{equation}\label{eq:uPQ}
u_{P,Q} := \sqrt{\frac{2}{\pi \beta}} \, \frac{\overline{P'Q - PQ'}}{|P|^2 + |Q|^2}
\end{equation}
is a minimizer for $\gamma_{\ast}(\beta) = 2\pi\beta$. 
We first need the following lemma. 

\begin{lemma}[Nonlinear Landau level]
\label{cor:generalformofminimizer}
Let $\beta \geq 0$. Then, a function $u \in H^1(\R^2)$
is a minimizer for $\gamma_{\ast}(\beta)= 2\pi \beta$ if and only if 
\begin{align*}
u = e^{-\beta \Phi[|u|^2]} \overline{f},
\end{align*} where $f$ is a polynomial of degree at most $\beta-1$, and $\int_{\R^2} |u|^2 = 1$.   
\end{lemma}
\begin{proof}
Assume that $u \in H^1(\R^2;\C)$ is a minimizer for $\gamma_{\ast}(\beta)=2\pi \beta$. Then, by \Cref{lem:regularityofminimizer}, we obtain that $u \in C^{\infty}(\R^2)$ and
$\cE_{\beta,\gamma_{\ast}(\beta)}[u]=0.$ Hence, by \Cref{lem:regularityofgauge,lem:mag-super}, we derive that 
\begin{align}
\label{eq:initialformminimizer}
u(z) = e^{-\beta \Phi[|u|^2](z)} \overline{f(z)},
\end{align}
for every $z \in \C$, where $f$ is an analytic function. Now, by \Cref{lem:newtonlemma}, we have 
\begin{align*}
e^{-\beta \Phi[|u|^2](\bx)} \geq \frac{C_{\varepsilon}}{(|\bx|+1)^{\beta+\varepsilon}}
\end{align*}
for every $\bx \in \R^2, \varepsilon>0$, where $C_{\varepsilon}$ is a positive constant depending on $\varepsilon.$ Hence, by $\|u\|_{L^2}=1$, we obtain that 
\begin{align*}
\int_{\R^2} \frac{|f(\bx)|^2}{(|\bx|+1)^{2\beta+\varepsilon}} \dd  \bx < \infty,
\end{align*}
for every $\varepsilon>0.$ Therefore, by writing $f(z) = \sum_{n=0}^{\infty} a_n z^n$ for every $z \in \C$ and using polar coordinates, we derive that 
\begin{align*}
2 \pi \int_0^{\infty} \frac{\sum_{n=0}^{\infty} |a_n|^2 r^{2n+1}}{(r+1)^{2\beta+\varepsilon}} \dd  r <\infty,
\end{align*}
for every $\varepsilon>0.$ In conclusion, $a_n =0$ for every $n >\beta -1 $ which, together with \eqref{eq:initialformminimizer}, derives the necessary part of the corollary. The other direction follows directly from \Cref{lem:mag-super}.
\end{proof}

Now, we prove our claim that \eqref{eq:uPQ} is a minimizer.

\begin{proposition}
\label{prop:integercaseminmizer}
Let $\beta \in 2 \N$ and $P,Q$ be two coprime and linearly independent complex polynomials such that 
$\max (\deg P,\deg Q) =\frac{\beta}{2}$.
Then, \eqref{eq:uPQ} is a minimizer for 
$\gamma_{\ast}(\beta) = 2\pi\beta$.
\end{proposition}
\begin{proof}
Since $P,Q$ are linearly independent, $W(P,Q)$ is not identically zero (but can have zeros).
Further, since $\gcd(P,Q)=1$ we have $u_{P,Q} \in C^\infty(\R^2)$.
By \Cref{cor:generalformofminimizer}, it is enough to verify that $u_{P,Q} \in H^1(\R^2),$  $\int_{\R^2} |u_{P,Q}|^2 = 1$, 
and that it can be written in the factorized form
$$
u_{P,Q} = K e^{-\beta \Phi[|u_{P,Q}|^2]}\bar{f},
$$
for $f := W(P,Q) \neq 0$ and a non-zero constant $K = K(P,Q)$ depending on $P,Q$. 

First, we prove that $\|u_{P,Q}\|_{L^2}=1.$
By using the identity \eqref{eq:Liouville-degree} in \Cref{thm:liouville}, we deduce that
\begin{align*}
\int_{\R^2} |u_{P,Q}|^2 
= \frac{1}{4\pi \beta} \int_{\R^2} |f|^2 e^{\psi}
= \frac{8\pi \max(\deg(P),\deg(Q))}{4\pi \beta} =1,
\end{align*}
where 
\begin{align*}
\psi = \psi_{P,Q} := \log(8) -2 \log\left( |P|^2 + |Q|^2 \right).
\end{align*}

Now, by \Cref{thm:liouville} and \Cref{lem:superpot-relations}, we have 
\begin{align}
\label{eq:connectionliouvileandLandau}
-\Delta \psi =  |f|^2 e^{\psi} = 4\pi \beta |u_{P,Q}|^2 =2\beta \Delta \Phi[|u_{P,Q}|^2]
\end{align}  in $\R^2.$
In conclusion, \begin{equation}\label{eq:Phi-general-solution}
\Phi[|u_{P,Q}|^2] = \frac{1}{\beta} \log(|P|^2 + |Q|^2) + \Psi, 
\end{equation}
in $\R^2$ where $\Psi = \Psi(P,Q)$ is a harmonic function $\R^2 \to \R$. 
By $u_{P,Q} \in C^\infty$, the fact that $P,Q$ are non-zero coprime polynomials, and \Cref{lem:regularityofgauge} we have that $\Psi \in C^\infty(\R^2;\R)$.

Now, we aim at proving that $\Psi$ is indeed a constant. To see this, we use \Cref{lem:newtonlemma} to deduce that 
\begin{align*}
    |\Psi(\bx)| \leq  C |\log|\bx||,
\end{align*}
for every $|\bx| \ge 2$, where $C =C(P,Q)$ is a positive constant which depends on $P,Q$. Then, by a simple application of Liouville's theorem, we obtain that $\Psi$ is a constant. Indeed, consider a complex analytic function $G: \C \to \C$ where $\Re(G)=\frac{ \Psi}{C}$. Then, $\frac{ e^{G}- e^{G(0)}}{z}$ is a meromorphic function such that 
\begin{align*}
    \biggl | \frac{e^{G(z)}-e^{G(0)}}{z} \biggr | \leq \frac{e^{ \frac{\Psi(z)}{C}} + e^{\frac{\Psi(0)}{C}}}{|z|} \leq 1+ e^{\frac{\Psi(0)}{C}},
\end{align*}
for $|z| \geq 2$. Moreover, $e^{G(z)}- e^{G(0)} = 0$ at $z = 0.$ Hence, $\frac{e^{G}-e^{G(0)}}{z}$ is a bounded entire function. By Liouville's theorem, we have that $\frac{e^{G}-e^{G(0)}}{z}$ is a constant. However, this implies that $e^{G}$ is a polynomial of degree at most $1$, 
and further, it cannot have a zero due to the regularity of $\Psi = C\Re(G)$.
Then, this can only occur if $e^G$ is a constant. Therefore, $\Psi = C \log|e^G|$ is a constant. Thus, this implies the desired factorization 
\begin{align*}
u_{P,Q} =e^{\beta \Psi} \sqrt{\frac{2}{\pi \beta}} e^{-\beta \Phi[|u|^2]} \Bar{f}.
\end{align*}

Finally, we show that $u_{P,Q} \in H^1(\R^2)$. Since $\|u_{P,Q}\|_{L^2}=1$, it is enough to prove that $(\partial_z u_{P,Q}, \partial_{\bar{z}} u_{P,Q} ) \in L^2(\R^2;\C^2).$ By \eqref{eq:mag-solution}, we have 
\begin{align*}
    \partial_z u_{P,Q} = -\frac{\overline{W(P,Q)} (P' \overline{P} + Q' \overline{Q})}{(|P|^2 + |Q| ^2)^2}.
\end{align*}
Hence, $\partial_z u_{P,Q}$ is smooth and, by \Cref{lem:degreeofwronskian}, we have 
\begin{align*}
     |\partial_z u_{P,Q}(z)| \leq C |z|^{-3},
\end{align*}
for $|z|>1$, where $C$ is a constant which depends on $P,Q$. In conclusion, $\partial_z u_{P,Q} \in L^2(\R^2)$. By a similar argument, we derive that $\partial_{\bar{z}} u_{P,Q} \in L^2(\R^2)$, which completes the proof.
\end{proof}

\subsubsection{Value of $\gamma_{\ast}(\beta)$ and general form of minimizers for $\beta \geq 2$}\label{sec:mag-proof-largebeta}

In this part, we derive that $\gamma_{\ast}(\beta)=2 \pi \beta$ for $\beta \ge 2$ and prove the quantization of the flux $\beta$ and the general necessary form of minimizers for every $\beta \in 2\N$.
We need the following lemma.

\begin{lemma}[Monotonicity property]\label{lem:keylemma}
The function $\beta \mapsto \frac{\gamma_*(\beta)}{\beta}$ is decreasing for $\beta >0$, i.e., for every $0< \beta' < \beta$, we have 
\begin{equation}\label{monotonicity}
\frac{\gamma_{\ast}(\beta)}{\beta}  \leq \frac{\gamma_{\ast}(\beta')}{\beta'}.
\end{equation}
\end{lemma}

\begin{proof}
Note that, by scaling $u \mapsto \frac{u}{\sqrt{\beta}}$
(compare \eqref{eq:LGN-selfmagnetic}), 
we have
$$
\frac{\gamma_{\ast}(\beta)}{\beta} = \inf\left\{\frac{\cE_{1}[u]}{\int_{\R^2} |u|^4} = \frac{\int_{\R^2} \left|(\nabla + {\rm i}\bA\left[|u|^2\right]) u\right|^2}{\int_{\R^2} |u|^4} : u \in H^1(\R^2;\C), \int_{\R^2}|u|^2 = \beta\right\}.
$$
For every $0<\beta'<\beta$, let $\{u_n\}, \{v\} \subset C^{\infty}_c(\R^2;\mathbb{C})$ be (a sequence of) functions such that $\int_{\R^2} |u_n|^2 = \beta' = \beta - \int_{\R^2} |v|^2$ and
\begin{align*}
\frac{\gamma_{\ast}(\beta')}{\beta'} = \lim_{n \to \infty} \frac{\int_{\R^2} \left|(\nabla + {\rm i}\bA[|u_n|^2]) u_n\right|^2}{\int_{\R^2} |u_n|^4}.
\end{align*}
By dilation $r_n u_n(r_n \bx)$, $t v(t \bx)$ for $r_n =\|u_n\|_{L^4}^{-2}$, $t = \|v\|^{-2}_{L^4}$, we can also assume that $$\|u_n\|_{L^4} = 1 = \|v\|_{L^4}.$$ 
Let $\bx_n \in \R^2$ and $\varepsilon_n >1$ be sequences and define $\Tilde{u}_n(\bx) = \varepsilon_n u_n(\varepsilon_n (\bx+\bx_n))$ for every $\bx \in \R^2$, $n \in \N$. 
Now, by choosing $|\bx_n| = 1+ \diam\frac{\supp(u_n)}{\varepsilon_n} + \diam \supp(v)$, we have $$d_n := \dist(\supp(\Tilde{u}_n), \supp(v)) \geq 1.$$
Here, $\frac{\supp(u_n)}{\varepsilon_n}:=\{\bx \in \R^2: \varepsilon_n \bx \in \supp(u_n)\}.$
Then,
$\int_{\R^2} |\tilde{u}_n|^4 = \varepsilon_n^2$, 
$\cE[\tilde{u}_n] = \varepsilon_n^2 \cE[u_n]$, and
\begin{align*}
\int_{\R^2} |\Tilde{u}_n + v |^2 
= \int_{\R^2} \left( |\Tilde{u}_n|^2 + |v|^2 \right) 
= \int_{\R^2} |u_n|^2 + \int_{\R^2} |v|^2 = \beta.
\end{align*}
Moreover, by \Cref{def:superpotentialdef}, we have 
\begin{equation*}
\begin{aligned}\label{lem:distance}
\int_{\R^2} |\bA[|\Tilde{u}_n|^2]|^2 |v|^2 
\leq \int_{\R^2} \left(\int_{\R^2} \frac{|\Tilde{u}_n(\by)|^2}{|\bx-\by|} \dd\by\right)^{2} |v(\bx)|^2 \,\dd\bx 
\leq \frac{\|\Tilde{u}_n\|_{L^2}^{4} \|v\|_{L^2}^2}{d_n^2} \leq \frac{\beta^{3}}{d_n^2},
\end{aligned}
\end{equation*}
and similarly with $\tilde{u}_n$ and $v$ exchanged.
Therefore,
\begin{align*}
\frac{\gamma_{\ast}(\beta)}{\beta} \leq{} & \frac{\int_{\R^2} \left|(\nabla + {\rm i}\bA[|\Tilde{u}_n + v|^2])(\Tilde{u}_n + v)\right|^2}{\int_{\R^2} |\Tilde{u}_n + v|^4} \\ 
\leq{} & \frac{\int_{\R^2} \left|(\nabla + {\rm i}\bA[|\Tilde{u}_n|^2])\Tilde{u}_n\right|^2 + \int_{\R^2} \left|(\nabla + {\rm i}\bA[|v|^2])v\right|^2}{\int_{\R^2}|\Tilde{u}_n|^4 + \int_{\R^2}|v|^4} 
\\& + \frac{2\left(\int_{\R^2} \left|(\nabla + {\rm i}\bA[|\Tilde{u}_n|^2])\Tilde{u}_n\right|^2\right)^{\frac{1}{2}} \left(\int_{\R^2} |\bA[|v|^2]|^2| \Tilde{u}_n|^2\right)^{\frac{1}{2}}}{\int_{\R^2}|\Tilde{u}_n|^4 + \int_{\R^2}|v|^4}
\\ & + \frac{2\left(\int_{\R^2} \left|(\nabla + {\rm i}\bA[|v|^2])v\right|^2\right)^{\frac{1}{2}} \left(\int_{\R^2} |\bA[|\Tilde{u}_n|^2]|^2 |v|^2\right)^{\frac{1}{2}}}{\int_{\R^2}|\Tilde{u}_n|^4 + \int_{\R^2}|v|^4}
\\ 
\le{} &  \dfrac{\varepsilon_n^2 \dfrac{\gamma_{\ast}(\beta')}{\beta'} + \cE_{1}[v]}{\varepsilon_n^2 + 1} + 2\beta^{\frac{3}{2}} \dfrac{\biggl( \frac{\varepsilon_n^2 \gamma_{\ast}(\beta')}{\beta'}\biggr)^{\frac{1}{2}} + \left(\cE_1[v]\right)^{\frac{1}{2}}}{d_n (\varepsilon_n^2 + 1)}.
\end{align*}
By choosing $\varepsilon_n=n$ and taking the limit $n \to \infty$ in the above, we obtain the desired inequality \eqref{monotonicity}.
\end{proof}

\Cref{lem:keylemma} turns out to be our key tool
which yields various properties of $\gamma_{\ast}(\beta)$ and its minimizers. We have the following immediate corollary.

\begin{proposition}[Supersymmetric saturation]
\label{cor:minimumofenergy}
For every $\beta \geq 2$, we have $\gamma_{\ast}(\beta)= 2\pi \beta$.
\end{proposition}
\begin{proof}
For $\beta = 2 $,
by choosing $P(z)= z$ and $Q(z)=1$ in \Cref{prop:integercaseminmizer}, we obtain $\gamma_{\ast}(2)=4\pi$. For other values of $\beta$, we use \Cref{lem:keylemma} and Bogomolnyi's bound \eqref{eq:bogobound-self} to obtain that
\begin{align*}
2 \pi \leq \frac{\gamma_{\ast}(\beta)}{\beta} \leq \frac{\gamma_{\ast}(2)}{ 2} = 2\pi
\end{align*}
for every $\beta > 2$, which  completes the proof.
\end{proof}

\begin{proposition}[Quantization of flux into nonlinear Landau levels]
\label{cor:nececcityofminimizers}
Let $\beta \ge 0$. Assume that there exists a minimizer $u \in H^1(\R^2)$ for $\gamma_{\ast}(\beta) = 2 \pi \beta$. Then, we must have that $\beta \in 2 \N$ and that there exist linearly independent $P,Q \in \cP^\times$ such that $\max(\deg P,\deg Q) =  \frac{\beta}{2}$, $\gcd(P,Q)=1$, and 
\begin{align*}
u = u_{P,Q} = \sqrt{\frac{2}{\pi \beta}} \, \frac{\overline{P'Q - PQ'}}{|P|^2 + |Q|^2}.
\end{align*}
\end{proposition}
\begin{proof}
By \Cref{lem:regularityofminimizer}, $u \in C^\infty(\R^2;\C)$. 
Also, by \Cref{cor:generalformofminimizer}, we must have 
\begin{align}
\label{eq:formofu}
u(z) = e^{-\beta \Phi[|u|^2](z)} \overline{f(z)},
\end{align} for every $z \in \C$, where $f$ is a non-zero polynomial of degree at most $\beta -1$ (in particular $\beta \ge 1$). 
Define the real function $\psi := -2 \beta \Phi[|u|^2]$. Then, 
$\psi \in C^\infty(\R^2;\R)$
by \Cref{lem:regularityofgauge}, and
\begin{align*}
-\Delta \psi=4 \pi \beta |u|^2 = |\sqrt{4 \pi \beta } f|^2 e^{\psi}.
\end{align*}
Hence, by normalization $\int_{\R^2} |u|^2 =1$ 
and \Cref{thm:liouville}, we have 
\begin{align*}
\psi = \log(8) - 2 \log(|P_0|^2 + |Q_0|^2),
\end{align*}
where $P_0,Q_0$ are complex polynomials which satisfy $\gcd(P_0,Q_0)=1, W(P_0,Q_0) = \sqrt{4\pi \beta} f,$ and \begin{align*}
\max(\deg(P_0),\deg(Q_0)) = \frac{\int_{\R^2} 4 \pi \beta |u|^2}{8 \pi} = \frac{\beta}{2}.
\end{align*}
In particular, we must have $\beta \in 2 \N$.
Now, define 
$$
P:= (4\pi \beta)^{-\frac{1}{4}} P_0 
\quad \text{and} \quad 
Q:= (4\pi \beta)^{-\frac{1}{4}} Q_0.
$$
Hence, $W(P,Q)= f$ and
\begin{align*}
-2 \beta \Phi[|u|^2] = \psi = \log\left(  \frac{2}{\pi \beta (|P|^2+ |Q|^2)^2} \right).
\end{align*}
In conclusion, by \eqref{eq:formofu}, we have 
\begin{align*}
u = \sqrt{\frac{2}{\pi \beta}} \, \frac{\overline{P'Q - PQ'}}{|P|^2 + |Q|^2} = u_{P,Q},
\end{align*}
which completes the proof.
\end{proof}

\begin{remark}[Real minimizers]
\label{rem:realminimizers}
    By \Cref{cor:nececcityofminimizers}, we also find that $\gamma_{\ast}(\beta)$ has a real minimizer for $\beta \geq 2$ if and only if $\beta =2$. To see this, we note that by replacing $P=z$, $Q=1$ for $\beta =2,$ we derive a real minimizer
\begin{align*}
    u_{P,Q} = \sqrt{\frac{1}{\pi}} \,\frac{1}{|z|^2+ 1},
\end{align*}
    for $\gamma_{\ast}(2)$.
    Conversely, assume that  \begin{align*}
    u_{P,Q} = \sqrt{\frac{2}{\pi \beta}} \, \frac{\overline{P'Q - PQ'}}{|P|^2 + |Q|^2},
\end{align*}
is a real-valued function. Then, $P'Q - PQ'$ must be a non-zero real constant. Hence, by our characterization of Wronskian pairs in
\Cref{lem:Wronskian-ODEequivalence,lem:degreecondi}, we derive that 
\begin{align*}
\frac{\beta}{2} = \max(\deg(P),\deg(Q)) =1,
\end{align*}
which completes the proof.
\end{remark}

\begin{remark}[Real interpolation constant]
\label{rem:realgammastar}
   For every $\beta \in \R$, we define the optimal constant of the real-valued minimization problem
\begin{align*}
      \gamma_{\R}(\beta) := \inf \left\{\frac{\cE_{\beta}[u]}{\int_{\R^2}|u|^4}: u \in H^1(\R^2;\R), \int_{\R^2}|u|^2 = 1\right\}.
\end{align*}
We claim that $$\gamma_{\R}(\beta)=\gamma_{\ast}(\beta)=2\pi \beta,$$ for every $\beta \geq 2.$ 
By the previous remark, this shows that, for every $\beta \geq 2,$ $\gamma_{\R}(\beta)$ has a minimizer if and only if $\beta =2$, and that still there are smooth real-valued functions $u$, satisfying $\int_{\R^2} |u|^2=1$ , which makes $\frac{\cE_{\beta}[u]}{\int_{\R^2}|u|^4}$ as close as possible to $2 \pi \beta.$ Now, to prove the claim, we first observe that, by a similar argument as \Cref{lem:keylemma} and \Cref{cor:minimumofenergy} for real functions,
\begin{align*}
    \frac{ \gamma_{\R}(\beta)}{\beta} \leq \frac{\gamma_{\R}(2)}{2}, \quad \forall \beta \geq 2.
\end{align*}
Hence, since there are real-valued minimizers for $\gamma_{\ast}(2)$; see Remark \ref{rem:realminimizers}, we obtain $\gamma_{\R}(2) = \gamma_{\ast}(2) = 4 \pi$ and
\begin{align*}
    \frac{ \gamma_{\R}(\beta)}{\beta} \leq 2 \pi,\quad \forall \beta \geq 2.
\end{align*}
On the other hand, by Bogomolnyi's bound \eqref{eq:bogobound-self} (\Cref{cor:magneticselfdualformula}), we have
\begin{align*}
    \gamma_{\R}(\beta) \geq 2 \pi \beta,
    \quad \forall \beta \ge 0.
\end{align*}
    This completes the proof of the claim.
\end{remark}

\subsubsection{The region $0 < \beta < 2$}\label{sec:mag-proof-smallbeta}

In the previous subsection, we determined the value of $\gamma_*(\beta)$ for every $\beta \geq 2$ and classified 
minimizers for every $\beta \in 2\N$. It leaves an open question concerning the value of $\gamma_*(\beta)$ and its minimizers for $0 < \beta < 2$. In this subsection, we derive some partial results in this region.

First, we derive $\gamma_{\ast}(\beta)>\CLGN$ for every $\beta >0.$ To prove this, we need the following proposition.

\begin{proposition}
\label{prop:magneticbound}
Let $u \in C^{\infty}_c(\R^2;\C), \beta \in \R$. 
Then,
\begin{align*}
\pi^2 \beta^2 \left( \int_{\R^2} |u|^4 \right)^2 
\leq \int_{\R^2} |u|^2\, \biggl| \frac{\bJ[u]}{|u|^2} + \beta  \, \bA[|u|^2]  \biggr|^2 \1_{\{u \neq 0\}} \,
\int_{\R^2} \bigl|\nabla |u|\bigr|^2.
\end{align*}
\end{proposition}
\begin{proof}
By the Hardy-type inequality \eqref{eq:AF-H1} and Lebesgue's dominated convergence, we derive that 
\begin{align}
\label{eq:epsilonconvergence1}
      \lim_{\varepsilon \to 0^+} \int_{\R^2} |u|^2\, \biggl| \frac{\bJ[u]}{|u|^2+\varepsilon} + \beta  \, \bA[|u|^2]  \biggr|^2 = \int_{\R^2} |u|^2\, \biggl| \frac{\bJ[u]}{|u|^2} + \beta  \, \bA[|u|^2] \biggr|^2 \1_{\{u \neq 0\}}.
\end{align}
Here, we used that $ |u| \frac{|\bJ[u]|}{|u|^2+\varepsilon} \leq |\nabla u| $ for every $\varepsilon >0$. Moreover, 
\begin{align*}
     |u|^2 \curl\biggl(\frac{\bJ[u]}{|u|^2+\varepsilon} + \beta  \, \bA[|u|^2]\biggr) = 
     \frac{2\varepsilon |u|^2 \Im( \partial_2 \overline{u} \, \partial_1 u)}{(|u|^2 +\varepsilon)^2} + 2 \pi \beta |u|^4,
\end{align*}
for every $\varepsilon>0.$
In conclusion, by $ \biggl |\frac{2\varepsilon |u|^2 \Im( \partial_2 \overline{u} \, \partial_1 u)}{(|u|^2 +\varepsilon)^2} \biggr | \leq 2 |\partial_1 u| \, |\partial_2 u|$ and Lebesgue's dominated convergence, we derive that
\begin{align}
\label{eq:epsilonconvergece2}
     \lim_{\varepsilon \to 0^+} \int_{\R^2}  |u|^2 \curl\biggl(\frac{\bJ[u]}{|u|^2+\varepsilon} + \beta  \, \bA[|u|^2]\biggr) = 2\pi \beta \int_{\R^2} |u|^4.
\end{align}
Hence, by partial integration and the Cauchy-Schwarz inequality,  we have
\begin{align*}
  &  2 \pi |\beta| \int_{\R^2} |u|^4  = \lim_{\varepsilon \to 0^+}\left | \int_{\R^2} |u|^2 \curl\biggl(\frac{\bJ[u]}{|u|^2+\varepsilon} + \beta  \, \bA[|u|^2]\biggr) \right|  \\
& =\lim_{\varepsilon \to 0^+} \left| \int_{\R^2} \curl\biggl(|u|^2 \frac{\bJ[u]}{|u|^2+\varepsilon} + \beta |u|^2  \, \bA[|u|^2]\biggr) -\int_{\R^2 }  2 |u| \nabla^{\perp} |u| \cdot \biggl(\frac{\bJ[u]}{|u|^2+\varepsilon} + \beta  \, \bA[|u|^2]\biggr) \right| \\
& \leq \lim_{\varepsilon \to 0^+} 2 \left(\int_{\R^2} |u|^2 \, \left|\frac{\bJ[u]}{|u|^2+\varepsilon} + \beta  \, \bA[|u|^2]\right|^2   \right)^{\frac{1}{2}} \left(\int_{\R^2} \bigl|\nabla |u|\bigr|^2 \right)^{\frac{1}{2}}\\
&= 2 \left(\int_{\R^2} |u|^2 \, \left|\frac{\bJ[u]}{|u|^2} + \beta  \, \bA[|u|^2]\right|^2  \1_{\{u \neq 0\}}  \right)^{\frac{1}{2}} \left(\int_{\R^2} \bigl|\nabla |u|\bigr|^2 
\right)^{\frac{1}{2}},
\end{align*}
where we used \eqref{eq:epsilonconvergence1} in the first equality and \eqref{eq:epsilonconvergece2} in the last equality. This completes the proof.
\end{proof}

\begin{proposition}[Weak-field bounds]
\label{cor:basicbounds}
For every $\beta \ge 0$, we have 
\begin{align*}
\gamma_{\ast}(0) + \frac{\pi^2 }{\gamma_{\ast}(\beta)} \beta^2 \leq  \gamma_{\ast}(\beta) \leq \gamma_{\ast}(0) + \frac{3 \gamma_{\ast}(0)}{2}  \beta^2.
\end{align*}
\end{proposition}
\begin{remark}
    In the r.h.s.\ the factor $3/2$ stems from the constant $\CH$
    in \eqref{eq:AF-H1} which may have some room for improvement.
\end{remark}
\begin{proof}
Let $\beta > 0$ and $u_n \in C^{\infty}_c(\R^2;\C)$
be a minimizing sequence for $\gamma_{\ast}(\beta)$. By the dilation 
$\|u_n\|_{L^4}^{-2} u_n\left(\|u_n\|_{L^4}^{-2} \bx\right)$ for every $\bx \in \R^2$, we can assume that $\int_{\R^2}|u_n|^4 =1$ for every $n$. Now, by 
the expansion $\nabla u_n = \left(\nabla|u_n| + {\rm i}\frac{\bJ[u_n]}{|u_n|} \right) \frac{u_n}{|u_n|}$ 
on $\R^2 \setminus \{\bx \in \R^2: u_n(\bx)=0\}$,
the diamagnetic inequality \eqref{eq:AF-diamagnetic}, 
and \Cref{prop:magneticbound}, we obtain
\begin{align*}
\gamma_{\ast}(\beta) = \lim_{n \to \infty} \cE_{\beta}[u_n] & \geq \lim_{n \to \infty} \left(\int_{\R^2} \bigl|\nabla |u_n|\bigr|^2 + \int_{\R^2} |u_n|^2 \left|\frac{\bJ[u_n]}{|u_n|^2} + \beta \bA[|u_n|^2]\right|^2 \1_{\{u_n \neq 0\}} \right) 
\\ & \geq  \liminf_{n \to \infty} \left( \gamma_{\ast}(0) + \frac{\pi^2 \beta^2}{\int_{\R^2} |\nabla |u_n|\, |^2} \right)
\\ & \geq  \liminf_{n \to \infty} \left( \gamma_{\ast}(0) + \frac{\pi^2 \beta^2}{\cE_{\beta}[u_n]} \right) = \gamma_{\ast}(0) + \frac{\pi^2 \beta^2}{\gamma_{\ast}(\beta)}.
\end{align*}

Also, by the Hardy-type inequality \eqref{eq:AF-H1}, we derive at
\begin{align*}
\gamma_{\ast}(\beta) \leq \frac{\cE_{\beta}[u_0]}{\int_{\R^2} |u_0|^4} 
= \frac{\int_{\R^2} |\nabla u_0|^2 + \beta^2 \int_{\R^2} \bA[|u_0|^2] |u_0|^2}{\int_{\R^2} |u_0|^4} 
\leq \gamma_{\ast}(0) + \frac{3 \gamma_{\ast}(0)}{2}  \beta^2,
\end{align*}
with the real optimizer $u_0 = \tau/\norm{\tau}_{L^2}$ (Townes profile),
which completes the proof.
\end{proof}

Now, we prove that $\gamma_{\ast}(\beta) > 2\pi \beta$ for $0 \leq \beta <2$ and that there exist minimizers for $\gamma_{\ast}(\beta)$ if $\beta >0$ is small enough. For this purpose, we use the Lipschitz continuity of $\gamma_{\ast}$ and employ the concentration-compactness method of \cite{Lions-84a}.

\begin{proposition}[Lipschitz continuity] 
\label{prop:Lipschitz}
The function $\gamma_{\ast}$ is locally Lipschitz with
$$
\Lip(\gamma_*|_{[0,\beta]}) 
\le \left(\sqrt{\frac{3}{2}}\left(1+ \sqrt{\frac{3}{2}} \beta \right)^{-2}   +3 \beta \right) 
\left(1 + \frac{3}{2} \beta^2 \right) \CLGN,
$$
for all $\beta > 0$.
\end{proposition}
\begin{proof}
Let $\beta, \beta' $ be positive and $u \in H^1(\R^2)$. Then,
\begin{equation}\label{eq:AF-crossterms-Lip}
\cE_{\beta}[u] = \int_{\R^2}  |\nabla u|^2 + 2\beta \int_{\R^2}  \bA\left[|u|^2\right] \cdot \bJ[u] + \beta^2 \int_{\R^2}  \left|\bA\left[|u|^2\right]\right|^2 |u|^2.
\end{equation}
Hence,
$$
\cE_{\beta}[u] - \cE_{\beta'}[u] 
= 2(\beta-\beta') \int_{\R^2} \bA\left[|u|^2\right] \cdot \bJ[u] 
+ (\beta-\beta')(\beta+\beta') \int_{\R^2} \left|\bA\left[|u|^2\right] |u|\right|^2.
$$
By the variational principle,
$$
\gamma_*(\beta) \le \frac{\cE_{\beta}[u]}{\int_{\R^2}  |u|^4}
=  \frac{\cE_{\beta'}[u]}{\int_{\R^2}  |u|^4}
+ 2(\beta-\beta') \frac{\int_{\R^2} \bA\left[|u|^2\right] \cdot \bJ[u]}{\int_{\R^2} |u|^4}
+ (\beta-\beta')(\beta+\beta') \frac{\int_{\R^2}  \left|\bA\left[|u|^2\right]\right|^2\, |u|^2}{\int_{\R^2} |u|^4}.
$$
Hence, by \labelcref{eq:AF-H1,lem:A.Jbound,lem:boundgradient}, we have 
\begin{align*}
\gamma_*(\beta) \leq  \frac{\cE_{\beta'}[u]}{\int_{\R^2}  |u|^4} \left( 1+  \sqrt{\frac{3}{2}}|\beta-\beta'| \left(1+ \sqrt{\frac{3}{2}} \beta' \right)^{-2}  + \frac{3|\beta-\beta'|(\beta + \beta')}{2} \right).
\end{align*}
Then, by taking the infimum over $u \in H^1(\R^2)$ such that $\int_{\R^2} |u|^2=1$, we arrive at
\begin{align*}
\gamma_*(\beta) \leq  \gamma_*(\beta') \left( 1+  \sqrt{\frac{3}{2}}|\beta-\beta'| \left(1+ \sqrt{\frac{3}{2}} \beta' \right)^{-2}  + \frac{3|\beta-\beta'|(\beta + \beta')}{2} \right).
\end{align*}
In conclusion, by \Cref{cor:basicbounds}, we obtain
\begin{align*}
\limsup_{\beta' \to \beta} \frac{  | \gamma_*(\beta) -  \gamma_*(\beta') |}{|\beta-\beta'|} &\leq \left(\sqrt{\frac{3}{2}}\left(1+ \sqrt{\frac{3}{2}} \beta \right)^{-2} +3 \beta \right) \limsup_{\beta' \to \beta} \gamma_{\ast}(\beta') \\ &\leq  \left(\sqrt{\frac{3}{2}}\left(1+ \sqrt{\frac{3}{2}} \beta \right)^{-2} +3 \beta \right)  \left(\gamma_{\ast}(0) + \frac{3 \gamma_{\ast}(0)}{2} \beta^2 \right).
\end{align*}
This completes the proof.
\end{proof}

\begin{lemma}[Concentration-compactness property]
\label{lem:minimizertrick}
Let $\beta > 0$.
If there exists no minimizer for $\gamma_{\ast}(\beta)$, 
then we have 
$$
\frac{\gamma_{\ast}(\beta)}{\beta} = \frac{\gamma_{\ast}(\beta')}{\beta'}
$$
for some $0<\beta'<\beta$.
\end{lemma}
\begin{proof}
The proof of this Lemma is based on the concentration-compactness method of \cite{Lions-84a,Lions-84b} (see also \cite{EstLio-89}). Let $u_n \in C_c^\infty(\R^2)$ be a sequence of minimizers for $\gamma_*(\beta)$, i.e.,
\begin{align}
\label{eq:definitionofsequence}
\gamma_{\ast}(\beta) = \lim_{n \to \infty} \frac{\int_{\R^2} \left|(\nabla + {\rm i}\beta\bA[|u_n|^2]) u_n\right|^2}{\int_{\R^2} |u_n|^4},
\end{align}
with $\int_{\R^2} |u_n|^2 =1$. By dilation 
$t_n^{-1} u_n(t_n^{-1}\bx)$ with $t_n = \|\nabla  u_n\|_{L^2}$, we can assume w.l.o.g that $\int_{\R^2} |\nabla u_n|^2 = 1$.
By the concentration-compactness lemma \cite[Lemma~I.1]{Lions-84a}, there are three possibilities. 

{\bf 1. Tightness.}
The first case is that up to a translation $\bx_n \in \R^2$ and subsequence $n_k \in \N$, re-denoted $\{u_n\}$, we derive
that, for every $\varepsilon >0$, there exists $R_{\epsilon}>0$,  such that 
\begin{align}
\label{eq:firstcase}
    \int_{B(0,R_{\epsilon})} (|u_{n}|^2 +|\nabla u_{n}|^2)  \geq 2 -\varepsilon.
\end{align}
Moreover, by the Rellich--Kondrachov compactness theorem, passing to a subsequence, again $\{u_{n}\}$, converges to some $u \in H^1(\R^2)$ weakly in $H^1(\R^2)$, pointwise almost everywhere in $\R^2$, and strongly in $L^2_{\loc}(\R^2)$. Hence, by \eqref{eq:firstcase} and Fatou's lemma, we have  $\int_{\R^2} |u|^2=1$. Then, by the Brezis--Lieb lemma \cite{BreLie-83} and passing to a subsequence, we obtain that $u_n$ converges to $u$ strongly in $L^2(\R^2)$. By the Gagliardo-Nirenberg interpolation inequality \eqref{eq:gagliardonirenberg}, we obtain that $u_n$ converges to $u$ in $L^p(\R^2)$ for every $2 \leq p <\infty$. Now, we prove the claim that the sequence $\nabla u_n + i\beta\bA[|u_n|^2] u_n$ weakly converges to $\nabla u + i\beta\bA[|u|^2] u$ in $L^2(\R^2)$. Since $u_n$ converges to $u$ in $L^p$ strongly for every $p \geq 2$, it is sufficient to prove that $\bA[|u_n|^2] u_n$ converges to  $\bA[|u|^2] u$ in $L^2(\R^2)$. To see this, we use the Cauchy-Schwarz inequality and Proposition \ref{prop:youngconv} to derive that (for a universal constant $C>0$)
\begin{equation}
\label{eq:L2differenceofgauge}
\begin{aligned}
   & \int_{\R^2} \left| \bA[|u_n|^2] u_n - \bA[|u|^2] u    \right| ^2 \leq 2\int_{\R^2} \left| \bA[|u_n|^2-|u|^2] u_n\right|^2 + 2\left|\bA[|u|^2] (u-u_n) \right|^2 
    \\ &\leq 2\left\| \bA[|u_n|^2-|u|^2]\right \|^2_{L^4} \|u_n\|^2_{L^4} + 2\left\|\bA[|u|^2] \right\|^2_{L^4} \left \| u-u_n  \right \|^2_{L^4}
    \\&\leq 2 C \left\| |u_n|^2-|u|^2\right \|^2_{L^{\frac{4}{3}}} \|u_n\|^2_{L^4} + 2C \left\| u \right\|^4_{L^{\frac{8}{3}}} \left \|u-u_n  \right \|^2_{L^4}
    \\ &\leq 8 C \left\|u_n-u \right \|^2_{L^{\frac{8}{3}}} \left(\left\|u_n \right \|^2_{L^{\frac{8}{3}}}+ \left\|u \right \|^2_{L^{\frac{8}{3}}}\right)  \|u_n\|^2_{L^4} + 2C\left\| u \right\|^4_{L^{\frac{8}{3}}} \left \|u-u_n  \right \|^2_{L^4}.
\end{aligned}
\end{equation}
Letting $n \to \infty$ in the above estimate concludes the claim. In conclusion, by the weak convergence of $\nabla u_n + i\beta\bA[|u_n|^2] u_n$ to $\nabla u + i\beta\bA[|u|^2] u$ in $L^2(\R^2)$ and the strong convergence of $u_n$ to $u$ in $L^4(\R^2)$, we arrive at
\begin{align*}
\gamma_{\ast}(\beta) = \lim_{n \to \infty} \frac{\cE_{\beta}[u_n] }{\int_{\R^2} |u_n|^4} \geq \frac{\cE_{\beta}[u]}{\int_{\R^2} |u|^4} \geq \gamma_{\ast}(\beta),
\end{align*}
which proves that $u$ is a minimizer for $\gamma_{\ast}(\beta)$, thus contradicting our assumptions.

{\bf 2. Vanishing.}
The second case is that, up to a subsequence of $\{u_n\}$, we have 
\begin{align}
\label{eq:badcase1}
\lim_{n \to \infty} \sup_{\by \in \R^2} \int_{B(\by,R)} (|u_n|^2 + |\nabla u_n|^2) = 0
\end{align}
for all $R < \infty$. Then, by \cite[Lemma I.1]{Lions-84b} and $\int_{\R^2} |\nabla u_n|^2=1$, we have that $\|u_n\|_{L^4}$ converges to zero.  Then, by \eqref{lem:boundgradient} and \eqref{eq:definitionofsequence}, we have
\begin{align*}
 1= \lim_{n \to \infty} \int_{\R^2} |\nabla u_n|^2  \leq \left( 1+ \sqrt{\frac{3}{2}} \beta \right)^2 \lim_{n \to \infty} \cE_{\beta}[u_n] = 0, 
\end{align*}
which is a contradiction.
 
{\bf 3. Dichotomy.}
The final case, in which there exists no minimizer for $\gamma_*(\beta)$, 
is that, up to a subsequence and translation of $\{u_{n}\}$, there exists
$0<\alpha <2$ such that for all $0<\varepsilon<1$ there exist sequences $\{u^{\varepsilon}_{n,1}\}, \{u^{\varepsilon}_{n,2}\} \subset H^1(\R^2)$, 
depending on $\varepsilon,$ such that 
\begin{equation}
\label{eq:dichotomycase}
    \begin{aligned}
      &| \|u^{\varepsilon}_{n,1}\|_{H^1(\R^2)}^2 - \alpha | \leq C \varepsilon, \quad  | \|u^{\varepsilon}_{n,2}\|_{H^1(\R^2)}^2 - (2-\alpha) | \leq C \varepsilon,\\
      & \|u^{\varepsilon}_{n,1}+u^{\varepsilon}_{n,2} - u_n \|_{H^1(\R^2)}^2 \leq C \varepsilon,
    \end{aligned}
\end{equation}
where $C>0$ is independent of $n,\varepsilon,$ and
\begin{align}
\label{eq:distanceofsupport}
0 < d^{\varepsilon}_n :=\dist(\supp(u^{\varepsilon}_{n,1}), \supp(u^{\varepsilon}_{n,2})) \to \infty.   
\end{align}
Then, by \eqref{eq:L2differenceofgauge} and the Gagliardo-Nirenberg interpolation inequality \eqref{eq:gagliardonirenberg}, we have 
\begin{equation}
\label{eq:energydifference}
\begin{aligned}
   \left| \cE_{\beta}[u_n] - \cE_{\beta}[u^{\varepsilon}_{n,1} + u^{\varepsilon}_{n,2}] \right| \leq \Tilde{C} \varepsilon,\\
   \left|\int_{\R^2} |u_n|^4- \left(\int_{\R^2} |u^{\varepsilon}_{n,1}|^4 + \int_{\R^2} |u^{\varepsilon}_{n,2}|^4 \right)\right| \leq \Tilde{C} \varepsilon,
\end{aligned}
\end{equation}
where $\Tilde{C}>0$ is independent of $n,\varepsilon.$ Furthermore, we note that, by \eqref{eq:dichotomycase}, 
\begin{align}
\label{eq:gaugesmallbound}
\left|\bA[|u_{n,1}^{\varepsilon}|^2](\bx)\right| \leq \int_{\R^2}\frac{|u^{\varepsilon}_{n,1}(\by)|^2}{|\bx-\by|} \dd \by \leq \frac{3+ C \varepsilon}{d^{\varepsilon}_n} \xrightarrow{n\to\infty} 0
\end{align}
for $\bx \in \supp(u^{\varepsilon}_{n,2})$ and likewise $\left|\bA[|u^{\varepsilon}_{n,2}|^2](\bx)\right| \leq \frac{3+C\varepsilon}{d^{\varepsilon}_n}$ for every $\bx \in \supp(u^{\varepsilon}_{n,1})$. Hence, by expanding the magnetic self-energy and by the Cauchy-Schwarz inequality, we derive
\begin{align*}
  & \left| \cE_{\beta}[u^{\varepsilon}_{n,1} + u^{\varepsilon}_{n,2}] -(\cE_{\beta}[u^{\varepsilon}_{n,1}] + \cE_{\beta}[u^{\varepsilon}_{n,2}])\right| \leq
 2\beta (\cE_{\beta}[u^{\varepsilon}_{n,1}])^{\frac{1}{2}} \left ( \int_{\R^2} \left |\bA[|u^{\varepsilon}_{n,1}|^2]  u^{\varepsilon}_{n,2} \right |^2 \right)^{\frac{1}{2}} \\&+2\beta (\cE_{\beta}[u^{\varepsilon}_{n,2}])^{\frac{1}{2}} \left ( \int_{\R^2} \left |\bA[|u^{\varepsilon}_{n,2}|^2]  u^{\varepsilon}_{n,1} \right |^2 \right)^{\frac{1}{2}} \!\!
  + \beta^2 \int_{\R^2} \left |\bA[|u^{\varepsilon}_{n,1}|^2]  u^{\varepsilon}_{n,2} \right |^2 
  + \beta^2 \int_{\R^2} \left |\bA[|u^{\varepsilon}_{n,2}|^2]  u^{\varepsilon}_{n,1} \right |^2
  \\ &\leq 2\beta \frac{(3+C \varepsilon)}{d^{\varepsilon}_n} \left( (\cE_{\beta}[u^{\varepsilon}_{n,2}])^{\frac{1}{2}} +(\cE_{\beta}[u^{\varepsilon}_{n,2}])^{\frac{1}{2}} \right) + 2\beta^2 \frac{(3+C \varepsilon)^3}{(d^{\varepsilon}_n)^2} \leq \Tilde{C} \varepsilon,
\end{align*}
for large enough $n$. In conclusion, by \eqref{eq:energydifference}, we obtain 
\begin{align}
\label{eq:energydifference2}
    \left| \cE_{\beta}[u_n] -(\cE_{\beta}[u^{\varepsilon}_{n,1}] + \cE_{\beta}[u^{\varepsilon}_{n,2}])\right| \leq 2 \Tilde{C} \varepsilon,
\end{align}
for large enough $n.$

Now, consider a sequence of $n_k \in \N, \varepsilon_{n_k} \in \R$, such that 
$\lim_{n_k \to \infty} \varepsilon_{n_k} =0$, 
\begin{equation*}
\lim_{n_k \to \infty} \int_{\R^2} |u^{\varepsilon_{n_k}}_{n_k,1}|^2 = \theta,
\qquad \text{and} \qquad 
\lim_{n_k \to \infty} \int_{\R^2} |u^{\varepsilon_{n_k}}_{n_k,2}|^2 = 1-\theta, 
\end{equation*}
where $ 0\leq \theta \leq 1.$
If $\theta =0$, then by \eqref{eq:dichotomycase} and the Gagliardo-Nirenberg interpolation inequality \eqref{eq:gagliardonirenberg}, we derive that $\lim_{n_k \to \infty}\int_{\R^2} |u^{\varepsilon_{n_k}}_{n_k,1}|^4 =0$, and, by \eqref{lem:boundgradient}, \eqref{eq:energydifference}, and \eqref{eq:energydifference2},
 \begin{align*}
    \gamma_{\ast}(\beta) &= \lim_{n_k \to \infty} \frac{\cE_{\beta}[u_{n_k}]}{\int_{\R^2} |u_{n_k}|^4} 
    \geq \liminf_{n_k \to \infty}\frac{\cE_{\beta}[u^{\varepsilon_{n_k}}_{n_k,1}] + \cE_{\beta}[u^{\varepsilon_{n_k}}_{n_k,2}] - 2 \Tilde{C} \varepsilon_{n_k}}{\int_{\R^2} |u_{n_k}|^4} \\
    &\geq 
    \liminf_{n_k \to \infty}\frac{\cE_{\beta}[u^{\varepsilon_{n_k}}_{n_k,1}] - 2 \Tilde{C} \varepsilon_{n_k}}{\int_{\R^2} |u_{n_k}|^4}
    + \liminf_{n_k \to \infty} \frac{\cE_{\beta}[u^{\varepsilon_{n_k}}_{n_k,2}]}{\int_{\R^2} |u^{\varepsilon_{n_k}}_{n_k,1}|^4 +\int_{\R^2} |u^{\varepsilon_{n_k}}_{n_k,2}|^4 + \Tilde{C} \varepsilon_{n_k}} \\
    &\geq 
    \alpha \CLGN
    + \liminf_{n_k \to \infty} \|u^{\varepsilon_{n_k}}_{n_k,2}\|_{L^2}^{-2} \gamma_{\ast}\left(\|u^{\varepsilon_{n_k}}_{n_k,2}\|_{L^2}^{2} \beta\right) \\
    &= \alpha \CLGN + \gamma_{\ast}(\beta),
\end{align*}
which is a contradiction. 
Here, we used, for large enough $n_k$,
by \eqref{lem:boundgradient}, \eqref{eq:definitionofsequence}, and \Cref{prop:LGN-bound},
\begin{align*}
    \left( 2\left(1+\sqrt{\frac{3}{2}}\beta\right)^2 \gamma_*(\beta) \right)^{-1} 
    &\le \int_{\R^2} |u_{n_k}|^4 \le \CLGN^{-1},
    \\
    \liminf_{n_k \to \infty} \cE_\beta[u^{\varepsilon_{n_k}}_{n_k,1}] 
    &\ge \liminf_{n_k \to \infty} \int_{\R^2} |\nabla u^{\varepsilon_{n_k}}_{n_k,1}|^2 = \alpha,
    \\
    \frac{ \cE_\beta[u^{\varepsilon_{n_k}}_{n_k,2}] }{ \int_{\R^2} |u^{\varepsilon_{n_k}}_{n_k,2}|^4 }
    &\ge \norm{u^{\varepsilon_{n_k}}_{n_k,2}}_{L^2}^{-2} \gamma_*\left(\|u^{\varepsilon_{n_k}}_{n_k,2}\|_{L^2}^{2} \beta\right),
\end{align*}
as well as \Cref{prop:Lipschitz} in the last limit. Similarly, it cannot occur that $\theta=1.$ Hence, $0<\theta <1$ and 
\begin{align}
\label{eq:lowerboundgammastar}
\gamma_{\ast}(\beta) = \lim_{n_k \to \infty} \frac{\cE_{\beta}[u^{\varepsilon_{n_k}}_{n_k,1}] + \cE_{\beta}[u^{\varepsilon_{n_k}}_{n_k,2}]}{\int_{\R^2} |u^{\varepsilon_{n_k}}_{n_k,1}|^4 + |u^{\varepsilon_{n_k}}_{n_k,2}|^4} \geq \min\left\{\liminf_{n_k \to \infty} 
\frac{\cE_{\beta}[u^{\varepsilon_{n_k}}_{n_k,1}]}{\int_{\R^2} |u^{\varepsilon_{n_k}}_{n_k,1}|^4},  \liminf_{n_k \to \infty} \frac{\cE_{\beta}[u^{\varepsilon_{n_k}}_{n_k,2}]}{\int_{\R^2} |u^{\varepsilon_{n_k}}_{n_k,2}|^4}\right\},
\end{align}
where we have used the fact that
$$
\frac{a+b}{c+d} \ge \min\left\{\frac{a}{c},\frac{b}{d}\right\},\quad \forall a,b,c,d>0.
$$
Then, w.l.o.g, by \eqref{eq:lowerboundgammastar}, 
 we may assume that 
\begin{align*}
\liminf_{n_k \to \infty}   \frac{\cE_{\beta}[u_{n_k,1}^{\varepsilon_{n_k}}]}{\int_{\R^2} |u_{n_k,1}^{\varepsilon_{n_k}}|^4} \leq  \gamma_{\ast}(\beta).
\end{align*}
Define $\Tilde{u}_{n_k,1}^{\varepsilon_{n_k}}:= \frac{u^{\varepsilon_{n_k}}_{n_k,1}}{\sqrt{\theta_{n_k}}}$, where $\theta_{n_k} := \int_{\R^2} |u^{\varepsilon_{n_k}}_{n_k,1}|^2 \to \theta$ as $n_k \to \infty$. Then, by \Cref{prop:Lipschitz}, we have
\begin{align*}
\gamma_{\ast}(\theta \beta) = \lim_{n_k \to \infty} \gamma_{\ast}\left( \theta_{n_k} \beta\right) \leq  \liminf_{n_k \to \infty}   \frac{\cE_{\theta_{n_k} \beta}[\Tilde{u}_{n_k,1}^{\varepsilon_{n_k}}]}{\int_{\R^2} |\Tilde{u}_{n_k,1}^{\varepsilon_{n_k}}|^4}  = \theta \liminf_{n_k \to \infty} \frac{\cE_{ \beta}[u_{n_k,1}^{\varepsilon_{n_k}}]}{\int_{\R^2} |u_{n_k,1}^{\varepsilon_{n_k}}|^4} 
\leq \theta \gamma_{\ast}(\beta). 
\end{align*}
In fact, by \Cref{lem:keylemma}, the equality must occur in the above. This completes the proof.
\end{proof}

\begin{proposition}[Breaking of supersymmetry]\label{prop:susy-breaking}
If $0 \le \beta < 2$, then $\gamma_{\ast}(\beta) > 2 \pi \beta$. 
\end{proposition}
\begin{proof}
If $0 < \beta < 2$ and $\gamma_{\ast}(\beta) = 2 \pi \beta$, then, by \Cref{lem:keylemma} and \Cref{cor:minimumofenergy}, we have $$\gamma_{\ast}(\alpha)= 2\pi \alpha,$$ for every $\beta \leq \alpha \leq 2$. Since $\gamma_{\ast}$ is a Lipschitz function, by \Cref{prop:Lipschitz}, and $\gamma_{\ast}(0) > 0$, we can take the smallest $0<\beta_{\ast} <2$, such that $\gamma_{\ast}(\beta_{\ast}) = 2\pi \beta_{\ast}$. Since, by \Cref{cor:nececcityofminimizers}, there exists no minimizer for such $\gamma_{\ast}(\beta_{\ast})$, by \Cref{lem:minimizertrick}, we obtain that $\gamma_{\ast}(\alpha \beta_{\ast}) = \alpha \gamma_{\ast}(\beta_{\ast})= 2 \pi \alpha \beta_{\ast}$ for some $0<\alpha <1$. This is in contradiction with the choice of $\beta_{\ast},$ since $\alpha \beta_{\ast} <\beta_{\ast}.$ 
\end{proof}

\begin{proposition}
For every $0 \leq \beta \le \min\left\{\frac{\CLGN}{ \Lip(\gamma_{\ast}|_{[0,\beta]})},2\right\}$, we have a minimizer for $\gamma_{\ast}(\beta)$.
\end{proposition}
\begin{proof}
Assume $0<\beta <2$ is such that there exists no minimizer for $\gamma_*(\beta)$. Then, by \Cref{lem:minimizertrick,lem:keylemma}, we have for $0<\alpha_{\beta}<1$, possibly depending on $\beta$, $ \frac{\gamma_{\ast}(\beta)}{\beta} = \frac{\gamma_{\ast}( \theta)}{ \theta}$ for every $\alpha_{\beta} \beta \leq \theta \leq \beta.$ Hence, by Lipschitz continuity of $\gamma_{\ast}$, see \Cref{prop:Lipschitz}, we have
\begin{align*}
0 =\frac{\dd }{\dd  \theta} \frac{\gamma_{\ast}( \theta)}{ \theta} = \frac{\gamma'_{\ast}( \theta)\theta - \gamma_{\ast}( \theta) }{\theta^2},
\end{align*}
for a.e. $\alpha_{\beta} \beta < \theta < \beta$. In conclusion, by \Cref{cor:basicbounds}, we derive  
\begin{align*}
\Lip(\gamma_{\ast}|_{[0,\beta]}) \geq \gamma'_{\ast}(\theta) = \frac{\gamma_*(\theta)}{\theta} > \frac{\gamma_*(0)}{\theta},
\end{align*}
for a.e. $\alpha_{\beta} \beta < \theta < \beta$. This concludes the proof since the above requires
$\beta > \frac{\CLGN}{ \Lip(\gamma_{\ast}|_{[0,\beta]})}.$
\end{proof}

\begin{remark}
    Using the rough upper bound
    $\Lip(\gamma_{\ast}|_{[0,\beta]}) \le (\sqrt{3/2}+6)\cdot 7\CLGN$ from \Cref{prop:Lipschitz}, we obtain that minimizers exist for any
    $0 \le \beta \le ((\sqrt{3/2}+6)\cdot 7)^{-1}$.
\end{remark}

This completes the proof of parts \ref{itm:mstab-gamma}-\ref{itm:mstab-mini} of \Cref{thm:magneticstability}.

\subsection{Symmetries of minimizers}\label{sec:mag-symm}

In this subsection, we prove the symmetries of our representation of minimizers for $\gamma_{\ast}(\beta)$ for every $\beta \in 2 \N$, thus completing the proof of \Cref{thm:magneticstability}.

\begin{proposition}
\label{prop:mag-symmetries}
Let $\beta \in 2 \N$ and $(P,Q), (\Tilde{P},\Tilde{Q}) \in \cP^\times \times \cP^\times$, which satisfy \begin{align*}
\gcd(P,Q)&=\gcd(\Tilde{P},\Tilde{Q})=1, \\
\max(\deg(P),\deg(Q))&=\max(\deg(\Tilde{P}),\deg(\Tilde{Q}))= \frac{\beta}{2}.
\end{align*}
Then, $u_{P,Q} = u_{\Tilde{P},\Tilde{Q}}$ if and only if there exists $\Lambda \in \R^+ \times \sSU(2)$ such that 
$(\Tilde{P},\Tilde{Q})= \Lambda (P,Q).$ 
\end{proposition}
\begin{proof}
If $\Lambda \in \R^+ \times \sSU(2)$ such that 
$\Lambda (P,Q) = (\Tilde{P},\Tilde{Q}),$ then $\Lambda = C \Lambda_0,$ where $C>0, \Lambda_0 \in \sSU(2)$, $W(\Tilde{P},\Tilde{Q})=\det(\Lambda)W(P,Q),$ and 
\begin{align*}
u_{\Tilde{P},\Tilde{Q}} =  \sqrt{\frac{2}{\pi \beta}} \, \frac{\overline{\Tilde{P}'\Tilde{Q} - \Tilde{P}\Tilde{Q}'}}{|\Tilde{P}|^2 + |\Tilde{Q}|^2} = \sqrt{\frac{2}{\pi \beta}} \, \frac{C^2 \overline{\det(\Lambda_0)} (\overline{P'Q - P Q'})}{C^2(|P|^2 + |Q|^2)}= u_{P,Q}.
\end{align*}

Now, if $u_{\Tilde{P},\Tilde{Q}}  = u_{P,Q}$ for pairs of coprime polynomials, then we have 
\begin{align}
\label{eq:equalityofminimizers}
\frac{\overline{P'Q - PQ'}}{|P|^2 + |Q|^2} 
= 
\frac{\overline{\Tilde{P}'\Tilde{Q} - \Tilde{P}\Tilde{Q}'}}{|\Tilde{P}|^2 + |\Tilde{Q}|^2},
\end{align}
where the denominators are uniformly bounded from below.
Hence, by the fundamental theorem of algebra, we have 
\begin{align}
\label{eq:Wronskianconnection}
P'Q - PQ' = C (\Tilde{P}'\Tilde{Q} - \Tilde{P}\Tilde{Q}'),
\end{align}
for a constant $C \in \C$. We deduce from \labelcref{eq:equalityofminimizers,eq:Wronskianconnection} that
\begin{align*}
|\Tilde{P}|^2 + |\Tilde{Q}|^2 =\overline{C} (|P|^2 + |Q|^2).
\end{align*}
Therefore, $C>0$ and 
\begin{align*}
|\Tilde{P}|^2 + |\Tilde{Q}|^2 = |\sqrt{C} P|^2 + |\sqrt{C} Q|^2.
\end{align*}
By \Cref{thm:liouville}, we derive that \begin{align*}
(\Tilde{P},\Tilde{Q}) = \Lambda_0 (\sqrt{C} P,\sqrt{C} Q),
\end{align*}
for some $\Lambda_0 \in \sU(2).$ 
Using the Wronskian property, \Cref{prop:actiononwronksian}, we obtain
\begin{align*}
W (\Tilde{P},\Tilde{Q}) = C \det(\Lambda_0) W(P,Q).
\end{align*}
In conclusion, by \eqref{eq:Wronskianconnection}, we derive $\det(\Lambda_0)=1$. Finally, letting $\Lambda = \sqrt{C}  \Lambda_0$ completes the proof.
\end{proof}

\begin{remark}
\label{rmk:degreesless-u}
    By \Cref{prop:mag-symmetries}, we also have $u_{P,Q} = u_{\tilde{P},\tilde{Q}}$ where $\deg \tilde{P} > \deg \tilde{Q}$ and $\tilde{P}$ is monic (cp.~\Cref{rmk:degreesless}).
    Namely, if $P = az^n + O(z^{n-1})$, $Q = bz^n + O(z^{n-1})$, and $\Lambda := \frac{1}{|a|^2+|b|^2} \begin{pmatrix} \overline{a} & \overline{b} \\ -b & a \end{pmatrix}$,
    then $(\tilde{P},\tilde{Q}) := \Lambda(P,Q) = (z^n + O(z^{n-1}), O(z^{n-1}))$.
\end{remark}

\subsection{Vortex rings and minimizers with radially symmetric density}\label{sec:symmetric}

In this subsection, 
we derive the vortex rings \eqref{eq:vortexring} and the minimizers for $\gamma_{\ast}(\beta)$ with radially symmetric density. 

\begin{corollary}[Single-root solutions / vortex rings]
Let $\beta = 2n \in 2\N$. Then, all the minimizers of $\gamma_{\ast}(\beta)$ which only have the root $z_0 \in \C$ are of the form  
\begin{align}\label{eq:vortexringform-symm}
u_n(z) = \sqrt{\frac{n}{\pi}} \frac{b\, \overline{a \, \,(z-z_0)^{n-1}}}{ |a(z-z_0)^{n} + c|^2 + b^2},
\quad n \in \N, a \in \C^{\times}, b \in \R^+, c \in \C,
\end{align}
for every $z \in \C.$
Moreover, if $|u_n|$ is radially symmetric, then
\begin{align}\label{eq:radialvortexring}
u_n(z) = \overline{C} \sqrt{\frac{n}{\pi}} \frac{ {\bar{z}}^{n-1}}{|z|^{2n} + |C|^2},
\end{align}
for $ C \in \C^\times.$
\end{corollary}
\begin{remark}\label{rmk:magneticsolsymm}
    The latter, radial case was already known from \cite{JacPi-90b}. 
    The case $c \neq 0$ appears to only have been discussed in the context of the corresponding Liouville equation; compare \Cref{rmk:Liouvillesolsymm}.
\end{remark}
\begin{proof}
Let $u_n$ be a minimizer for $\gamma_{\ast}(\beta)$ where $\beta = 2n \in  2 \N$. Then, by \Cref{thm:magneticstability}, we derive 
\begin{align}
\label{eq:initialformvortexring}
u_n =   \sqrt{\frac{1}{ \pi n}} \, \frac{\overline{P'Q - PQ'}}{|P|^2 + |Q|^2},
\end{align}
where $P,Q \in \cP$ and 
\begin{align}
\label{eq:degreecondPQ}
\max(\deg(P),\deg(Q)) = n.    
\end{align}
Now, if $z_0 \in \C$ is the only root of $u_n$, then $P'Q - PQ' =\lambda (z-z_0)^k$ for some $\lambda \in \C^\times, k \in \N$. By \Cref{lem:vortexwronskian}, we have 
$$
P= \alpha_1 (z-z_0)^{k+1} + \beta_1
\quad \text{and} \quad Q= \alpha_2 (z-z_0)^{k+1} + \beta_2,
$$ 
where 
\begin{align*}
\alpha_1 \beta_2 - \alpha_2 \beta_1 = \frac{\lambda}{k+1}.
\end{align*}
Hence, by \eqref{eq:degreecondPQ}, we have
$$
P= \alpha_1 (z-z_0)^{n} + \beta_1 \quad \text{and} \quad Q=  \alpha_2 (z-z_0)^{n} + \beta_2,
$$
where 
\begin{align*}
\alpha_1 \beta_2 - \alpha_2 \beta_1 = \frac{\lambda}{n}.
\end{align*}
In conclusion, by \eqref{eq:initialformvortexring}, we derive 
\begin{align}\label{eq:vortexringform}
u_n(z) = \sqrt{\frac{n}{\pi}} \frac{\overline{(\alpha_1 \beta_2 -\alpha_2 \beta_1) (z-z_0)^{n-1}}}{ |\alpha_1 (z-z_0)^{n} + \beta_1|^2 + |\alpha_2 (z-z_0)^{n} + \beta_2|^2},
\quad n \in \N.
\end{align}
Now (cp.~\Cref{rmk:degreesless-u}), we define the matrix
\begin{align*}
    \Lambda := \frac{1}{\sqrt{|\alpha_1| ^2 + |\alpha_2|^2}} \begin{pmatrix}
\overline{\alpha_1} & \overline{\alpha_2} \\
-\alpha_2 & \alpha_1
\end{pmatrix}.
\end{align*}
Then, $\Lambda \in \sSU(2)$ and, by the symmetry of solutions in Theorem \ref{thm:magneticstability}, we have 
\begin{align}
    u_n =  \sqrt{\frac{1}{ \pi n}} \, \frac{\overline{W(P_0,Q_0)}}{|P_0|^2 + |Q_0|^2},
\end{align}
where $(P_0,Q_0)= \Lambda (P,Q).$ Hence, by Proposition \ref{prop:actiononwronksian}, we have
\begin{align*}
 u_n(z) = \sqrt{\frac{n}{\pi}} \frac{\overline{a b \, (z-z_0)^{n-1}}}{ |a(z-z_0)^{n} + c|^2 + |b|^2},
\quad n \in \N,
\end{align*}
for every $z \in \C,$ where 
$a = |\alpha_1|^2 + |\alpha_2|^2 > 0,$ 
$b = \lambda/n \in \C^\times$, 
$c = \overline{\alpha_1}\beta_1 + \overline{\alpha_2}\beta_2 \in \C$.
Moreover, if $|u_n|$ is radially symmetric, we derive that $c=0$
(unless $n=1$).
Hence, by taking $
C := \frac{b}{\Bar{a}},$
we obtain \eqref{eq:radialvortexring}. 
\end{proof}

The following tells us something about how close the vortex rings are to being minimizers for arbitrary $\beta$ 
(they get closer like $O(1/\beta)$ uniformly as $\beta \to \infty$).

\begin{proposition}[Vortex ring energy]
\label{prop:vortexringenergy}
The radial vortex rings $u_n$ in \eqref{eq:radialvortexring} satisfy the identities 
$\int_{\R^2} |u_n|^2 = 1$ 
and
\begin{equation}\label{vortexringenergy}
\frac{\cE_{\beta,2\pi\beta}[u_n]}{\int_{\R^2} |u_n|^4} = \frac{\pi(2n-1)}{n(n+1)}(\beta-2n)^2,\quad \forall n \in \N, \ \beta \ge 0.
\end{equation}
In particular,
$$
\gamma_*(\beta) \leq 2\pi + \frac{\pi}{2}\beta^2,\quad \forall \beta \ge 0.
$$
\end{proposition}
\begin{proof}
We first verify straightforwardly
\begin{equation}\label{eq:vring-L2}
\int_{\R^2} |u_n|^2
= \frac{n}{\pi} \int_0^{2\pi}\int_0^\infty \frac{|C|^2 r^{2n-1}}{(r^{2n}+|C|^2)^2} \dd r\dd \theta = \int_0^\infty \frac{|C|^2}{(t+|C|^2)^2} \dd t = 1.
\end{equation}
Then we compute (see e.g, \cite[p. 325]{GraRyz-14})
\begin{equation}\label{eq:vring-L4}
\int_{\R^2} |u_n|^4 
= \frac{n^2}{\pi^2} \int_0^{2\pi}\int_0^\infty \frac{|C|^4 r^{4n-3}}{(r^{2n}+|C|^2)^4} \dd r\dd \theta = \frac{n}{6\pi|C|^{\frac{2}{n}}} \Gamma\left(2+\frac{1}{n}\right) \Gamma\left(2-\frac{1}{n}\right).
\end{equation}
And we also verify that $u_n$ are smooth and in $H^1(\R^2)$.

Now we claim that, due to the radial symmetry,
the convolution amounts to
$$
\Phi_n := \Phi[|u_n|^2] = - \frac{1}{2n} \log \sqrt{\frac{n}{\pi}}\frac{|C|}{|z|^{2n}+|C|^2},
$$
which is locally bounded and satisfies $\nabla^\perp \Phi_n = \bA[|u_n|^2]$. Next, we insert the concrete expressions $\phi=\beta\Phi_n$, $u=u_n$ in the 
r.h.s.\ of the supersymmetric identity 
\eqref{eq:mag-super},
$$
\cL_\phi[u] := 4 \int_{\R^2} |\partial_{z}(e^\phi u)|^2 e^{-2\phi}.
$$
We use the polar coordinates
$$
\partial_z = \frac{e^{{\rm i}\theta}}{2}\left( 
\partial_r + \frac{{\rm i}}{r}\partial_\theta
\right)
$$
and observe
$$
e^{\beta\Phi_n}u_n
= e^{(\beta-2n)\Phi_n} \bar{z}^{n-1},
$$
where $\Phi_n$ is radial.
Thus, for any $\beta \ge 0$,
by \Cref{lem:mag-super},
\begin{align*}
\cE_{\beta,2\pi\beta}[u_n] = \cL_{\beta\Phi_n}[u_n] & = 4 \int_{\R^2} |\partial_{z}(e^{\beta\Phi_n} u_n)|^2 e^{-2\beta\Phi_n} \\
& = \int_{\R^2} \left|\partial_r e^{(\beta-2n)\Phi_n}\right|^2 |z|^{2n-2} e^{-2\beta\Phi_n} \\
& = (\beta-2n)^2 \int_{\R^2} |\partial_r\Phi_n|^2 |z|^{2n-2} e^{-4n\Phi_n} \\
& = (\beta-2n)^2 \frac{n}{\pi} \int_0^{2\pi}\int_0^\infty \frac{|C|^2 r^{6n-3}}{(r^{2n} + |C|^2)^4} \dd r\dd \theta \\
& = \frac{(\beta-2n)^2}{6|C|^{\frac{2}{n}}} \Gamma\left(1+\frac{1}{n}\right) \Gamma\left(3-\frac{1}{n}\right).
\end{align*}
 Here we again used the computation in \cite[page 325]{GraRyz-14}. Therefore, by \eqref{eq:vring-L4},
$$
\frac{\cL_{\beta\Phi_n}[u_n]}{\int_{\R^2} |u_n|^4}
= \frac{\pi}{n}(\beta-2n)^2\frac{\Gamma\left(1+\frac{1}{n}\right) \Gamma\left(3-\frac{1}{n}\right)}{\Gamma\left(2+\frac{1}{n}\right) \Gamma\left(2-\frac{1}{n}\right)}.
$$
We use the identity $\Gamma(1+z) = z\Gamma(z)$ to simplify the last quantity in the above as follows
$$
\frac{\Gamma\left(1+\frac{1}{n}\right) \Gamma\left(3-\frac{1}{n}\right)}{\Gamma\left(2+\frac{1}{n}\right) \Gamma\left(2-\frac{1}{n}\right)} = \frac{\Gamma\left(1+\frac{1}{n}\right) \left(2-\frac{1}{n}\right) \Gamma\left(2-\frac{1}{n}\right)}{\Gamma\left(1+\frac{1}{n}\right) \left(1+\frac{1}{n}\right) \Gamma\left(2-\frac{1}{n}\right)} = \frac{2n-1}{n+1}.
$$
Putting it all together we obtain the desired identity \eqref{vortexringenergy}.

Finally, by the variational principle, we have
$$
\gamma_*(\beta) \leq \frac{\cE_{\beta,0}[u_1]}{\int_{\R^2} |u_1|^4} = \frac{\cE_{\beta,2\pi\beta}[u_1]}{\int_{\R^2} |u_1|^4} +2\pi\beta = \frac{\pi}{2}(\beta - 2)^2 + 2\pi\beta = 2\pi + \frac{\pi}{2}\beta^2
$$
for every $\beta \ge 0$.
\end{proof}

\bigskip


%
%
\def\MR#1{} 




\newcommand{\etalchar}[1]{$^{#1}$}


\end{document}